\documentclass[journal]{IEEEtran}
\usepackage[utf8]{inputenc}
\usepackage{mathtools}
\mathtoolsset{showonlyrefs,showmanualtags}
\usepackage{amssymb,amsthm}
\usepackage{dsfont} 
\usepackage{booktabs}
\usepackage{color}
\usepackage{bbold} 
\usepackage[final]{graphicx}
\usepackage[yyyymmdd,hhmmss]{datetime}
\usepackage{bm,upgreek}
\usepackage{tablefootnote}
\usepackage{nth}
\usepackage{cite}
\newtheorem{Assumption}{Assumption}
\newtheorem{Theorem}{Theorem}
\newtheorem{Definition}{Definition}
\newtheorem{Remark}{Remark}
\newtheorem{Counterexample}{Counterexample}
\newtheorem{Lemma}{Lemma}
\newtheorem{Corollary}{Corollary}
\newtheorem{Proposition}{Proposition}
\newtheorem{Problem}{Problem}

\DeclareMathOperator{\VEC}{vec}

\DeclareMathOperator{\DIAG}{diag}
\DeclareMathOperator{\TR}{Tr}
\DeclareMathOperator{\VAR}{AC}
\newcommand{\GMFS}{\mathcal{G}_{\text{DSS}}}
\newcommand{\GNS}{\mathcal{G}_{\text{NS}}}
\newcommand{\GC}{\mathcal{G}_{\text{PS}}}
\newcommand{\Exp}[1]{\mathbb{E}[ #1]}

\newcommand{\Zero}{\mathbf{0}_{d_x \times d_x}}

\newcommand{\barF}{{{}\bar F_t(\alpha)}}
\newcommand{\DeltaF}{{{}F_t(\alpha)}}
\newcommand{\barK}{{{}\bar K_t(\alpha)}}
\newcommand{\DeltaK}{{{}K_t(\alpha)}}
\newcommand{\barM}{{{}\mathbf P^\alpha_t}}
\newcommand{\DeltaM}{{{}P^\alpha_t}}

\newcommand\edit{\textcolor{black}}

\newcommand{\Compress}{\medmuskip=0mu
\thinmuskip=0mu
\thickmuskip=0mu}

 \usepackage{tikz}
\usetikzlibrary{graphs,graphs.standard}

\tikzset{%
  every neuron/.style={
    circle,
    draw,
    minimum size=.7cm
  },
  neuron missing/.style={
    draw=none, 
    scale=3,
    text height=0.333cm,
    execute at begin node=\color{black}$\vdots$
  },
}

\author{Jalal~Arabneydi, Amir G. Aghdam, and Roland P. Malham\'e, 
\thanks{Jalal Arabneydi and Amir G. Aghdam are with the  Department of Electrical and Computer Engineering, 
        Concordia University, 1455 de Maisonneuve Blvd, Canada. 
        {\tt\small Email:jalal.arabneydi@mail.mcgill.ca} and
        {\tt\small Email:aghdam@ece.concordia.ca}}
\thanks{Roland Malham\'e is with the  Department of Electrical Engineering, University of Polytechnique Montreal, Canada.  {\tt\small Email: roland.malhame@polymtl.ca}}
}

\title{Explicit Sequential Equilibria in LQ Deep Structured Games and    Weighted  Mean-Field Games: A Unified Non-Standard Riccati Equation}

\title{Explicit Sequential Equilibria in LQ Deep Structured Games and    Weighted  Mean-Field Games}

\begin{document}
\maketitle
\thispagestyle{empty}
\pagestyle{empty}


\begin{abstract}
In this paper, we investigate  a class of  nonzero-sum dynamic stochastic  games, where  players have linear dynamics  and quadratic cost functions. The players are coupled in both dynamics and cost through  a linear regression (weighted average) as well as a  quadratic regression (weighted covariance  matrix) of  the states and actions, where  the linear regression of states is called deep state. We study collaborative and non-collaborative games under three information structures:  perfect sharing, deep state sharing, and no sharing for three different types of weights: positive, homogeneous and asymptotically vanishing weights.   For perfect  and  deep state sharing information structures, we propose a new  technique  by introducing  gauge transformation to solve for the best-response equations of players and    identify a few sufficient conditions  under which a  unique subgame perfect Nash equilibrium exists. The equilibrium is linear in the local state and  deep state, and the  corresponding  gains are obtained by solving  a novel  non-standard Riccati equation (whose   dimension  is independent of the number of players, thus making the solution scalable).  
When the information structure is  no-sharing and the number of players is asymptotically large, one asymptotic population-size-dependent equilibrium and one  asymptotic population-size-independent  equilibrium (also called sequential weighted mean-field equilibrium)  are proposed, and their convergence  to the infinite-population limits are established.
 In addition, the main results are extended   to  infinite-horizon cost function, and   generalized   to multiple  linear regressions and heterogeneous sub-populations. A numerical example is provided to demonstrate the difference between  the  two proposed approximate  equilibria. To the best of our knowledge, this paper  is  the first to propose a  unified framework to obtain  the exact solution in   closed form for an arbitrary number of  players. 
\end{abstract}

\section{Introduction}\label{sec:introduction}
The history of nonzero-sum linear quadratic (LQ)  games  can be traced back to work of  Case \cite{case1969toward} and Starr and Ho~\cite{starr1969a,starr1969b}  over five decades ago.  Early results  demonstrated  that  finding the Nash solution with perfect information structure  can be much more  difficult than  solving the LQ optimal control problem.  For  instance,  a   two-player game with strictly convex  cost functions may be not playable (i.e.,  the game  may not  admit a solution)~\cite{lukes1971global};  may admit  uncountably many   solutions (where each solution leads to a different set of pay-offs)~\cite{basar1976uniqueness};  may admit a unique Nash solution in a stochastic game  but admit infinitely many  solutions  in  its deterministic form~\cite{basar1976uniqueness},  or  may have a unique Nash solution for an arbitrary   finite horizon  but have no solution, a unique solution, or infinitely many solutions in the infinite horizon~\cite{Papavassilopoulos1984}. For a stochastic LQ game, it is well known  that the necessary and sufficient condition for the Nash solution to exist uniquely is equivalent to the existence and uniqueness condition for the solution of  a particular  set of \emph{coupled} Riccati equations~\cite{basar1976uniqueness}.
For a deterministic LQ game,  however,  it is   more difficult to establish  an existence condition for a unique Nash solution because the above-mentioned  coupled Riccati equations do not  necessarily identify all  the possible Nash solutions. The interested reader is referred to~\cite{eisele1982nonexistence,engwerda1998open}  for some counterexamples where the coupled Riccati equations do not have a solution but the deterministic game admits  one,  and to~\cite{basar1976uniqueness} for a counterexample where the coupled Riccati equations have a unique solution but  the deterministic game admits uncountably many solutions.  In addition, as  the number of players increases,  the  curse of dimensionality  exacerbates the above challenges from a computational  viewpoint.  
On the other hand, due to  limited communication resources  in practice,  it may not be feasible to implement   the perfect information structure,  especially when the number of players is large. In such a case,  one may need to consider  ``imperfect"  structures with players having private information,  making  it conceptually difficult for players to reach an agreement. Finding the  Nash solution in  LQ games under  imperfect information structure    involves  non-convex optimization,  which  in general  destroys  the linearity of the solution. For instance, it is shown in~\cite{Basar1974counter} via a counterexample that nonlinear strategies  can   outperform linear strategies.  As a result,   it has been  a long-standing challenge  in game theory  to identify  games that possess  tractable solutions  irrespective of  the number of players. The reader is referred to~\cite{Basarbook,dockner2000differential,engwerda2005lq} for more details on  the theory and application of LQ games.

Due to difficulties outlined in the previous paragraph and inspired by recent developments in  deep structured teams~\cite{Jalal2019risk,Jalal2019MFT},    we  consider a class of  LQ deep structured games  wherein  the interaction between players is modelled by a set of linear regressions (weighted averages) of the states and actions of players.    These models are called deep structured because of the fact that  the interaction between the players  is  similar to  that of the neurons of a deep feed-forward neural network.  In this paper, we focus on three classes of weights: homogeneous weights, asymptotically vanishing (negligible) weights, and positive weights with social cost function. To this end,  we consider  three information structures: perfect  sharing,  deep state  sharing and no  sharing,  where deep state refers to the linear regression of the states.   In perfect sharing  information structure, every player has access to the  joint state vector; in the deep state sharing, every player has only access to its local state and the deep state, and  in the no-sharing structure, every player has only access to  its local state.  Under perfect  and  deep state  sharing information structures,  we introduce a transformation-based technique  to identify  a  few sufficient conditions for the existence  of a  unique subgame perfect Nash equilibrium  in terms of a novel scalable non-standard Riccati equation.  In addition, we propose two  approximate  sequential equilibria  for  the no-sharing information structure that are both  sequential asymptotic reciprocal and sequential asymptotic Nash  equilibria  under mild assumptions, and  converge to the subgame perfect Nash equilibrium as the number of players goes to infinity.

The closest field of research to   LQ deep structured games is 
LQ  mean-field  games,  which is a special  class of exchangeable games with  homogeneous weights and asymptotically large number of players~\cite{huang2003individual,huang2007large,Lasry2006I,Lasry2006II}.
 In the simplest formulation of LQ mean-field games,  the players are  coupled  in the dynamics and cost  through the empirical mean of states (i.e.,  deep state with homogeneous weights) with  a tracking cost formulation. When  the number of players  goes to infinity,   the effect of a single player on other players  is  negligible; hence,  the  infinite-population  game reduces  to a \emph{constrained}  optimal control problem where: (a)  a generic player computes its best-response strategy  by solving an optimal tracking problem  described by a backward Hamilton-Jacobi-Bellman equation  with respect to  some hypothesized reference trajectory, and (b)   for the reference trajectory to be admissible, it is constrained  to be equal to    
the mean-field (deep state of the infinite-population game with homogeneous weights) whose dynamics is  expressed  by a forward advection or Fokker-Planck-Kolmogorov equation.  This is a fixed-point requirement, and sufficient conditions can be imposed to guarantee the existence of a unique solution~\cite{huang2007large}.  The resultant equilibrium is called mean-field equilibrium,  which  is not generally a sequential equilibrium, and its extension  to   non-negligible components, e.g., major-minor and common-noise  is conceptually challenging because the mean-field trajectory becomes unpredictable.  On the other hand,  when players use identical state feedback strategies and share the same cost function, the mean-field  converges to  the expectation of the generic player's state due to the strong law of large numbers.  This converts  the above two-body optimization problem into  a one-body optimization problem of McKean-Vlasov type, where the expectation of the state appears in the dynamics and cost function. The  resultant  problem  is known as the LQ mean-field-type  game problem.  The solution of this problem is obtained by forward-backward equations in~\cite{carmona2013control} and  by backward  equations (Riccati equations) in~\cite{yong2013linear,elliott2013discrete,duncan2018linear} In contrary to the above  results  that are only relevant when the number of players is  large, the authors in~\cite{bardi2014linear,Priuli2015} consider a finite-population  LQ mean-field  game wherein  players are  decoupled in   dynamics. When attention is restricted to  Gaussian random variables with  identical and affine stationary feedback  strategies, the finite-population game reduces to a constrained optimal control problem where the existence of a  Nash solution is conditioned on the solution of a  coupled  algebraic Riccati-Sylvester equation.\footnote{Note that the above mean-field models  are  continuous-time  while our model is in  discrete-time. However,  it is straightforward to translate  the results from continuous time to discrete time and vice versa, specially for  LQ models.}  
For more details on the theory and applications of mean-field  models, the reader is referred to~\cite{AdlJohWei2015Equilibria,Bensoussan2016, cardaliaguet2015master, Tembine2020book,Caines2018book,carmona2018probabilistic} and references  therein.

In contrast to LQ mean-field games and mean-field-type games,  we  consider a more general formulation that  encompasses mean-field  games and mean-field-type  games models, simultaneously, with any finite number of players (that does
not necessarily reduce to an optimal control problem). In particular, we   present an  explicit formulation for the   subgame perfect Nash equilibrium in terms of  a novel  non-standard Riccati equation,  by introducing a gauge transformation technique,   for any arbitrary number of players (not necessarily large) with possibly correlated noises (not necessarily independent) and population-size-dependent models (not necessarily population-size-independent).  Furthermore,  when the number of players is sufficiently large, we  compute the infinite-population solution for any arbitrary vanishing weights (not necessarily homogeneous weights). For no-sharing information structure,  we consider two approximate solution concepts, called sequential asymptotic reciprocal and Nash equilibria, and  propose one population-size-dependent and one population-size-independent equilibria, where the former often outperforms the latter.
For a more detailed comparison with mean-field models, the reader is referred to Section~\ref{sec:special}.  It is to be noted that our  framework  is different from the one in~\cite{bardi2014linear,Priuli2015} because we provide a stronger result albeit  under a richer information structure. In particular, we consider  the case in which players are  coupled in both dynamics and cost, and random variables are not necessarily Gaussian. Without restricting attention  to any particular class of strategies or infinite-horizon cost function, we  show  that the  Nash equilibrium is identical across players  and is  linear  in the local state and deep state.

The remainder of this paper is organized as follows. In Section~\ref{sec:problem}, the problem  is formulated and  the main  contributions of the paper are outlined. In Section~\ref{sec:mfs},  the subgame perfect Nash solution  under  perfect and  deep state   sharing  information structures is  identified, and   in Section~\ref{sec:ns},  two  approximate sequential equilibria  under  no-sharing information structure are proposed.  The main results are extended to infinite-horizon cost function in Section~\ref{sec:infinite}, and  a detailed comparison with mean-field models is presented in Section~\ref{sec:special}.  In Section~\ref{sec:generalization}, a brief discussion on  possible generalizations  of the presented results is provided and in Section~\ref{sec:numerical},  a numerical example is given to validate the theoretical findings. Finally in  Section~\ref{sec:conclusion}, the main results are summarized and conclusions are drawn.

\subsection*{Notation}
Throughout this paper,  $\mathbb{N}$ and $\mathbb{R}$ denote   natural and real numbers, respectively. For  any $n \in \mathbb{N}$, the short-hand notations $x_{1:n}$ and $\mathbb{N}_n$ are used to denote  the vector $(x_1,\ldots,x_n)$ and  the set $\{1,\ldots,n\}$, respectively.   For any vectors $x$, $y$, and $z$, $\VEC(x, y, z)$ represents the vector $[x^\intercal, y^\intercal, z^\intercal]^\intercal$. Given any square matrices $A$, $B$, and $C$, $\DIAG(A, B, C)$ is the block diagonal  matrix with matrices $A$, $B$ and $C$ on its main diagonal.  Let $A$ be an $n \times n$ block matrix; then $A^{i,j}$ refers to the  block matrix  located  at the $i$-th row and the $j$-th column, $i,j \in \mathbb{N}_n$.  For any $n,m \in \mathbb{N}$, $\mathbf{1}_{n\times m}$ and $\mathbf 0_{n \times m}$ denote  an $n$-by-$m$ matrix whose   arrays are all ones and  zeros, respectively.  Given a set  $\mathbf x=\{x_1,\ldots, x_n\}$,  $\mathbf x^{-i}$  is the set without  its $i$-th component, $i \in \mathbb{N}_n$. $\TR(\boldsymbol \cdot)$ is the trace of a matrix, $\VAR(\boldsymbol \cdot)$ is the auto-covariance  matrix of a random vector, $\Exp{\boldsymbol \cdot}$  is the expectation of an event, and $ \prec$  ($\preceq$) is the element-wise inequality.

\section{Problem formulation}\label{sec:problem}
Consider a  stochastic  dynamic game with  $n \in \mathbb{N}$  players  that have  linear dynamics with quadratic cost functions.  In  order to  have a meaningful game, the number of players  $n$ is assumed to be greater than $1$ in the sequel. Let $x^i_t \in \mathbb{R}^{d_x}$,  $u^i_t \in \mathbb{R}^{d_u}$ and $w^i_t \in \mathbb{R}^{d_x}$,   $d_x,d_u \in \mathbb{N}$,  denote  the local  state,  the local action and the local noise of player $i \in \mathbb{N}_n$  at time $t \in \mathbb{N}_T$, where $T \in \mathbb{N}$ denotes the game horizon.  Denote  by $\mathbf{x}_t=\VEC(x^1_t,\ldots,x^n_t) $, $\mathbf{u}_t=\VEC(u^1_t,\ldots,u^n_t)$ and  $\mathbf{w}_t=\VEC(w^1_t,\ldots,w^n_t)$   the stacked state,  the stacked   action, and the stacked noise at time $t \in \mathbb{N}_T$, respectively.  Let $\alpha^i_n \in \mathbb{R}$ denote the \emph{influence factor} (weight) of player $i$ among its peers, where interactions between players are modelled by the following linear regressions (weighted averages):
\begin{equation}\label{eq:def:mean-field}
 \bar x_t:= \sum_{i=1}^n \alpha^i_n x^i_t,  \quad  \bar u_t:= \sum_{i=1}^n \alpha^i_n u^i_t.
\end{equation} 
For ease of reference and inspired by  recent developments in deep structured models~\cite{Jalal2019risk,Jalal2019MFT}, we refer to the above linear regressions as \emph{deep state} and \emph{deep action} in the sequel. We assume that the weights are normalized such that $\sum_{i=1}^n \alpha^i_n=~1$.
 
The dynamics of player $i \in \mathbb{N}_n$  at time $t \in \mathbb{N}_T$  is  coupled to other players as follows:
\begin{equation}\label{eq:dynamics-mf}
 x^i_{t+1}=A_t x^i_t+B_tu^i_t+ \bar A_t \bar{x}_t+ \bar B_t \bar{ u}_t+ w^i_{t},
\end{equation}
where  $A_t$, $B_t$, $\bar A_t$ and $\bar B_t$ are  time-varying matrices of appropriate dimensions, and  $w^i_t$ is  a zero-mean local noise process.  Let the  auto-covariance matrices of the initial states  and local noises be  uniformly bounded in $t$ and $n$, i.e., there exist  constant matrices $C_x$ and $C_w$ (that do not depend on $t$ and $n$) such that   $\mathbf 0  \preceq \VAR(\mathbf x_1) \preceq C_x$ and  $\mathbf 0  \preceq \VAR(\mathbf w_t) \preceq C_w$, $t \in \mathbb{N}_T$, $n \in \mathbb{N}_n$.  Note that  the boundedness of the auto-covariance matrices  is a realistic assumption in practice.   The random variables $(\mathbf{x}_1,\mathbf w_1,\ldots, \mathbf w_T)$  are defined on a common probability space and are mutually independent. 

 The per-step cost of player $i \in \mathbb{N}_n$ at time $t \in \mathbb{N}_T$ is a function of its local state and  action as well as the empirical  weighted  mean and
covariance matrix of the states and actions  of all players as follows:
\begin{align}\label{eq:cost-mfs}
c^i_t(\mathbf x_t, \mathbf u_t)&= ({x^i_t})^\intercal Q_t x^i_t+ 2({x^i_t})^\intercal S^x_t \bar x_t + ({\bar x_t})^\intercal \bar Q_t \bar x_t  \nonumber \\
& +({u^i_t})^\intercal R_t u^i_t+2 ({u^i_t})^\intercal S^u_t \bar u_t + ({\bar u_t})^\intercal \bar R_t \bar u_t \nonumber \\
&+ \sum_{j=1}^n  \alpha^j_n \big( ({x^j_t})^\intercal G^{x}_t x^j_t +  \sum_{j=1}^n ({u^j_t})^\intercal G^{u}_t u^j_t\big),
\end{align}
where  matrices  $Q_t,R_t,S^x_t,S^u_t,\bar Q_t$, $\bar R_t,G^x_t$ and $G^u_t$  are symmetric   with  appropriate dimensions.  It is also  possible to consider  cross terms  between the states and actions in~\eqref{eq:cost-mfs}; however,  since such an extension is  straightforward, it is not considered  here for simplicity of presentation.

\begin{Remark}
\emph{The  dimensions of the matrices defined above are independent of the number of players $n$; however,  their values can be population-size-dependent, e.g. $\bar A_t=n^2 \mathbf I_{d_x\times d_x}$. In addition, the initial states and driving noises can be arbitrarily correlated across players (not necessarily independent).
} 
\end{Remark}
Below, we define three types of weights.
\begin{Definition}
The weights may belong to the following types.
\begin{itemize}
\item Positive weights. This is when  $\alpha^i_n>0, \forall i  \in \mathbb{N}_n$.
\item  Homogeneous weights. This is when the weights are positive and equal, i.e. $\alpha^i_n=1/n, \forall i  \in \mathbb{N}_n$.
\item Asymptotically vanishing weights. This is when there exists  a sufficiently large $n_0 \in \mathbb{N}$ such that for every $n_0<n$,  $\alpha^i_n=\beta^i/n$, $\beta^i \in [-\beta_{\text{max}},\beta_{\text{max}}]$, $\beta_{\text{max}}>0$.\footnote{This can be generalized to any  $\alpha^i_n$ such that $\lim_{n \rightarrow \infty} \alpha^i_n=0$.}
\end{itemize}
\end{Definition}
 Let $\mathcal{A}_n:=\{\alpha^i_n, \forall i \in \mathbb{N}_n\}$ denote the set of all feasible weights. It is observed that  for homogeneous weights, $\mathcal{A}_n=\{\frac{1}{n}\}$ is a singleton, irrespective of the number of players $n$. In addition,  if all weights are asymptotically vanishing weights (including homogeneous weights), $\mathcal{A}_\infty=\{0\}$ is a singleton, irrespective of $\beta^i$, $i \in \mathbb{N}$.  In this paper,  we mainly focus on  the homogeneous  and asymptotically vanishing weights, except the common cost that allows for arbitrary positive weights.

For  the infinite-population limit, the  following standard assumption will be  imposed on the model accordingly.
 \begin{Assumption}[Population-size-independent model]\label{assump:independent_n}
All  matrices in the dynamics~\eqref{eq:dynamics-mf} and per-step cost function~\eqref{eq:cost-mfs}  are independent of the number of players $n$. 
\end{Assumption}
\begin{Remark}
\emph{The main results of this paper pertained to  the asymptotically vanishing weights hold under Assumption~\ref{assump:independent_n} for a sufficiently  large number of players.  However, those of   homogeneous  and positive weights   hold   for arbitrary number of players (not necessarily large) with  possibly  population-size-dependent   models.}
\end{Remark}
For simplicity of analysis, we restrict attention to a single linear regression of states and a single population  in this article. However,  main results  can  naturally be extended to a setup with multiple  linear regressions and  heterogeneous  sub-populations, where  the interaction between the players is modelled by a set of  deep states and deep actions. See~\cite{Jalal2019risk}  for an analogous model and Section~\ref{sec:generalization} for a brief discussion.  

\subsection{Exchangeable players}
In this section,  it is shown that any LQ game with exchangeable players is  equivalent to an LQ deep structured  game,  where all  players are equally important, i.e., they have  homogeneous weights   $\alpha^i_n=1/n$, $\forall i \in \mathbb{N}_n$. To this end, we  consider a general LQ game in the following augmented form:
\begin{equation}\label{eq:dynamics-general}
\mathbf{x}_{t+1}=\mathbf A_t  \mathbf x_t+\mathbf B_t \mathbf u_t+ \mathbf{w}_t,
\end{equation} 
\begin{equation}\label{eq:cost-general}
\mathbf c_t(\mathbf x_t,\mathbf u_t)=\VEC(c^1_t(\mathbf x_t, \mathbf u_t),\ldots,c^n_t(\mathbf x_t, \mathbf u_t)),
\end{equation}
where  the  per-step cost function of player $i$  is given by:
\begin{equation}
c^i_t(\mathbf x_t, \mathbf u_t)= {\mathbf x_t}^\intercal Q^i_t {\mathbf x_t} + {\mathbf u_t}^\intercal R^i_t {\mathbf u_t}.
\end{equation}
\begin{Definition}[Exchangeable players]\label{def:exchangeable-players}
The players  are called exchangeable if their dynamics and cost functions  are invariant to the way the players are indexed. In particular,  for every pair $i,j \in \mathbb{N}_n$, the following holds:
\begin{itemize}
\item $
\sigma_{i,j} (\mathbf{x}_{t+1})=\mathbf A_t (\sigma_{i,j} (\mathbf x_t))+\mathbf B_t (\sigma_{i,j} (\mathbf u_t))+ \sigma_{i,j} (\mathbf{w}_t),
$
\item $       \sigma_{i,j}(\mathbf c_t(\mathbf x_t,\mathbf u_t))=\mathbf c_t(\sigma_{i,j}(\mathbf x_t),\sigma_{i,j}(\mathbf u_t))$,
\end{itemize}
where $\sigma_{i,j}$ swaps $i$-th and $j$-th components.
 \end{Definition}
 \begin{Proposition}\label{proposition:exchangeable}
Let  the dynamics~\eqref{eq:dynamics-general} and cost functions~\eqref{eq:cost-general} be exchangeable. There exist matrices $(A_t,B_t,\bar A_t,\bar B_t)$ such that the dynamics of player $i \in \mathbb{N}_n$ can be expressed by~\eqref{eq:dynamics-mf}, where  $\alpha^i_n=1/n$, $\forall i \in \mathbb{N}_n$.   In addition, there exist matrices $Q_t,R_t,S^x_t,S^u_t,\bar Q_t$, $\bar R_t,G^x_t$ and $G^u_t$  such that the cost  function of player $i$   can be written as~\eqref{eq:cost-mfs}, where  $\alpha^i_n=1/n$, $\forall i \in \mathbb{N}_n$.
\end{Proposition}
\begin{proof} 
The proof is presented in~Appendix~\ref{sec:proof_proposition:exchangeable}. 
\end{proof}
Proposition~\ref{proposition:exchangeable} holds irrespective of the information structure of players.  It is  worth highlighting that any exchangeable  function with Lipschitz bounds  can be equivalently expressed as a function of empirical distribution~\cite{carmona2018probabilistic}.  For more  details on  exchangeable functions, see~\cite[Chapter 2]{arabneydi2016new}.
\begin{Remark}[Mean-field models]
\emph{It is to be noted that cost function~\eqref{eq:cost-mfs} is the most general representation of  a generic LQ game  with exchangeable players, where weights are homogeneous. This generalization leads to a unified framework for numerous cases of exchangeable models such  as  mean-field games and  mean-field-type  games, which have been studied separately within the context of  LQ games. Also, we are particularly interested in the so-called extended mean-field games~\cite{gomes2016extended}, where  the effect of the aggregate action is accounted for (this type of games has applications in financial models). Importantly, we strive to give an exact and explicit  but  scalable  version of the solution   to the finite-population version  of these games.    In addition,  the covariance term in~\eqref{eq:cost-mfs} is of particular interest  in applications such as robust and quantile sensitive games, where high variance is not desirable.  
 }
\end{Remark}

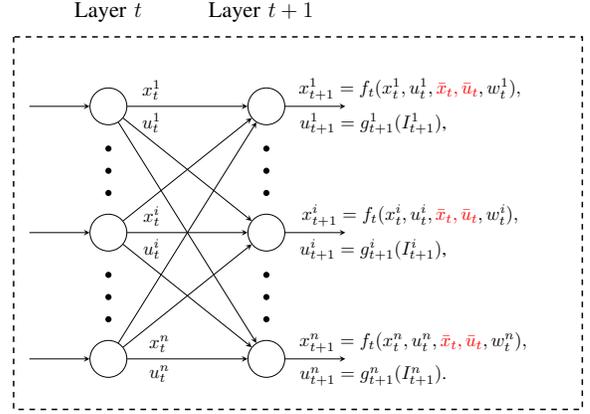
\begin{figure}[t!]
\hspace{-.5cm}
\scalebox{.7}{
\begin{tikzpicture}[x=1.5cm, y=1.2cm, >=stealth]
  \draw[font=\large](0,3cm) node {Layer $t$} ;
\draw[font=\large](2.9cm,3cm) node {Layer $t+1$} ;
\draw[black, thick, dashed] (-1.2,2.1) rectangle (6,-3.8);

\foreach \m/\l [count=\y] in {1,missing,2,missing,3}
  \node [every neuron/.try, neuron \m/.try] (input-\m) at (0,2-\y) {};

\foreach \m/\l [count=\y] in {1,missing,2,missing,3}
  \node [every neuron/.try, neuron \m/.try ] (output-\m) at (2,2-\y) {};

  \draw [<-] (input-1) -- ++(-1,0)
    node [above,midway] {\hspace{3.5cm}$x^1_t$}
        node [below,midway] {\hspace{3.5cm}$u^1_t$};
     \draw [<-] (input-2) -- ++(-1,0)
    node [above,midway] {\hspace{3.5cm}$x^i_t$}
        node [below,midway] {\hspace{3.5cm}$u^i_t$};
        
          \draw [<-] (input-3) -- ++(-1,0)
    node [above,midway] {\hspace{3.8cm}$x^{n}_t$}
        node [below,midway] {\hspace{3.8cm}$u^{n}_t$};

  \draw [->]  (output-1) -- ++(1,0)  node[above, midway]  
   {\hspace*{3.6cm}$x^1_{t+1}= f_t(x^1_t,u^1_t, \textcolor{red}{\bar{x}_t,\bar{u}_t},w^1_t),$}
     node [below,midway] {\hspace{2.2cm}$u^1_{t+1} = g^1_{t+1}( I^1_{t+1})$,};   
      \draw [->]  (output-2) -- ++(1,0)  node[above, midway]     {\hspace*{3.6cm}$x^i_{t+1}= f_t(x^i_t,u^i_t, \textcolor{red}{\bar{x}_t,\bar{u}_t},w^i_t),$}
     node [below,midway] {\hspace{2.2cm}$u^i_{t+1}= g^i_{t+1}( I^i_{t+1})$,};  
    
      \draw [->]  (output-3) -- ++(1,0)  node[above, midway]  { \hspace*{3.6cm} $x^n_{t+1}= f_t(x^n_t,u^n_t, \textcolor{red}{\bar{x}_t,\bar{u}_t},w^n_t),$}
     node [below,midway] {\hspace{2.2cm}$u^n_{t+1}= g^n_{t+1}( I^n_{t+1})$.};     
     
    \foreach \i in {1,...,3}
  \foreach \j in {1,...,3}
    \draw [->] (input-\i) -- (output-\j);    
\end{tikzpicture}}
\caption{The coupling  between the players in a deep structured model is similar to that between the neurons of a deep feed-forward neural network. In this paper,  function $f_t$ is a multivariate affine function,  and $I^i_t$ denotes the information set of player $i$ at time $t$.}
\end{figure}
\begin{figure}[t!]
\hspace{-.5cm}
\scalebox{0.7}{
\begin{tikzpicture}[x=1.5cm, y=1.2cm, >=stealth]
  \draw[font=\large](0,4cm) node {Layer $t$} ;
\draw[font=\large](2.9cm,4cm) node {Layer $t+1$} ;
\draw[black, thick, dashed] (-1.2,3) rectangle (6,-3.8);

\foreach \m/\l [count=\y] in {1,missing,2,missing,3}
  \node [every neuron/.try, neuron \m/.try] (input-\m) at (0,2-\y) {};

\foreach \m/\l [count=\y] in {1,missing,2,missing,3}
  \node [every neuron/.try, neuron \m/.try ] (output-\m) at (2,2-\y) {};

\draw[->, color=black, line width=.5mm] (-1,2.1) -> (5.5,2.1) node [label=Weighted mean-field trajectory $\textcolor{red}{m^x_{1:T}}$ and $\textcolor{red}{m^u_{1:T}}$] at (3,2.1) {};

\draw[->, color=black, line width=.5mm] (0,2.1) -> (0,1.3);
\draw[->, color=black, line width=.5mm] (2,2.1) -> (2,1.3);

\foreach \m/\l [count=\y] in {1,missing,2,missing,3}
  \node [every neuron/.try, neuron \m/.try] (input-\m) at (0,2-\y) {};

\foreach \m/\l [count=\y] in {1,missing,2,missing,3}
  \node [every neuron/.try, neuron \m/.try ] (output-\m) at (2,2-\y) {};

  \draw [<-] (input-1) -- ++(-1,0)
    node [above,midway] {\hspace{3.5cm}$x^1_t$}
        node [below,midway] {\hspace{3.5cm}$u^1_t$};
     \draw [<-] (input-2) -- ++(-1,0)
    node [above,midway] {\hspace{3.5cm}$x^i_t$}
        node [below,midway] {\hspace{3.5cm}$u^i_t$};
        
          \draw [<-] (input-3) -- ++(-1,0)
    node [above,midway] {\hspace{3.8cm}$x^{n}_t$}
        node [below,midway] {\hspace{3.8cm}$u^{n}_t$};

  \draw [->]  (output-1) -- ++(1,0)  node[above, midway]  
   {\hspace*{3.6cm}$x^1_{t+1}= f_t(x^1_t,u^1_t, \textcolor{red}{m^x_t,m^u_t},w^1_t),$}
     node [below,midway] {\hspace{2.2cm}$u^1_{t+1} = g^1_{t+1}(x^1_{t+1})$,};   
      \draw [->]  (output-2) -- ++(1,0)  node[above, midway]     {\hspace*{3.6cm}$x^i_{t+1}= f_t(x^i_t,u^i_t, \textcolor{red}{m^x_t,m^u_t},w^i_t),$}
     node [below,midway] {\hspace{2.2cm}$u^i_{t+1}= g^i_{t+1}(x^i_{t+1})$,};  
    
      \draw [->]  (output-3) -- ++(1,0)  node[above, midway]  { \hspace*{3.6cm} $x^n_{t+1}= f_t(x^n_t,u^n_t, \textcolor{red}{m^x_t,m^u_t},w^n_t),$}
     node [below,midway] {\hspace{2.2cm}$u^n_{t+1}= g^n_{t+1}(x^n_{t+1})$.};    
    \foreach \i in {1,...,3}
    \draw [->] (input-\i) -- (output-\i);    
\end{tikzpicture}}
    \caption{
    Due to the asymptotically vanishing effect of individuals, one can express the infinite-population game as  a two-player game between  a generic player and an infinite-population player. When local noises are independent,  the dynamics of the deep state (i.e., the state of the infinite-population player) becomes deterministic; hence,  the deep state reduces to an external predictable trajectory  called (weighted) mean-field.}
    \end{figure}
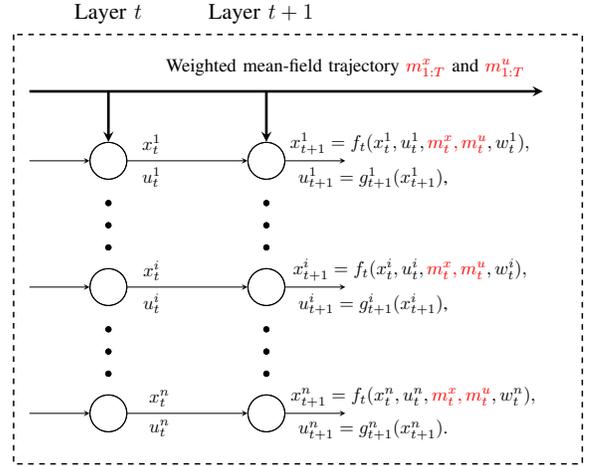  
\subsection{Information structure}
In this paper, we consider three information structures. The first one is called \emph{Perfect Sharing} (PS),  where every player   has access to the stacked  state, i.e., the control action of player $i \in \mathbb{N}_n$ at time $t \in \mathbb{N}_T$ is given by: 
\begin{equation}
 u^i_t=g^i_t(\mathbf x_{1:t}), \tag{\text{PS} }
\end{equation}
where $g^i_t: \mathbb{R}^{ntd_x}   \rightarrow \mathbb{R}^{d_u}$  is a measurable function adapted to the $\sigma$-algebra generated by $\{\mathbf x_1, \mathbf{w}_1, \ldots, \mathbf w_t\}$. The second information structure is  called \emph{Deep State Sharing} (DSS),  and  requires  that every  player observes  its  local state as well as  the deep state, i.e., the control action of player $i \in \mathbb{N}_n$ at time $t \in \mathbb{N}_T$ is  described by:
\begin{equation}
 u^i_t=g^i_t(x^i_{1:t}, \bar x_{1:t}), \tag{\text{DSS} }
\end{equation}
where $g^i_t: \mathbb{R}^{2td_x}   \rightarrow \mathbb{R}^{d_u}$ is a measurable function adapted to the $\sigma$-algebra generated by $\{x^i_1,\bar x_1, w^i_1,\bar w_1,\ldots, w^i_T,\bar w_T\}$.   In practice, there are various applications in which DSS  is plausible.  For example,  it is a common practice  in  the stock markets to provide players (e.g., buyers, sellers and brokers) with statistical data  on the total amount of  shares, trades and  exchanges   It is also possible to communicate  the deep state, without a central authority,  by means of consensus algorithms.
The third information structure  is called \emph{No Sharing} (NS), where each player knows nothing more than  its local state:
\begin{equation}
u^i_t=g^i_t(x^i_{1:t}),  \tag{\text{NS}}
\end{equation}
where  $g^i_t: \mathbb{R}^{td_x} \rightarrow \mathbb{R}^{d_u}$ is a measurable function adapted to the $\sigma$-algebra generated by $\{x^i_1,w^i_1,\ldots,w^i_T\}$.  When  the number of players is very large and  sharing the deep state is infeasible,  NS information structure is more  practical as  it  requires no   communication among players, except at the beginning of the control horizon. In general,  DSS and NS information structures are appropriate  for    cloud-based  and  cyber-security applications  wherein players are concerned with sharing their private states.

The admissible set of control actions is assumed to be  square summable, i.e.,  given any $\gamma \in (0,1)$, $\Exp{\sum_{t=1}^T (u^i_t)^\intercal u^i_t} < \infty$ and  $\Exp{\lim_{T \rightarrow \infty} (1-\gamma) \sum_{t=1}^T \gamma^{t-1} (u^i_t)^\intercal u^i_t} <  \infty$, $\forall i \in \mathbb{N}_n$. In the sequel,  we refer to  $\mathbf g^i_n:=g^i_{1:T}$  as the \emph{strategy of player} $i \in \mathbb{N}_n$ and to  $\mathbf{g}_n:= \{g^1_{1:T}, \ldots, g^n_{1:T}\}$  as the \emph{strategy of  the game}. \edit{In addition, we denote  by $\mathcal{G}^i$  the set of strategies  of player $i \in \mathbb{N}_n$, and by   $\GC$, $\GMFS$ and $\GNS$   the set of PS, DSS and NS strategies, respectively.}

\subsection{Solution concepts}
Define the following  cost-to-go function  for   player $i\in \mathbb{N}_n$:
\begin{equation}\label{eq:total_cost}
J^i_n(\mathbf g^i_n, \mathbf  g^{-i}_n)_{t_0}:=\mathbb{E}^{\mathbf g_n}\big[\sum_{t=t_0}^T  c^i_t(\mathbf x_t, \mathbf u_t) \big], \quad t_0 \in \mathbb{N}_T,
\end{equation}
where the above expectation is taken with respect to the probability measures  induced by the choice of  strategy $\mathbf g_n$.
\edit{For PS  information structure, we  consider a standard sequential solution concept introduced in~\cite{selten1975reexamination} as follows.
\begin{Definition}[Subgame Perfect Nash Equilibrium (SPNE)]\label{def:DSS}
Any strategy $\mathbf g^\ast_n$ with perfect  sharing information structure is said to be a subgame perfect Nash equilibrium if for any player $i \in \mathbb{N}_n$ at any stage of the game $t_0\in \mathbb{N}_T$:
\begin{equation}
{J^i_n(\mathbf g^{\ast,i}_n,\mathbf g^{\ast,-i}_n)}_{t_0} \leq {J^i_n(\mathbf g^{ i},\mathbf g^{\ast,-i}_n)}_{t_0}, \quad \forall \mathbf g^i \in \GC.
\end{equation}
\end{Definition}
For no-sharing information structure,  we  present two \emph{approximate} solution concepts. The first one is motivated from the role of reciprocation  in game theory and  introduced in~\cite{Jalal2020Nash}. 
\begin{Definition}[Sequential Asymptotic  Reciprocal Equilibrium (SARE)]\label{def:pDSS}
Any strategy $\hat{\mathbf g}$ with imperfect information structure   is said to be  a  sequential asymptotic reciprocal equilibrium  if at any stage of the game $t_0 \in \mathbb{N}_T$, the performance loss of  every player  $i \in \mathbb{N}_n$  compared to that of the SPNE  is  less than a threshold $ \varepsilon_{t_0}(n) \geq 0$,  where  $\lim_{n \rightarrow \infty}\varepsilon_{t_0}(n)=0$. In particular, there exists $n_0 \in \mathbb{N}$ such that
\begin{equation}
| {J^i_n(\hat{\mathbf g}^i,\hat{\mathbf g}^{-i})}_{t_0} - {J^i_n(\mathbf g^{i,\ast}_n,\mathbf g^{\ast,-i}_n)}_{t_0}| \leq  \varepsilon_{t_0}(n), \quad  n_0 \leq \forall n.
\end{equation}
\end{Definition} 
The second one   is defined based on the notion of individual rationality and $\varepsilon$-Nash equilibrium in game theory.
\begin{Definition}[Sequential Asymptotic Nash Equilibrium (SANE)]\label{def:pDSS2}
Any strategy $\hat{\mathbf g}$ with imperfect  information structure is said to be a sequential  asymptotic Nash  equilibrium if   the benefit of every player $i \in \mathbb{N}_n$  from any  unilateral deviation at any stage of the game $t_0 \in \mathbb{N}_T$ does not exceed a threshold  $\bar \varepsilon_{t_0}(n)$,  where $\lim_{n \rightarrow \infty} \bar \varepsilon_{t_0}(n)=0$. In particular, there exists $n_0 \in \mathbb{N}$ such that for any $\mathbf g^i \in \mathcal{G}^i$:
\begin{equation}
{J^i_n(\hat{\mathbf g}^i,\hat{\mathbf g}^{-i})}_{t_0} \leq {J^i_n(\mathbf g^i,\hat{\mathbf g}^{-i})}_{t_0} + \bar  \varepsilon_{t_0}(n), \quad  n_0 \leq \forall n.
\end{equation}
\end{Definition}
 In general, SARE and SANE are conceptually two different solution concepts, and  they  are not  unique even if the SPNE is unique. In particular,  the existence of SARE implies the existence of  the finite-population SPNE solution $\mathbf g^\ast_n$ while the existence  of SANE is independent of the  finite-population  solution. This independence comes at a  price  that  SANE  does not provide much insight about  the finite-population  SPNE solution, which  can generally lead to counter-intuitive cases wherein the infinite-population equilibrium is SANE but it is not the limit of the finite-population solution.   In this paper, we introduce two NS approximate strategies that are both  SARE and SANE under mild assumptions. The first one is called   \emph{sequential asymptotic population-size-dependent equilibrium} (SAPDE) and the second one is called \emph{sequential weighted mean-field equilibrium} (SWMFE).   
 Note  that SWMFE is different from FBMFE (forward-backward mean-field equilibrium), in general.  For more information on  the relations between SARE, SANE, FBMFE, and SWMFE, the reader is referred to~\cite{Jalal2020Nash}. 
}

  We now formulate the following three problems.
\begin{Problem}\label{prob:aux}
If it exists, find a subgame perfect Nash equilibrium under perfect sharing information structure.
\end{Problem}
\begin{Problem}\label{prob:mfs}
If it exists, find a sequential  Nash strategy $\mathbf{g}^\ast_n $ under deep state sharing information structure  such that for any player $i \in \mathbb{N}_n$ at any stage of the game $t_0\in \mathbb{N}_T$:
\begin{equation}
{J^i_n(\mathbf g^{\ast,i}_n,\mathbf g^{\ast,-i}_n)}_{t_0} \leq {J^i_n(\mathbf g^{ i},\mathbf g^{\ast,-i}_n)}_{t_0}, \quad \forall \mathbf g^i \in \GC.
\end{equation}
\end{Problem}
\begin{Problem}\label{prob:ns}
If they  exist, find  sequential asymptotic reciprocal and sequential asymptotic  Nash equilibria under no-sharing information structure.
\end{Problem}

\subsection{Main  challenges and contributions}
To solve Problem~\ref{prob:aux}, one may  use a formulation based on the standard backward induction,  leading to an  investigation of  the  solution of $n$ coupled matrix Riccati  equations~\cite{Basarbook},  which  is  a difficult task in general.  In addition,   if a solution exists, its numerical computation suffers from  the curse of dimensionality, as  matrices in  Problem~\ref{prob:aux} are fully dense. 
  To solve Problem~\ref{prob:mfs} for a finite number of players,  neither standard nor LQ mean-field games~\cite{huang2007large} can be directly employed.  This is because the former generally  requires some additional information which is normally  not available, while the latter becomes a viable alternative only for a sufficiently  large number of players. To solve Problem~\ref{prob:ns}, on the other hand,  one can  use  an  LQ mean-field  game approach~\cite{huang2007large}  but  because of the considered objective functions, the problem  is not necessarily reducible  to the classical  mean-field game with a tracking cost function, in particular  because of  empirical actions in the cost function of players.  In addition, similarly to~\cite{huang2003individual,huang2007large},   the  mean-field game approach  tends to overshadow the role  of the number of players  in the solution by sending directly the number of players to infinity and then  evaluating the quality of the approximation.

The main contributions of this paper are outlined below.  
\begin{enumerate}
\item  We propose a systematic approach to solve LQ games with an arbitrary number of exchangeable players by introducing  a gauge transformation. More precisely,  we identify  a few  sufficient conditions  for the existence of a unique subgame perfect  Nash equilibrium  for Problems~\ref{prob:aux} and~\ref{prob:mfs}, where the effect of an individual player on others  is not necessarily negligible.   In addition,  we  obtain the  equilibrium  in an explicit scalable  form, in particular, one where the complexity of calculations is independent of the number of players (Theorem~\ref{thm:aux} and Corollary~\ref{thm:mfs}). The above-mentioned condition along with  the subgame perfect Nash equilibrium  is  described by  a novel non-standard Riccati equation.  Also, the main results are  generalized to weighted averages to  better distinguish  the differences between exchangeable games, infinite-population games and  collaborative games.

\item  We define performance and rationality gaps, and propose two no-sharing approximate strategies for Problem~\ref{prob:ns} in Theorems~\ref{thm:NS-optimality} and~\ref{cor:ns-infinite};   see~Remark~\ref{remark:differences-NS} and the numerical example presented in Section~\ref{sec:numerical} for more details on the differences between the two  strategies.

\item We   extend our main results  to the discounted cost  infinite-horizon  function (Theorems~\ref{thm:mfs-inf_inf}--\ref{cor:ns-infinite-inf}). 

\item   We study two special games wherein  the proposed non-standard Riccati equation  reduces  to two  standard Riccati equations. 
In addition,  we shed some light on  the non-uniqueness feature  of the solution in deterministic models as well as in infinite-population games, and   present a unified framework to study similarities and differences between diverse formulations of  mean-field games.
For example, when cost functions are identical, we show  that  the Nash and team-optimal  solutions coincide under PS and DSS structures. This is a surprising result  because  Nash and  team-optimal solutions are not necessarily the same  with  identical cost functions, where  local  convexity (Nash solution)  is different from global convexity (team-optimal solution).

\item  To illustrate the effect of the number of players on the subgame perfect Nash equilibrium,  we present  two counterexamples wherein  the finite-population game admits  a unique solution  while the infinite-population game has no solution,  and vice versa.  For the special case  of a social cost function,  however, we prove  that   under PS and DSS information structures,  the solution is insensitive to the individual  weights  of players.   In this case,  the two  proposed no-sharing  approximate strategies  are  identical and independent of the number of players.
  
\item  We briefly discuss  a few  generalizations of the proposed  non-standard Riccati equation in Section~\ref{sec:generalization}, including multiple   linear regressions and sub-populations, which is  applicable to  non-trivial cases such as  major-minor and common-noise problems with  the advantage that  no additional complication arises.
\end{enumerate}

For ease of reference, a summary  of assumptions used in this paper is outlined in Table~\ref{table:space}.

\begin{table}[ht]
\caption{A  summary of  assumptions used in this paper.}\label{table:space}
\centering
\setlength\tabcolsep{0em}
\begin{tabular}{l} 
  \toprule
\multicolumn{1}{c}{\textbf{Finite horizon conditions}}\\
\midrule
A 1.  Population-size-independent model \\
\midrule
A 2.  Sufficient condition for a unique solution with homogeneous weights\\
\midrule
A 3.  Sufficient condition for a unique solution in  infinite-population model\\
\midrule
A 4.  Sufficient condition for a unique solution in  infinite-population model \\
\hspace{.7cm} with decoupled standard Riccati equations\\
\midrule
A 5.  Sufficient condition for a unique solution of social cost function\\
\hspace{.7cm} with positive weights and decoupled standard Riccati equations\\
\midrule
A 6. Independent random variables for NS information structure\\
\midrule
A 7.  Existence of  finite-population solution with asymp. vanishing weights\\
\midrule
A 8.  Sufficient condition for asymptotic convergence of rationality gaps\\
  \toprule
\multicolumn{1}{c}{\textbf{Additional conditions in infinite horizon}}\\
 \midrule
A 9.   The infinite-horizon limit of the non-standard Riccati equation\\
 \midrule
 A 10.  A contractive solution for A 9.\\
 \midrule
 A 11.  A subspace invariance solution for A 4. and A.9\\
  \midrule
 A 12.  A subspace invariance solution for A 5. and A.9\\
   \midrule
 A 13.  Asymptotic stability condition for NS information structure\\
  \bottomrule
\end{tabular}
\end{table}

\section{Main Results for Problems~\ref{prob:aux} and~\ref{prob:mfs}}\label{sec:mfs}

In this section, we  present   solutions for  Problems~\ref{prob:aux} and~\ref{prob:mfs}. 

\subsection{Gauge transformation and non-standard Riccati equation}
Our analysis  crucially hinges on  a change of variables whereby we replace players' states, actions  and noises by a combination of deviations from  the weighted averages of state, action and noise, along with the  weighted quantities themselves. Define the following variables for any player $i \in \mathbb{N}_n$ at time $t \in \mathbb{N}_T$:
\begin{equation}\label{eq:def-breve}
\Delta x^i_t:=x^i_t - \bar x_t,  \quad \Delta u^i_t:=u^i_t - \bar u_t,\quad \Delta w^i_t:=w^i_t - \bar w_t,
\end{equation}
where $\bar w_t:=\sum_{i=1}^n \alpha^i_n w^i_t$.  
The above change of coordinates is a \emph{gauge transformation}, which is a powerful tool in the theoretical analysis of  invariant physics~\cite{moriyasu1983elementary}. In short, a gauge transformation manipulates  the degrees of freedom    of  a physical (invariant) system without altering  its behaviour.  It is to be noted that the gauge transformation technique  is more general than the completion-of-square method introduced  in the mean-field-type game~\cite{yong2013linear,elliott2013discrete}, which is a single-agent control problem. In particular,  the  transformation technique  was initially introduced in~\cite{arabneydi2016new}, showcased in mean-field teams~\cite{JalalCDC2015} and weighted mean-field teams~\cite{Jalal2017linear},  and   has recently been  extended  to  deep structured teams in~\cite{Jalal2019risk,Jalal2020CCTA}.  The present paper is the first work investigating  the efficiency of this technique  in  a nonzero-sum game, with both collaborative  and non-collaborative costs. Note also that the nonzero-sum game is
 a more  complicated problem than  the above control problems (as briefly mentioned in Section~\ref{sec:introduction}).

 For PS information structure, knowing $\mathbf x_t$ and $\mathbf u_t$  is equivalent to knowing $\tilde{\mathbf x}_t:=\VEC(\Delta x^1_t,\ldots,\Delta{x}^n_t,\bar x_t)$ and $\tilde{\mathbf u}_t:=\VEC(\Delta u^1_t,\ldots,\Delta u^n_t,\bar u_t)$,  respectively. 
From~\eqref{eq:def:mean-field},~\eqref{eq:dynamics-mf} and~\eqref{eq:def-breve},  it follows that for every  $i \in \mathbb{N}_n$:
\begin{equation}\label{eq:dynamics_joint}
\begin{bmatrix}
\Delta x^i_{t+1} \\
\bar{x}_{t+1}
\end{bmatrix}=\mathbf A_t\begin{bmatrix}
\Delta x^i_t \\
\bar{x}_t
\end{bmatrix}+ \mathbf B_t\begin{bmatrix}
\Delta u^i_t \\
\bar u_t \\
\end{bmatrix} + \begin{bmatrix}
\Delta w^i_t \\
\bar w_t \\
\end{bmatrix},
\end{equation}
where $
\mathbf  A_t:=\DIAG(A_t, A_t+\bar A_t)$ and   $\mathbf  B_t:=\DIAG(B_t,B_t+\bar B_t)$.
The gauge transformation introduces the  orthogonal relations:
\begin{equation}\label{eq:orthogonal_relation}
 \sum_{i=1}^n \alpha^i_n(\Delta{x}^i_t)^\intercal G^x_t \bar x_t=0, \quad \sum_{i=1}^n \alpha^i_n (\Delta{u}^i_t)^\intercal G^u_t \bar u_t=0.
\end{equation} 
From~\eqref{eq:orthogonal_relation}, the per-step cost of  player $i \in \mathbb{N}_n$ at time $t \in \mathbb{N}_T$, given by~\eqref{eq:cost-mfs}, can be expressed as follows:
\begin{align}\label{eq:cost-breve-bar} 
&c^i_t(\tilde{\mathbf x}_t,\tilde{ \mathbf u}_t):=\begin{bmatrix}
\Delta x^i_t \\
\bar x_t \\
\end{bmatrix}^\intercal  \mathbf Q_t^{\alpha^i_n} \begin{bmatrix}
\Delta x^i_t \\
\bar x_t \\
\end{bmatrix}+  \begin{bmatrix}
\Delta u^i_t \\
\bar u_t \\
\end{bmatrix}^\intercal  \mathbf R_t^{\alpha^i_n}  \begin{bmatrix}
\Delta u^i_t \\
\bar u_t \\
\end{bmatrix} \nonumber \\
&\quad+ \sum_{j=1}^n \alpha^j_n \big((\Delta x^j_t)^\intercal G^x_t (\Delta x^j_t) + (\Delta u^j_t)^\intercal G^u_t (\Delta u^j_t)\big) \nonumber \\
&\quad -\frac{\alpha^i_n}{1-\alpha^i_n} \big((\Delta x^i_t)^\intercal G^x_t (\Delta x^i_t)+(\Delta u^i_t)^\intercal G^u_t (\Delta u^i_t)\big),\\
\end{align}  
where 
\begin{align}\label{eq:def-barQ}
\mathbf Q_t^{\alpha^i_n}&:=\Big[ \begin{array}{cc}
Q_t +\frac{\alpha^i_n}{1-\alpha^i_n} G^x_t&  Q_t+S^x_t\\
Q_t+S^x_t  & Q_t+2S^x_t+\bar Q_t+G^x_t
\end{array}\Big], \nonumber \\
\mathbf R_t^{\alpha^i_n}&:=\Big[ \begin{array}{cc}
R_t +\frac{\alpha^i_n}{1-\alpha^i_n} G^u_t&  R_t+S^u_t\\
R_t+S^u_t  & R_t+2S^u_t+\bar R_t+G^u_t
\end{array}\Big].
\end{align}

In the sequel,  we occasionally
 use superscript $n$ and $\alpha$ to highlight  the dependence of some parameters with respect to  the number of players and  the individual  weight (which itself depends on the number of players), respectively.  We define a non-standard Riccati equation which will help formulate the  solution.  For any $t \in \mathbb{N}_{T}$ and $\alpha \in \mathcal{A}_n$, define matrices  $\mathbf P^\alpha_{1:T}$  backward in time as follows: 
    \begin{equation}\label{eq:riccati-bar-m}
    \begin{cases}
  \barM= \mathbf  Q_t^\alpha+ \mathbf A_t^\intercal \mathbf P^\alpha_{t+1} \mathbf A_t+    (\boldsymbol \theta_t^\alpha)^\intercal \mathbf B_t^\intercal \mathbf P^\alpha_{t+1} \mathbf A_t\\
  \qquad +\mathbf A_t^\intercal \mathbf P^\alpha_{t+1} \mathbf B_t \boldsymbol \theta^\alpha_t + (\boldsymbol \theta^\alpha_t)^\intercal (\mathbf R^\alpha_t+ \mathbf B^\intercal_t \mathbf  P^\alpha_{t+1} \mathbf B_t)\boldsymbol \theta^\alpha_t,\\
  \mathbf P^\alpha_{T+1}=\mathbf 0_{2 d_x \times 2 d_x},
  \end{cases}
  \end{equation} 
  where   $\boldsymbol \theta^\alpha_t=:\DIAG(\theta_t(\alpha), \bar \theta_t(\alpha))$,  $\theta_t(\alpha)$ and  $ \bar \theta_t(\alpha)$ are defined below: 
\begin{align}\label{eq:breve-f}
&  \theta_t(\alpha):= (\DeltaF)^{-1} \DeltaK, \quad  \bar \theta^\alpha_t:=(\barF)^{-1} \barK, \nonumber \\
&\DeltaF:= (1-\alpha)\Big[R_t+ \frac{\alpha}{1-\alpha}G^u_t +B_t^\intercal {{}\mathbf P_{t+1}^{\alpha}}^{\hspace{-.2cm} 1,1} B_t  \Big] \nonumber  \\
&\quad + \alpha \Big[R_t + S^u_t +(B_t+\bar B_t)^\intercal {{}\mathbf P_{t+1}^\alpha}^{\hspace{-.2cm}1,2} B_t \Big], \nonumber \\
&\barF:= (1-\alpha)\left[R_t+ S^u_t +B_t^\intercal {{}\mathbf P_{t+1}^\alpha}^{\hspace{-.2cm} 2,1} (B_t+\bar B_t)  \right]  \nonumber \\
 &+ \alpha \left[R_t + 2S^u_t + \bar R_t+ G^u_t+(B_t+\bar B_t)^\intercal {{}\mathbf P_{t+1}^\alpha}^{ \hspace{-.2cm}2,2} (B_t + \bar B_t)  \right], \nonumber \\
&\DeltaK:= -(1-\alpha)\Big[B_t^\intercal {{}\mathbf P_{t+1}^\alpha}^{\hspace{-.2cm }1,1} A_t  \Big] - \alpha \Big[(B_t + \bar B_t)^\intercal {{}\mathbf P_{t+1}^\alpha}^{\hspace{-.2cm} 1,2} A_t   \Big], \nonumber \\
&\barK:= -(1-\alpha)\left[B_t^\intercal {{}\mathbf P_{t+1}^\alpha}^{\hspace{-.2cm} 2,1} (A_t+\bar A_t)  \right] \nonumber  \\ 
&\hspace{.9cm}- \alpha \left[(B_t + \bar B_t)^\intercal {{}\mathbf P_{t+1}^\alpha}^{ \hspace{-.2cm} 2,2} (A_t+\bar A_t)   \right].
\end{align}
Also, define matrices  $P^\alpha_{1:T}$  backward in time as follows: 
\begin{equation}\label{eq:riccati-breve-m} 
\begin{cases}
\DeltaM:=G^x_t+  A_t^\intercal P^\alpha_{t+1}A_t + (\theta_t^\alpha)^\intercal B_t^\intercal P^\alpha_{t+1} A_t \\
\qquad +A_t^\intercal  P^\alpha_{t+1} B_t \theta^\alpha_t+ (\theta_t^\alpha)^\intercal  (G^u_t +  B_t^\intercal P^\alpha_{t+1}  B_t ) \theta_t^\alpha,\\
P^\alpha_{T+1}=\Zero.
\end{cases}
\end{equation}
\begin{Remark}
\emph{Equations~\eqref{eq:riccati-bar-m} and~\eqref{eq:riccati-breve-m} are  symmetric, yet non-standard, Riccati equations that always admit a unique solution if matrices  $\DeltaF$ and $\barF$ are invertible, $t \in \mathbb{N}_T$.  In the scalar case, for example, this means  $\DeltaF \neq 0$ and $\barF \neq 0$.}
\end{Remark} 
\subsection{Sufficient conditions}
We now introduce  a few  sufficient conditions for the existence of a unique sub-game perfect Nash solution, which  can be easily verified by recursively using  equations~\eqref{eq:riccati-bar-m}  and~\eqref{eq:breve-f}. 
\begin{Assumption}[Homogeneous weights]\label{ass:invertible}
At any  time $t \in \mathbb{N}_T$,  matrix $(1-\frac{1}{n}) F_t(\frac{1}{n}) + \frac{1}{n} \bar F_t(\frac{1}{n})$  is  positive definite, and matrices  $F_t(\frac{1}{n})$ and $\bar F_t(\frac{1}{n})$ are invertible.
\end{Assumption}
\begin{Assumption}[Asymptotically vanishing weights]\label{ass:invertible_infinite_pop}
Let Assumption~\ref{assump:independent_n} hold. In addition,  there exist a sufficiently large $n_0 \in \mathbb{N}$ and  a constant matrix $C$ such that  $(1-\alpha) \DeltaF + \alpha \barF$  is  positive definite and   $F^{-1}_t(\alpha), \bar F^{-1}_t(\alpha) \preceq C$,   $ \forall \alpha \in [-\beta_{\text{max}}/n_0,\beta_{\text{max}}/n_0]$.
\end{Assumption}
\begin{Assumption}[Social cost with arbitrary positive weights]\label{ass:collaborative_cost}
Let $Q_t$, $S^x_t$, $R_t$,  $S^u_t$ be zero, $t \in \mathbb{N}_T$, and  weights be positive $\alpha^i_n >0$, $\forall i \in \mathbb{N}_n$. Let also $G^x_t$ and  $\bar Q_t+G^x_t$ be positive semi-definite, and $G^u_t$ and $\bar R_t + G^u_t$ be positive definite.
\end{Assumption}
\begin{Proposition}[Decoupled standard Riccati equations  for social cost] Let Assumption~\ref{ass:collaborative_cost} hold. The equation~\eqref{eq:riccati-bar-m} reduces to  two decoupled standard Riccati equations as follows:
    \begin{equation}\label{eq:Riccati_collaborative}
    \begin{cases}
  \mathbf P_t=\DIAG(G^x_t,\bar Q_t+G^x_t)
  +\mathbf A_t^\intercal \mathbf P_{t+1} \mathbf A_t -  \mathbf A_t^\intercal \mathbf P_{t+1} \mathbf B_t \\
\qquad   \times  ( \DIAG(G^u_t, \bar R_t+ G^u_t)+ \mathbf B^\intercal_t \mathbf  P_{t+1} \mathbf B_t)^{-1} \mathbf B^\intercal \mathbf P_{t+1} \mathbf A_t,\\
  \mathbf P_{T+1}=\mathbf 0_{2 d_x \times 2 d_x},
  \end{cases}
  \end{equation} 
  where  at any $t \in \mathbb{N}_T$,  $\mathbf P^{1,1}_t=P_t$ and
\begin{equation}\label{eq:optimal_collaborative_strategy}
\Compress
\DIAG(\theta_t,\bar \theta_t):=- ( \DIAG(G^u_t, \bar R_t+ G^u_t)+ \mathbf B^\intercal_t \mathbf  P_{t+1} \mathbf B_t)^{-1} \mathbf B^\intercal \mathbf P_{t+1} \mathbf A_t.
\end{equation}
\end{Proposition}
\begin{proof}
The proof is deferred to Theorem~\ref{thm:aux}. 
\end{proof}
 \subsection{Solution of Problem~\ref{prob:aux}}
\begin{Theorem}\label{thm:aux}
The following holds  for  Problem~\ref{prob:aux}.
\begin{itemize}
\item  Let Assumption~\ref{ass:invertible} hold for  homogeneous weights.  There exists a unique  sub-game  perfect Nash equilibrium for any number of   players  such that  for any player $ i \in \mathbb{N}_n$,
\begin{equation}\label{eq:thm-optimal-mfs-u_1}
u^i_t=\theta_t^n x^i_t +  ( \bar \theta^n_t  -\theta^n_t)\bar x_t,  \quad t\in \mathbb{N}_T,
\end{equation} 
where  $\theta_t^n:=\theta(\frac{1}{n})$ and $\bar \theta^n_t:=\bar \theta_t(\frac{1}{n})$ are given by~\eqref{eq:breve-f}.

\item  Let Assumption~\ref{ass:invertible_infinite_pop} hold for asymptotically vanishing weights.  As  $n \rightarrow \infty$, there exists a unique  sub-game perfect  Nash equilibrium  such that  for any player $ i \in \mathbb{N}$,
\begin{equation}\label{eq:thm-optimal-mfs-u_2}
u^i_t=\theta^{\infty}_t x^i_t +  (\bar \theta^{\infty}_t  -\theta^{\infty}_t)\bar x_t,  \quad t\in \mathbb{N}_T, 
\end{equation} 
where  $\theta^\infty_t:=\theta_t(0)$ and $\bar \theta^\infty_t:=\bar \theta_t(0)$  are given by \eqref{eq:breve-f}.

\item  Let Assumption~\ref{ass:collaborative_cost} hold.  There exists a unique  sub-game  perfect Nash equilibrium for  arbitrary number of players  and positive weights   such that  for any player $ i \in \mathbb{N}_n$ and any weight $\alpha^i_n > 0$, 
\begin{equation}\label{eq:thm-optimal-mfs-u_3}
u^i_t=\theta_t x^i_t +  ( \bar \theta_t -\theta_t)\bar x_t,  \quad t\in \mathbb{N}_T,
\end{equation} 
where  $\theta_t:=\theta_t(\alpha^i_n)$ and $\bar \theta_t:=\bar \theta_t(\alpha^i_n)$ do not depend on the weight $\alpha^i_n$, and   are given by~\eqref{eq:optimal_collaborative_strategy}. 
\end{itemize}
\end{Theorem}

\begin{proof}
The outline of the proof is as follows. We fix the strategies of all players except the $i$-th player and write down the dynamic program associated with the best-response equation of player $i \in \mathbb{N}_n$. Subsequently, we end up with $n$-coupled dynamic programs.  By backward induction, we show that the above dynamic programs admit a quadratic value function where the associated minimizer (i.e., best-response strategy) is a unique linear state-feedback strategy across all players at each time instant.  The obtained state-feedback strategy is also the unique sub-game perfect Nash equilibrium because  it has a  unique representation in  stochastic games with additive noise~\cite[Theorem 5]{basar1976uniqueness} and \cite[Proposition 2]{van1980note}.  Our proof may be viewed as  a reformulation of  the standard  $n$-coupled matrix Riccati equations, introduced by the gauge transformation~\eqref{eq:def-breve}, in order to obtain  a low-dimensional (scale-free) recursive equation, that we refer to as non-standard Riccati equation~\eqref{eq:riccati-bar-m}. A detailed proof is presented in Appendix~\ref{sec_proof_1}.
\end{proof}

\begin{Assumption}\label{ass: mean-field_decoupled}
Suppose  Assumption~\ref{assump:independent_n} holds  and $n$ is sufficiently large. Let  matrices $\bar A_t$ and $\bar B_t$ be zero. Let also matrices $Q_t$ and $Q_t+S^x_t$ be  positive semi-definite and matrices $R_t$ and $R_t+S^u_t$ be positive definite.
\end{Assumption}
For  a  special  model  described in Assumption~\ref{ass: mean-field_decoupled},  Assumption~\ref{ass:invertible_infinite_pop} holds and  the non-standard Riccati equation~\eqref{eq:riccati-bar-m} gets decomposed into   two  standard Riccati equations, as $n \rightarrow \infty$.

\begin{Proposition}[Decoupled standard Riccati equations when $n=\infty$]\label{cor:decoupled1}
Let Assumption~\ref{ass: mean-field_decoupled} hold.  The infinite-population limit of the non-standard Riccati equation~\eqref{eq:riccati-bar-m} can be expressed  as
    \begin{equation}\label{eq:Riccati_non_collaborative}
    \begin{cases}
  \mathbf P_t=\DIAG(Q_t, Q_t+S^x_t)
  +\mathbf A_t^\intercal \mathbf P_{t+1} \mathbf A_t -  \mathbf A_t^\intercal \mathbf P_{t+1} \mathbf B_t \\
\qquad   \times  ( \DIAG(R_t, R_t+S^u_t)+ \mathbf B^\intercal_t \mathbf  P_{t+1} \mathbf B_t)^{-1} \mathbf B_t^\intercal \mathbf P_{t+1} \mathbf A_t,\\
  \mathbf P_{T+1}=\mathbf 0_{2 d_x \times 2 d_x},
  \end{cases}
  \end{equation} 
  where  at any $t \in \mathbb{N}_T$,  
\begin{multline}\label{eq:optimal_non_collaborative_strategy}
\DIAG(\theta_t(0),\bar \theta_t(0)):=- ( \DIAG(R_t, R_t+S^u_t)+ \mathbf B^\intercal_t \mathbf  P_{t+1} \mathbf B_t)^{-1}\\
\times  \mathbf B_t^\intercal \mathbf P_{t+1} \mathbf A_t.
\end{multline}
\end{Proposition}
\begin{proof}
The proof is presented in Appendix~\ref{sec_proof_2}.
\end{proof}
\begin{Remark}[\textbf{Individualized cost versus collaborative cost in the infinite-population model}]
\emph{Let the weights be homogeneous. An interesting observation is that the solution of the infinite-population game  under Assumption~\ref{ass: mean-field_decoupled}, given by~\eqref{eq:optimal_non_collaborative_strategy}, is independent of  matrices $\bar Q_t, \bar R_t, G^x_t$ and $G^u_t$ (i.e., the collaborative term). On the other hand, when the cost is equal to  the collaborative term, as described in Assumption~\ref{ass:collaborative_cost}, the solution in~\eqref{eq:optimal_collaborative_strategy} depends on  matrices $\bar Q_t, \bar R_t, G^x_t$ and $G^u_t$.  This contrast  highlights the conceptual  difference of the solution  with  non-collaborative (individualized) cost and the collaborative (common) cost. In particular, when the individualized cost is not zero,  every player in the infinite-population game decides to ignore the common welfare
and   makes their decisions solely based on their individualized cost. However, when  the players have identical  cost function, they have no reason to compete, so they collaborate to minimize the common cost (resulting in a different solution). Note that the solution of the collaborative game  is equal to  that of the competitive one if $Q_t=G^x_t$, $R_t=G^u_t$, $S^x_t=\bar Q_t$, and $S^u_t=\bar R_t$, $\forall t \in \mathbb{N}$,  according to~\eqref{eq:optimal_non_collaborative_strategy} and~\eqref{eq:optimal_collaborative_strategy}.  In such a case, the  best selfish action is equal  to  the best selfless action.  
}
\end{Remark}\
\begin{Remark}
\emph{It is to be noted that the non-standard Riccati equation~\eqref{eq:riccati-bar-m} is different from the standard coupled matrix Riccati equations. In particular, the size of $n$-coupled  standard matrix Riccati equations  increases with  $n$ while that of the non-standard one in~\eqref{eq:riccati-bar-m} is independent of $n$. 
 In addition, for homogeneous weights, one  would get $n$ exchangeable coupled matrix Riccati equations due to the exchangeablity of players; however,  our  proposed Riccati equation would  not be exchangeable because $F_t(\alpha) \neq \bar F_t(\alpha)$. This difference becomes  clearer when we consider $n=2$ players with homogeneous weights. In such a case, the classical standard two coupled Riccati equations are exchangeable, i.e., $\mathbf{P}^{1,1}_t=\mathbf{P}^{2,2}_t$, but this  is not the case  for the  non-standard Riccati equation.
}
\end{Remark}

\subsection{Solution of Problem~\ref{prob:mfs}}
\begin{Corollary}\label{thm:mfs}
The solutions presented in Theorem~\ref{thm:aux} are also the solutions of  Problem~\ref{prob:mfs}.
\end{Corollary}
 \begin{proof}
From Theorem~\ref{thm:aux}, it follows that the  solution of Problem~\ref{prob:aux} is unique and  implementable under  the DSS information structure,   implying that the solution of  Problem~\ref{prob:mfs} must coincide with that of Problem~\ref{prob:aux} because $\GMFS \subseteq \GC$. 
 \end{proof}

\begin{Remark}
\emph{Theorem~\ref{thm:aux} and Corollary~\ref{thm:mfs}  hold even if the initial states and local noises are non-exchangeable, non-Gaussian and fully correlated across players. This is an immediate consequence of the proof of Theorem~\ref{thm:aux} which shows that the sub-game perfect equilibrium does not depend on the probability distribution of the  initial states and driving noises as long as they are white (i.e., independent over time). In addition, it is to be noted that the computational complexity of the solution  does not depend on the number of players   because the dimension of the non-standard Riccati equation~\eqref{eq:riccati-bar-m} is independent of the number of players.
}
\end{Remark}
According to  the Nash equilibrium presented in  Corollary~\ref{thm:mfs},   at every stage of the game,  players  make  their decisions  based on their current local states (private information) and the deep state (public information),   irrespective of what has happened in the past, i.e.,  the players are not penalized or rewarded for their past actions.   However, this will not generally hold for NS information structure because players need to construct an error-prone  prediction of the unobserved deep state; see Section~\ref{sec:ns} for more details.

\subsection{Informationally non-unique solutions in finite- and infinite-population games}\label{sec:informationally}

\edit{For  stochastic games wherein  initial states and local noises have non-zero probability distribution over any open set in $\mathbb{R}^{d_x}$ (e.g., multivariate Gaussian probability distribution),  the sub-game perfect  Nash equilibrium presented in Theorem~\ref{thm:aux} and Corollary~\ref{thm:mfs} is also the unique Nash solution, according to~\cite[Proposition 2]{van1980note}.} In addition, it is the unique open-loop Nash equilibrium  for deterministic models,  according to~\cite[Theorem 3]{engwerda1998open}; however,  it is not necessarily the only  Nash equilibrium.  \edit{In particular,  the stochastic game admits a unique Nash solution if  Assumptions~\ref{ass:invertible}--\ref{ass:collaborative_cost} are satisfied, i.e., the unique  Nash solution coincides with the unique feedback Nash  solution of the deterministic version of the game;  see~\cite[Corollary 4 of Chapter 6]{Basarbook}. Note that the  deterministic game may   admit more than one Nash solutions under the above assumptions.} 
To demonstrate this, we borrow some remarks from~\cite{basar1976uniqueness}  to explain the non-uniqueness aspect of the solution, also called \emph{informationally non-unique} solutions. Consider a two-player game (i.e., $n=2$), where the Nash strategy  is linear in the states of players. When  there is no uncertainty (deterministic case), the state  of player 1 can  be represented by  the previous states of players 1 and 2 in   many  different ways. Each representation may  lead to a different optimization problem for  player~2, resulting in a distinct solution. On the other hand, when the model is stochastic, there is only one unique  representation of  the solution,  which is the closed-loop memoryless representation in Theorem~\ref{thm:aux} and Corollary~\ref{thm:mfs}.

The above informational non-uniqueness  feature is useful in  explaining some of the non-uniqueness results in the infinite-population game.   Due to the asymptotically negligible (vanishing) effect of individuals, one can express the infinite-population game as   a two-player game between  a generic player and an infinite-population player. When local noises are independent,  the dynamics of the deep state (i.e., the state of the infinite-population player) becomes deterministic; hence,  the deep state may be viewed as an external predictable  effect called (weighted) mean-field.  The mean-field can be  represented in different ways based on its previous states,  where each representation may lead to a different best-response equation at the generic player, resulting in non-unique solutions. Consequently,    the infinite-population game can have a unique \emph{sequential} equilibrium but admit  uncountably many  non-sequential ones, or  have no sequential equilibrium but still admit a Nash equilibrium. An immediate implication of the above discussion is that  mean-field equilibrium  does  not necessarily  coincide with the  sub-game perfect Nash equilibrium. For classical models with a tracking cost function,  however,  it is  shown that Assumption~\ref{ass:invertible_infinite_pop}  holds, implying that the classical  mean-field equilibrium coincides with the sub-game perfect Nash equilibrium; see Subsection~\ref{sec:special2}.

\section{Main Results for Problem~\ref{prob:ns}}\label{sec:ns}
So far, we have assumed that the deep state can be shared among players; however, this is not always feasible, specially when the number of players is large. In this case, the dynamic programming decomposition proposed in the previous section does not work because the optimization  problem from each player's point of view is no longer  Markovian. In other words, players do not have access to a sufficient statistic of the past history of states and actions of all players.  As a result, we  propose asymptotic equilibria under NS information structure, when the weights are  asymptotically vanishing. In particular,  we study the homogeneous weights  with a special interest  because they lead to a  finite-population (population-size-dependent) approximation that is  generally different from the infinite-population (population-size-independent) approximation. To establish the main results, we restrict attention to Assumption~\ref{assump:independent_n}  as well as the  following assumptions, unless  stated otherwise. 
\begin{Assumption}\label{assump:iid}
The initial local states $(x^1_1,\ldots,x^n_1)$ are independent random variables with  identical mean $\mu_x \in \mathbb{R}^{d_x}$ (that is bounded and  independent of $n$). In addition,  local noises $(w^1_t,\ldots,w^n_t)$, $t \in \mathbb{N}_T$,  are independent random variables.  Furthermore,  auto-covariance matrices $ \VAR(x^i_1)$ and $\VAR(w^i_t)$ are  bounded and  independent of $n$  for any $i \in \mathbb{N}_n$ and  $t \in \mathbb{N}_T$.
\end{Assumption}
\begin{Assumption}\label{ass:existence_finite}
For asymptotically vanishing weights,  there exists a sufficiently large $n_0 \in \mathbb{N}$ such that the SPNE strategy $\mathbf g^\ast_n$  exists and is unique for any $n \geq  n_0$, and takes the  form:
\begin{equation}\label{eq:suppose_strategy}
u^i_t=\theta^n_tx^i_t+(\bar \theta^n_t -\theta^n_t)\bar x_t,
\end{equation}
where $\theta^n_t$ and $\bar \theta^n_t$ are continuous and uniformly bounded in $n$ such that $\lim_{n \rightarrow \infty} \theta^n_t=\theta^\infty_t$ and $\lim_{n \rightarrow \infty} \bar  \theta^n_t=\bar \theta^\infty_t$.
\end{Assumption}
\begin{Proposition}[Sufficient conditions for Assumption~\ref{ass:existence_finite}]\label{cor:sufficient_7}
 Two sufficient conditions for Assumption~\ref{ass:existence_finite}  are as follows.  For  homogeneous weights,  Assumptions~\ref{assump:independent_n} and~\ref{ass:invertible} hold  for every $n \geq n_0$.  For the social cost function with  positive weights,  Assumptions~\ref{assump:independent_n} and~\ref{ass:collaborative_cost} hold.
\end{Proposition}
\begin{proof}
The proof of  the homogeneous case follows from Theorem~\ref{thm:aux} and the fact that  matrices in the  dynamics and cost functions are independent of $n$ under Assumption~\ref{assump:independent_n},  the weight $\alpha^i_n=1/n$ is continuous  and \edit{monotonically decreasing  (or increasing)} in $n$,  and  strategy~\eqref{eq:thm-optimal-mfs-u_1} has an infinite-population limit, according to~\eqref{eq:riccati-bar-m} and~\eqref{eq:breve-f}. Similarly, the proof of the social cost follows from  Theorem~\ref{thm:aux} and the fact that  $\mathbf g^\ast_n$ is independent of $n$.
\end{proof}
We  now present two types of prediction for the deep state, where the first one  is  a  population-size-dependent  prediction and the second one is  a population-size-independent  one.
\begin{Definition}[Population-size-dependent prediction]
For a finite-population game with $n$ players, the expected value of the deep state  under the SPNE strategy, given by~\eqref{eq:suppose_strategy}, is referred to as  the population-size-dependent  prediction, 
\begin{equation}\label{eq:deterministic_process}
z^n_{t+1}:=(A_t+\bar A_t+ (B_t +\bar B_t) \bar \theta^n_t)z^n_t,  \quad t \in \mathbb{N}_T,
\end{equation}
where  the  initial value $z^n_1:=\Exp{\bar x_1}$.
\end{Definition}
\begin{Definition}[Population-size-independent (mean-field) prediction]\label{def:PI_prediction}
For an  infinite-population game, the deep state  under the SPNE strategy, given by~\eqref{eq:suppose_strategy}, is referred to as  the population-size-independent (mean-field)  prediction,
\begin{equation}\label{eq:deterministic_process-inf}
z^\infty_{t+1}:=(A_t+\bar A_t+ (B_t +\bar B_t) \bar \theta^\infty_t)z^\infty_t,  \quad t \in \mathbb{N}_T,
\end{equation}
where  the  initial value $z^\infty_1:=\Exp{\bar x_1}$.
\end{Definition}
An immediate observation is that: (a)  $z^n_t=\bar x_t$ in deterministic game, (b) $z^n_t=z^\infty_t$ for the social cost with population-size-independent model as  $\theta^n_t = \theta^\infty_t$, and (c) $z^n_t \neq z^\infty_t$ for homogeneous weights, in general, as $\theta^n_t \neq \theta^\infty_t$.  The above predictions can be used in the SPNE strategy  $\mathbf g^\ast_n$ in~\eqref{eq:suppose_strategy}  to obtain approximate equilibria.  To measure the quality of the approximate equilibria,  the following  two gaps are defined.
\edit{
\begin{Definition}[Performance gap]\label{def:performance_gap}
The performance gap of an NS strategy $\hat{\mathbf g}$  is  defined  as the maximum  loss of performance that a player  can experience   when all players switch from the SPNE strategy  $\mathbf g^\ast_n$ to  the strategy $\hat{\mathbf g}$. In particular,   at any stage of the game  $t_0 \in \mathbb{N}_T$,
\begin{equation}
\Delta J_P(\hat{\mathbf g}, \mathbf g^\ast_n)_{t_0}:= \max_{i \in \mathbb{N}_n}|\Delta J^i_P(\hat{\mathbf g}, \mathbf g^\ast_n)_{t_0}|,
\end{equation}
where $\Delta J^i_P(\hat{\mathbf g}, \mathbf g^\ast_n)_{t_0}:= J^i_n(\hat {\mathbf g}^i,\hat {\mathbf g}^{-i})_{t_0} - J^i_n(\mathbf g^{\ast,i}_n,\mathbf g^{\ast,-i}_n)_{t_0}$.
\end{Definition}
\begin{Definition}[Rationality gap]\label{def:rationality_gap}
The rationality gap   of  an NS strategy $\hat{\mathbf g}$  is  defined  as the maximum  benefit  that a player can achieve by unilateral deviation   when all other  players switch from the SPNE strategy  $\mathbf g^\ast_n$ to  the strategy $\hat{\mathbf g}$. In particular, at any stage  $t_0 \in \mathbb{N}_T$,
\begin{equation}
\Delta J_R(\hat{\mathbf g}, \mathbf g^\ast_n)_{t_0}:=\sup_{\mathbf g^i} |\Delta J^i_R(\mathbf g^i,\hat{\mathbf g}, \mathbf g^\ast_n)_{t_0}|,
\end{equation}
where $\Delta J^i_R(\mathbf g^i, \hat{\mathbf g}, \mathbf g^\ast_n)_{t_0}:=J^i_n(\mathbf g^i, \hat{\mathbf g}^{-i})_{t_0}-J^i_n(\mathbf g^i,\hat{\mathbf g}^{\ast,-i}_n)_{t_0}$.
\end{Definition}
By definition,   the above gaps are zero when $ \hat{\mathbf g}=\mathbf g^\ast_n$. }

\subsection{Two asymptotic equilibria}
In this subsection,  we propose two asymptotic equilibria, where the first one uses the population-size-dependent prediction  and the second one uses the population-size-independent prediction for  asymptotically vanishing weights.

To distinguish between games under DSS and NS information structures,  let  $\hat x^i_t \in \mathbb{R}^{d_x}$ and $\hat u^i_t \in \mathbb{R}^{d_u}$, respectively, denote  the state and action of player $i \in \mathbb{N}_n$ at time $t \in \mathbb{N}_T$ under NS information structure. 
Let $\hat{\bar x}_t:= \sum_{i=1}^n \alpha^i_n \hat x^i_t$ and $\hat{\bar u}_t:= \sum_{i=1}^n \alpha^i_n \hat u^i_t$, $t \in \mathbb{N}_T$.  Let also the initial state of player  $i$ be  $\hat x^i_1=x^i_1$. At  time $t \in \mathbb{N}_T$, the state of player $i$ evolves according to~\eqref{eq:dynamics-mf}  as follows:
\begin{equation}\label{eq:dynamics-ns}
\hat x^i_{t+1}=A_t \hat x^i_t + B_t \hat u^i_t + \bar A_t \hat{\bar x}_t + \bar B_t \hat{\bar u}_t+w^i_t, \quad t \in \mathbb{N}_T,
\end{equation}
where the proposed NS control action of player $i $ is given by:
\begin{equation}\label{eq:action-ns}
\hat u^i_t=\theta^n_t \hat x^i_t + ( \bar \theta^n_t -\theta^n_t) z^n_t. \hspace{.5cm} \text{(SAPDE)}
\end{equation}
Alternatively,  one can use the  infinite-population solution (that does not depend on $n$)   as follows:
\begin{equation}\label{eq:action-ns-inf}
\hat u^i_t=\theta^\infty_t \hat x^i_t + ( \bar \theta^\infty_t -\theta^\infty_t) z^\infty_t. \hspace{.5cm} \text{(SWMFE)}
\end{equation}

In the sequel, strategies~\eqref{eq:action-ns} and  \eqref{eq:action-ns-inf} are denoted by  $\hat{\mathbf g}_n$ and  $\hat{\mathbf g}_\infty$, respectively,  and are  referred to as~\emph{sequential asymptotic population-size-dependent equilibrium} (SAPDE) and \emph{sequential weighted mean-field equilibrium} (SWMFE).\footnote{Note that  strategies~\eqref{eq:action-ns} and~\eqref{eq:action-ns-inf}  are  not in the form of forward-backward equations and their numerical computation  requires no  fixed-point condition, i.e.,  their complexity increases  only linearly with respect to  the horizon~$T$.} \edit{It is to be noted that the above  strategies can be implemented under NS information structure, where each player $i$ observes its local state $\hat x^i_t$ and computes the prediction signals $z^n_{1:T}$ and  $z^\infty_{1:T}$, according to \eqref{eq:deterministic_process} and~\eqref{eq:deterministic_process-inf}, respectively.}

\subsection{Asymptotic analysis}
\begin{Lemma}\label{lemma:PG_asymptotic}
Let Assumptions~\ref{assump:independent_n},~\ref{assump:iid} and~\ref{ass:existence_finite} hold for any $n \geq n_0$.  The following performance gaps converge to zero at any $t_0 \in \mathbb{N}_T$, as $n \rightarrow \infty$: 
$\lim_{n \rightarrow \infty} \Delta J_P(\hat{\mathbf g}_n,\mathbf g^\ast_n)_{t_0} =0$ and $ \lim_{n \rightarrow \infty} \Delta J_P(\hat{\mathbf g}_\infty,\mathbf g^\ast_n)_{t_0}=0. 
$
\end{Lemma}
\begin{proof} 
The proof is presented in Appendix~\ref{sec:proof_lemma:PG_asymptotic}.
\end{proof}

Let $\mathcal{C}_L[0,T]$ denote the  space of measurable functions $g_t: \mathbb{R}^{td_x} \rightarrow \mathbb{R}^{d_u}$, $t \in \mathbb{N}_T$, that are adapted to  the $\sigma$-algebra  generated  by the random variables $\{x^i_1, w^i_1,\ldots,w^i_T\}$, and  are Lipschitz  continuous  in the (local state) space and uniformly  bounded in time. A special case of such functions is the set of stable feedback strategies.

\begin{Assumption}\label{ass:rationality_gap}
Suppose one of the followings holds.
\begin{itemize}
\item[I.] (Decoupled dynamics): Let $\bar A_t$ and $\bar B_t$, $t \in \mathbb{N}_T$, be zero.

\item[II.] (Continuous unilateral deviation): The  unilateral deviations are restricted to Lipschitz continuous strategies in the local state space, i.e. $\mathbf g^i \in \mathcal{C}_L[0,T]$.
\end{itemize}
\end{Assumption}
\begin{Remark}
\emph{For coupled dynamics with  discontinuous unilateral  strategy  $\mathbf{g}^i$, there is no guarantee that the proposed strategies   are asymptotic Nash equilibria because in such a case,  $\lim_{\varepsilon \rightarrow 0} g^i_t(x^i_t+\varepsilon \mathbf 1_{d_x \times 1}) \neq g^i_t(x^i_t) $.   The  restriction to  the  continuous strategies  is a standard assumption in the literature of mean-field games and  extends the asymptotic results to coupled dynamics  with PS unilateral deviations $\mathbf g^i \in \GC$. }
\end{Remark}

\begin{Lemma}\label{lemma:rationality_gap}
Let Assumptions~\ref{assump:independent_n},~\ref{assump:iid},~\ref{ass:existence_finite} and~\ref{ass:rationality_gap} hold  for any $n \geq n_0$, $n_0 \in \mathbb{N}$. The following rationality gaps  converge to zero at any $t_0 \in \mathbb{N}_T$, as $n \rightarrow \infty$: 
$
\lim_{n \rightarrow \infty} \Delta J_R(\hat{\mathbf g}_n, \mathbf g^\ast_n)_{t_0}=0$ and $ \lim_{n \rightarrow \infty} \Delta J_R(\hat{\mathbf g}_\infty,\mathbf g^\ast_n)_{t_0}=0.
$
\end{Lemma}
\begin{proof}
The proof is presented in Appendix~\ref{sec:proof_lemma:rationality_gap}.
\end{proof}

\subsection{A non-asymptotic result}  
Define the following   relative distances at  time $t \in \mathbb{N}_T$ for any player $i \in \mathbb{N}_n$:
\begin{equation}\label{eq:relative_error}
e^i_t:=\hat x^i_t -\hat{ \bar x}_t,\quad \zeta^i_t=x^i_t-\bar x_t,  \quad  e_t:=\hat{\bar x}_t - z^n_t, \quad  \zeta_t:=\bar x_t -z^n_t.
\end{equation}
In addition, denote $\Delta \hat x^i_t=\hat x^i_t - \hat{\bar x}_t$ and $\Delta \hat u^i_t=\hat u^i_t- \hat{\bar u}_t$. 
\begin{Lemma}\label{lemma:breve_equality}
Let Assumption~\ref{ass:existence_finite}  hold. Given  any player $i \in \mathbb{N}_n$ and  any time $t \in \mathbb{N}_T$,  $\Delta x^i_t=\Delta \hat{x}^i_t$ and $\Delta u^i_t=\Delta \hat u^i_t$. This implies that the individualized relative distances are equal, i.e.,  $e^i_t=\zeta^i_t$, $\forall t \in \mathbb{N}_T$, $\forall i \in \mathbb{N}_n$. Therefore, the dynamics of the relative distances in~\eqref{eq:relative_error} between the finite-population strategies~\eqref{eq:suppose_strategy} and~\eqref{eq:action-ns} can be described by: 
 \begin{equation}\label{eq:dynamics_error_relative}
\VEC(e^i_{t+1}, e_{t+1},  \zeta_{t+1})=\tilde A^n_t \VEC(e^i_{t}, e_{t},  \zeta_{t})+ \VEC( \Delta w^i_t,  \bar w_t,  \bar w_t),
 \end{equation}
 where $\tilde A^n_t:=\DIAG(A_t+B_t \theta^n_t,A_t+\bar A_t+(B_t+\bar B_t)\theta^n_t,A_t+\bar A_t+(B_t+\bar B_t)\bar \theta^n_t)$. Similarly, the dynamics of the relative distances  between  the infinite-population limit of strategy~\eqref{eq:suppose_strategy} and strategy~\eqref{eq:action-ns-inf}, when applied to the finite-population game, can be expressed by:
  \begin{equation}\label{eq:dynamics_error_relative_inf}
\VEC(e^i_{t+1}, e_{t+1},  \zeta_{t+1})=\tilde A^\infty_t \VEC(e^i_{t}, e_{t},  \zeta_{t})+ \VEC( \Delta w^i_t,  \bar w_t,  \bar w_t).
 \end{equation}
\end{Lemma}
\begin{proof}
The proof is presented in Appendix~\ref{sec:proof_lemma:breve_equality}.
\end{proof}

Since  the SPNE strategy in Theorem~\ref{thm:aux} and  the proposed strategy~\eqref{eq:action-ns}   are  linear, one can find the exact value  of the  performance gap  in  terms of Lyapunov equations for any arbitrary  (not necessarily large) number of players,  where the primitive random variables can be correlated (not necessarily independent).  To illustrate this point,  we show that the performance gap  $\Delta J^i_{P}(\hat{\mathbf g}_n, \mathbf g^\ast_n)_{t_0}$ with homogeneous weights can be expressed  as a  quadratic function  of the relative distances.

 \begin{Lemma}\label{lemma:relative_dynamics_1}
 Let Assumption~\ref{ass:invertible}  hold. The performance gap $\Delta J^i_{P}(\hat{\mathbf g}_n, \mathbf g^\ast_n)_{t_0}$ can be described as follows:
  \begin{equation}\label{eq:cost_relative_error}
\Delta J^i_{P}(\hat{\mathbf g}_n, \mathbf g^\ast_n)_{t_0}=\Exp{\sum_{t=t_0}^T [e^i_t  \quad e_t \quad \zeta_t]^\intercal \tilde Q^n_t [e^i_t  \quad  e_t \quad \zeta_t]}.
\end{equation}
such that
\begin{align}
&{{}\tilde Q_t^n}^{1,1}:={{}\tilde Q_t^n}^{2,3}:={{}\tilde Q_t^n}^{3,2}:=\Zero,\\
&{{}\tilde Q_t^n}^{1,2 \hspace{-.3cm}}:=\mathbf  Q^{1,2}_t+ (\theta^n_t)^\intercal \mathbf R_t^{1,2} \theta^n_t,
{{}\tilde Q_t^n}^{1,3} \hspace{-.3cm}:=-\mathbf Q_t^{1,2} - (\theta^n_t)^\intercal \mathbf R_t^{1,2} \bar \theta^n_t,\\
& {{}\tilde Q_t^n}^{2,1} \hspace{-.3cm} := \mathbf  Q^{2,1}_t+ (\theta^n_t)^\intercal \mathbf R_t^{2,1} \theta^n_t, {{}\tilde Q_t^n}^{2,2} \hspace{-.2cm}:=\mathbf Q^{2,2}_t+ (\theta^n_t)^\intercal \mathbf R_t^{2,2} \theta^n_t, \\
&{{}\tilde Q_t^n}^{3,1} \hspace{-.3cm}:= -\mathbf Q^{2,1}_t- (\theta^n_t)^\intercal \mathbf R_t^{2,1} \bar \theta^n_t, 
{{}\tilde Q_t^n}^{3,3} \hspace{-.3cm}:=-\mathbf Q_t^{2,2} - (\bar \theta^n_t)^\intercal \mathbf R_t^{2,2}  \bar \theta^n_t,
\end{align}
where $\mathbf Q_t$ and $\mathbf R_t$ are given by~\eqref{eq:def-barQ}, when $a^i_n= \frac{1}{n}$. 
 \end{Lemma}
 \begin{proof}
 The proof follows from  Lemma~\ref{lemma:breve_equality} and the fact that $\Exp{e_t}=\Exp{\zeta_t}=\mathbf{0}_{d_x \times 1}, \forall t \in \mathbb{N}_T$, which is a consequence of  the linear dynamics of the relative distances and zero-mean noises.  For more details, the reader is referred to Appendix~\ref{sec:proof_lemma:relative_dynamics_1}.
 \end{proof}
Define matrices $H^{x,i}_{t}$ and $H^{w,i}_t$, $i \in \mathbb{N}_n$, $t \in \mathbb{N}_T$ such that:
\begin{equation}
\Compress
H^{x,i}_t:=\Big[ 
\begin{array}{cc}
  \VAR(x^i_t-\bar x_t) &   \Exp{(x^i_t- \bar x_t)\bar x_t^\intercal} \\
 \mathbf{1}_{2\times 2} \otimes  \Exp{(x^i_t- \bar x_t)\bar x_t^\intercal} &  \mathbf{1}_{2\times 2} \otimes \VAR(\bar x_t)
\end{array}
\Big],
\end{equation}
\begin{equation}\label{eq:hw} 
\Compress
H^{w,i}_t:=\Big[  
\begin{array}{cc}
\VAR(w^i_t-\bar w_t) &  \Exp{(w^i_t- \bar w_t)\bar w_t^\intercal}  \\
 \mathbf{1}_{2\times 2} \otimes \Exp{(w^i_t- \bar w_t)\bar w_t^\intercal} &  \mathbf{1}_{2\times 2} \otimes\VAR(\bar w_t) 
\end{array}
 \Big],
\end{equation}
where  $H^x_t:=H^{x,i}_t$  and  $H^w_t:=H^{w,i}_t$, $\forall i \in \mathbb{N}_n$,  for the case when  the  primitive random variables are  i.i.d.  
\begin{Theorem}[A non-asymptotic result for homogeneous weights]\label{thm:delta_j_correlated}
Let Assumption~\ref{ass:invertible} hold. The performance  gap of  the NS strategy~\eqref{eq:action-ns}, described in Definition~\ref{def:performance_gap},  can be computed  by a Lyapunov equation.  In particular, the following  holds given  matrices  $\tilde A^n_t$ in Lemma~\ref{lemma:breve_equality}  and $\tilde Q^n_t$ in Lemma~\ref{lemma:relative_dynamics_1}:
\begin{equation}\label{eq:delta_expansion}
\Delta J_P(\hat{\mathbf g}_n,\mathbf g^\ast_n)_{t_0}= \max_{i \in \mathbb{N}_n} \big(\TR(H^{x,i}_{t_0} \tilde M^n_1) + \sum_{t=t_0}^{T-1} \TR(H^{w,i}_t \tilde M^n_{t+1})\big), 
\end{equation} 
where
\begin{equation}\label{eq:lyapunov-finite}
\begin{cases}
\tilde M^n_{t}={{}\tilde A_t^n}^\intercal \tilde M^n_{t+1} \tilde A^n_t + \tilde Q_t^n, \quad  t_0\leq t \leq T,\\
\tilde M^n_{T+1} = \mathbf{0}_{3 d_x \times 3 d_x}.
\end{cases}
\end{equation}
\end{Theorem}
\begin{proof}
From Lemmas~\ref{lemma:breve_equality} and~\ref{lemma:relative_dynamics_1}, $\Delta J^i_P(\hat{\mathbf g}_n,\mathbf g^\ast_n)_{t_0}$,  $i \in \mathbb{N}_n$,   is  a quadratic function of the relative distances, and  the relative distances have  linear dynamics. Therefore,  this may be viewed as an uncontrolled linear quadratic system where the total expected cost can be expressed in terms of the auto-covariance matrices of  the initial states  and  local noises  (i.e., $H^{x,i}_{t_0}$ and $H^{w,i}_t$, $t_0\leq t \leq T $) and  the Lyapunov equation~\eqref{eq:lyapunov-finite}.
\end{proof}
\subsection{Solution of Problem~\ref{prob:ns}}  
It is desired now to show that  the proposed strategies, given by~\eqref{eq:action-ns} and~\eqref{eq:action-ns-inf},  are  solutions for Problem~\ref{prob:ns}.

\begin{Theorem}\label{thm:NS-optimality}
Let Assumptions~\ref{ass:invertible_infinite_pop}  and~\ref{assump:iid} hold for the special case of  homogeneous weights for  any  $ n \geq n_0$, $n_0 \in \mathbb{N}$. The  following holds for Problem~\ref{prob:ns}.
\begin{enumerate}
\item The NS strategy~\eqref{eq:action-ns} is  a sequential asymptotic reciprocal equilibrium such that
\begin{equation}\label{eq:epsilon_finite}
\varepsilon_{t_0} (n):= \Delta J_P(\hat{\mathbf g}_n,\mathbf g^\ast_n)_{t_0} \in \mathcal{O}(\frac{1}{n}), \quad t_0 \in \mathbb{N}_T.
\end{equation}
\item Let  also Assumption~\ref{ass:rationality_gap}  hold. Strategy~\eqref{eq:action-ns} is a sequential asymptotic Nash equilibrium  such that
$
\bar \varepsilon_{t_0}(n):= \Delta J_P(\hat{\mathbf g}_n,\mathbf g^\ast_n)_{t_0}+  \Delta J_R(\hat{\mathbf g}_n,\mathbf g^\ast_n)_{t_0}$, $t_0 \in \mathbb{N}_T$, 
where $\lim_{n \rightarrow \infty} \bar \varepsilon_{t_0}(n)=0$.
\end{enumerate} 

\end{Theorem}
\begin{proof}
The proof is presented in Appendix~\ref{sec:proof_thm:NS-optimality}.
\end{proof} 

\begin{Theorem}\label{cor:ns-infinite}
Let  Assumptions~\ref{ass:invertible_infinite_pop} and~\ref{assump:iid} hold for  any  $ n \geq n_0$, $n_0 \in \mathbb{N}$.   The following holds for Problem~\ref{prob:ns}.
\begin{enumerate}
\item For the special case of homogeneous weights,  the NS strategy~\eqref{eq:action-ns-inf} is a sequential asymptotic reciprocal equilibrium   such that 
$
\varepsilon_{t_0} (n):= \Delta J_P(\hat{\mathbf g}_\infty,\mathbf g^\ast_n)_{t_0}, t_0 \in \mathbb{N}_T,
$
where $\lim_{n \rightarrow \infty}  \varepsilon_{t_0}(n)=0$.
\item Let also  Assumption~\ref{ass:rationality_gap} hold. For asymptotically vanishing weights, strategy~\eqref{eq:action-ns-inf} is a sequential asymptotic Nash equilibrium, where $\lim_{n \rightarrow \infty} \bar \varepsilon_{t_0}(n)=0$, $t_0 \in \mathbb{N}_T$.
\end{enumerate} 
\end{Theorem}
\begin{proof}
The proof is presented in Appendix~\ref{sec:proof_cor:ns-infinite}.
\end{proof}
It is shown in  Theorem~\ref{cor:ns-infinite}  that  SWMFE in~\eqref{eq:action-ns-inf} is   SANE, irrespective of  whether or not  the finite-population game has a solution. This independence from the finite-population solution comes from the fact  that the infinite-population solution is not necessarily  the limit of the finite-population one.  However, if  the finite-population solution exists according to Assumption~\ref{ass:existence_finite}, one can  provide a unified form for both SAPDE and SWMFE, as described   below.
\begin{Theorem}
Let Assumptions~\ref{assump:independent_n},~\ref{assump:iid} and~\ref{ass:existence_finite} hold  for  any  $ n \geq n_0$, $n_0 \in \mathbb{N}$. The  following holds for NS strategies~\eqref{eq:action-ns} and \eqref{eq:action-ns-inf}.
\begin{enumerate}
\item They are sequential asymptotic reciprocal equilibria such that at any time $t_0 \in \mathbb{N}_T$: 
\begin{align}
\varepsilon^{\text{SAPDE}}_{t_0} (n)&:= \Delta J_P(\hat{\mathbf g}_n,\mathbf g^\ast_n)_{t_0},\\
\varepsilon^{\text{SWMFE}}_{t_0} (n)&:= \Delta J_P(\hat{\mathbf g}_\infty,\mathbf g^\ast_n)_{t_0},
\end{align}
where $\lim_{n \rightarrow \infty}  \varepsilon^{\text{SAPDE}}_{t_0}(n)=0$, $\lim_{n \rightarrow \infty}  \varepsilon^{\text{SWMFE}}_{t_0}(n)=0$.
\item Let also Assumption~\ref{ass:rationality_gap}  hold. They are  sequential asymptotic Nash equilibria  such that  at  any time $t_0 \in \mathbb{N}_T$,
\begin{align}
\bar \varepsilon^{\text{SAPDE}}_{t_0}(n)&:= \Delta J_P(\hat{\mathbf g}_n,\mathbf g^\ast_n)_{t_0}+  \Delta J_R(\hat{\mathbf g}_n,\mathbf g^\ast_n)_{t_0},\\
\bar \varepsilon^{\text{SWMFE}}_{t_0}(n)&:= \Delta J_P(\hat{\mathbf g}_\infty,\mathbf g^\ast_n)_{t_0}+  \Delta J_R(\hat{\mathbf g}_\infty,\mathbf g^\ast_n)_{t_0},\\
\end{align}
where $\lim_{n \rightarrow \infty} \bar \varepsilon^{\text{SAPDE}}_{t_0}(n)=0$, $\lim_{n \rightarrow \infty} \bar \varepsilon^{\text{SWMFE}}_{t_0}(n)=0$.
\end{enumerate} 
\end{Theorem}
\begin{proof}
The proof follows from Definitions~\ref{def:pDSS},~\ref{def:pDSS2},~\ref{def:performance_gap},~\ref{def:rationality_gap},  Lemmas~\ref{lemma:PG_asymptotic},~\ref{lemma:rationality_gap}, and  a triangle inequality  used in Theorem~\ref{thm:NS-optimality}.
\end{proof}
\begin{Remark}[Population-size-dependent versus population-size-independent asymptotic  equilibria]\label{remark:differences-NS}
\emph{
Although NS strategies~\eqref{eq:action-ns} and~\eqref{eq:action-ns-inf} converge to the same  (unique) Nash solution as $n \rightarrow \infty$ under Assumption~\ref{ass:existence_finite}, they  have subtle differences.
In what follows, we compare them from three different angles.
\begin{enumerate}
\item Existence condition: The existence of the finite-population NS strategy~\eqref{eq:action-ns} depends on the finite-population solution while that of the infinite-population NS strategy~\eqref{eq:action-ns-inf} depends on the infinite-population one. See two counterexamples in~Section~\ref{sec:special3} demonstrating that the finite-population game may  admit a solution  while the infinite-population  may  not and vice versa. 
\item Performance and rationality  gaps: Since strategy~\eqref{eq:action-ns} takes the number of players into account, it potentially leads to  smaller performance and rationality gaps.   For example, when the  game  is deterministic and $\bar x_1$ is known,   strategy~\eqref{eq:action-ns}  becomes  the  SPNE (where performance and rationality gaps are zero) while strategy~\eqref{eq:action-ns-inf} remains  an approximate Nash  solution, in  general. See a numerical example  in Section~\ref{sec:numerical}.
\item Complexity of analysis: The advantage of  strategy~\eqref{eq:action-ns-inf} over strategy~\eqref{eq:action-ns} is that its analysis is based  on  the infinite-population model,  which is simpler than that of the finite-population one.
\end{enumerate}}
\end{Remark}

\section{Infinite Horizon}\label{sec:infinite}

In this section, we  extend our main results  to the  infinite horizon. To this end,  it is assumed that the model described in~Section~\ref{sec:problem}  is time-homogeneous, and  hence, subscript~$t$ is dropped from the model parameters.  Given  discount factor  $\gamma \in (0,1)$, define  the following discounted cost  for player $i \in \mathbb{N}_n$:
\begin{equation}\label{eq:cost_infinite}
J^{i,\gamma}_n(\mathbf g) =(1-\gamma) \Exp{ \sum_{t=1}^\infty \gamma^{t-1} c^i(\mathbf x_t,\mathbf u_t)},
\end{equation}
where the per-step cost is given by~\eqref{eq:cost-mfs}. Define the  algebraic counterpart of the non-standard equation~\eqref{eq:riccati-bar-m} for $\alpha \in \mathcal{A}_n$  as 
 \begin{align}\label{eq:riccati-bar-m-algebraic}
  \mathbf P^\alpha&= \mathbf  Q^\alpha+ \gamma\big( \mathbf A^\intercal \mathbf P^\alpha \mathbf A+    (\boldsymbol \theta^\alpha)^\intercal \mathbf B^\intercal \mathbf P^\alpha \mathbf A \nonumber \\
  &\quad +\mathbf A^\intercal \mathbf P^\alpha \mathbf B \boldsymbol \theta^\alpha \big)+ (\boldsymbol \theta^\alpha)^\intercal (\mathbf R^\alpha+\gamma  \mathbf B^\intercal \mathbf  P^\alpha \mathbf B)\boldsymbol \theta^\alpha,
  \end{align} 
  where   $\boldsymbol \theta^\alpha=:\DIAG(\theta(\alpha), \bar \theta(\alpha))$,  $\theta(\alpha)$ and  $ \bar \theta(\alpha)$ are defined as 
\begin{align}\label{eq:breve-f-stationary}
&  \theta(\alpha):= (F(\alpha))^{-1} K(\alpha), \quad  \bar \theta^\alpha:=(\bar F(\alpha))^{-1} \bar K(\alpha), \nonumber \\
&F(\alpha):= (1-\alpha)\Big[R+ \frac{\alpha}{1-\alpha}G^u +\gamma B^\intercal {{}\mathbf P^{\alpha}}^{ 1,1} B  \Big] \nonumber  \\
&\quad + \alpha \Big[R + S^u +\gamma (B+\bar B)^\intercal {{}\mathbf P^\alpha}^{1,2} B \Big], \nonumber \\
&\bar F(\alpha):= (1-\alpha)\left[R+ S^u +\gamma B^\intercal {{}\mathbf P^\alpha}^{ 2,1} (B+\bar B)  \right]  \nonumber \\
 &+ \alpha \left[R + 2S^u+ \bar R+ G^u+\gamma (B+\bar B)^\intercal {{}\mathbf P^\alpha}^{ 2,2} (B + \bar B)  \right], \nonumber \\
&K(\alpha):= -(1-\alpha)\gamma\Big[B^\intercal {{}\mathbf P^\alpha}^{1,1} A  \Big] - \alpha \gamma \Big[(B + \bar B)^\intercal {{}\mathbf P^\alpha}^{ 1,2} A   \Big], \nonumber \\
&\bar K(\alpha):= -(1-\alpha)\gamma \left[B^\intercal {{}\mathbf P^\alpha}^{ 2,1} (A+\bar A)  \right] \nonumber  \\ 
&\hspace{.9cm}- \alpha \gamma \left[(B + \bar B)^\intercal {{}\mathbf P^\alpha}^{  2,2} (A+\bar A)   \right].
\end{align}

Also, define 
\begin{align}\label{eq:riccati-breve-m_algebraic} 
P^\alpha &:=G^x+ \gamma\big( A^\intercal P^\alpha A + (\theta^\alpha)^\intercal B^\intercal P^\alpha A \nonumber \\
&\quad +A^\intercal  P^\alpha B \theta^\alpha\big)+ (\theta^\alpha)^\intercal  (G^u + \gamma B^\intercal P^\alpha  B ) \theta^\alpha.
\end{align}

In general, it is not  straightforward  to derive  conditions under which  the solutions of equations~\eqref{eq:riccati-bar-m}  and~\eqref{eq:riccati-breve-m} converge to  a   bounded limit as $T \rightarrow \infty$. In Assumption~\ref{ass:invertible-algebraic},  we  consider  one such condition that induces this property, and  in Assumptions~\ref{ass:contractive}--\ref{ass:decoupled2}, we provide three sufficient verifiable conditions under which Assumption~\ref{ass:invertible-algebraic} is satisfied. 

\begin{Assumption}\label{ass:invertible-algebraic}
Finite-horizon solutions of the backward ordinary difference equations~\eqref{eq:riccati-bar-m}  and~\eqref{eq:riccati-breve-m} admit infinite-horizon limits as $T$ goes to infinity.  More precisely, for any $\alpha \in \mathcal{A}_n$,  and  fixed $t \in \mathbb{N}_T$, $
\lim_{T\rightarrow \infty} \gamma^{-T+t}  \mathbf P^\alpha_{T-t+1}=:\mathbf P^\alpha$  and $ 
\lim_{T \rightarrow \infty} \gamma^{-T+t}  P^\alpha_{T-t+1}=:P^\alpha$,
where  $ \mathbf P^\alpha$ and $P^\alpha$ are the solutions of the algebraic  equations~\eqref{eq:riccati-bar-m-algebraic}  and~\eqref{eq:riccati-breve-m_algebraic}, respectively.
\end{Assumption}

Let  $L_1$ denote the mapping  from  $\mathbf P^\alpha$ to $\boldsymbol \theta^\alpha$ displayed in~\eqref{eq:breve-f-stationary} (where $\boldsymbol \theta^\alpha=L_1(\mathbf P^\alpha)$),  and $L_2$ denote the mapping from $\boldsymbol \theta^\alpha$ to $\mathbf P^\alpha$ expressed in~\eqref{eq:riccati-bar-m-algebraic} (where $\mathbf P^\alpha=L_2(\boldsymbol \theta^\alpha)$). As a result, $\mathbf P^\alpha=L_2(L_1(\mathbf P^\alpha))$ is a fixed-point equation, that can be solved and studied further by various  fixed-point methods. 
\begin{Assumption}[Contractive solution]\label{ass:contractive}
Let the mapping $L_2(L_1(\boldsymbol \cdot))$ be a contraction, which implies that equations~\eqref{eq:riccati-bar-m-algebraic} and~\eqref{eq:breve-f-stationary} admit a unique   solution.\footnote{By using  a  change of variable proposed   in the proof of Theorem~\ref{thm:mfs-inf_inf}, one can show that the  finite-horizon (time-varying) equations~\eqref{eq:riccati-bar-m} and~\eqref{eq:riccati-breve-m} can be rewritten as two finite-horizon (time-invariant) equations in the form of~\eqref{eq:riccati-bar-m-algebraic} and~\eqref{eq:riccati-breve-m_algebraic}, respectively. Therefore, the time recursion of the latter equations admit a unique solution due to the contraction assumption, which also implies that the infinite-horizon limit exists.} 
\end{Assumption}
For the two special models presented in Assumptions~\ref{ass: mean-field_decoupled} and~\ref{ass:collaborative_cost}, one can use the standard  schur method~\cite{laub1979schur} to derive the following sufficient conditions.\footnote{Given $\gamma \in (0,1)$,  $(\sqrt \gamma A,\sqrt \gamma B)$ is stablizable and  $(\sqrt \gamma A, Q^{1/2})$ is detectable  if $(A,B)$ is stablizable and  $(A, Q^{1/2})$ is detectable, respectively.}
\begin{Assumption}[Invariance subspace solution with  asymptotically vanishing weights]\label{ass:decoupled1}
\emph{Let Assumption~\ref{ass: mean-field_decoupled} hold. Let also $(A,B)$ be stabilizable, and $(A,Q^{1/2})$ and $(A,(Q+S^x)^{1/2})$ be  detectable.}
\end{Assumption}

\begin{Assumption}[Invariance subspace solution for social cost with positive weights]\label{ass:decoupled2}
\emph{Let Assumption~\ref{ass:collaborative_cost} hold. Let also $(A,B)$ and $(A+\bar A, B+\bar B)$ be stabilizable, and $(A,(G^x)^{1/2})$ and $(A+\bar A,(\bar Q+G^x)^{1/2})$ be  detectable.}
\end{Assumption}

\begin{Proposition}[Decoupled standard algebraic Riccati equations for infinite population]\label{cor:decoupled1_inf}
Let Assumption~\ref{ass:decoupled1} hold.  The infinite-population solution of the non-standard Riccati equation~\eqref{eq:riccati-bar-m-algebraic} can be expressed  by the  following standard algebraic Riccati equation:
$
  \mathbf P=\DIAG(Q, Q+S^x)
  + \gamma \mathbf A^\intercal \mathbf P \mathbf A - \gamma^2 \mathbf A^\intercal \mathbf P \mathbf B 
  ( \DIAG(R, R+S^u)+\gamma \mathbf B^\intercal \mathbf  P \mathbf B)^{-1} \mathbf B^\intercal \mathbf P \mathbf A
$,
  where  
\begin{equation}
\Compress
\DIAG(\theta(0),\bar \theta(0)):=- \gamma ( \DIAG(R, R+S^u)+ \gamma \mathbf B^\intercal \mathbf  P \mathbf B)^{-1} \mathbf B^\intercal \mathbf P \mathbf A.
\end{equation}
\end{Proposition}

\begin{Proposition}[Decoupled standard algebraic Riccati equations for social cost]\label{cor:decoupled2_inf}
Let Assumption~\ref{ass:decoupled2} hold.  The social cost solution  can be expressed  by  the  following standard algebraic Riccati equation:
$
  \mathbf P=\DIAG(G^x, \bar Q+G^x)
  + \gamma \mathbf A^\intercal \mathbf P \mathbf A - \gamma^2 \mathbf A^\intercal \mathbf P \mathbf B 
  ( \DIAG(G^u, \bar R+G^u)+\gamma \mathbf B^\intercal \mathbf  P \mathbf B)^{-1} \mathbf B^\intercal \mathbf P \mathbf A
$, where 
\begin{equation}\label{eq:optimal_collaborative_strategy_inf}
\DIAG(\theta,\bar \theta):=- \gamma ( \DIAG(G^u, \bar R+G^u)+ \gamma \mathbf B^\intercal \mathbf  P \mathbf B)^{-1} \mathbf B^\intercal \mathbf P \mathbf A.
\end{equation}
\end{Proposition}
\subsection{Solutions of Problems~\ref{prob:aux} and~\ref{prob:mfs}}

\begin{Theorem}\label{thm:mfs-inf_inf}
 The following holds for  Problem~\ref{prob:aux} with the  infinite-horizon  cost  function~\eqref{eq:cost_infinite}.
\begin{itemize}
\item  Let Assumptions~\ref{ass:invertible} and~\ref{ass:invertible-algebraic} hold  for  homogeneous weights.  There exists a stationary  sub-game  perfect Nash equilibrium  for any number of   players  such that  for any player $ i \in \mathbb{N}_n$,
\begin{equation}\label{eq:thm-optimal-mfs-u_1_inf}
u^i_t=\theta^n x^i_t +  ( \bar \theta^n-\theta^n)\bar x_t, \quad t\in \mathbb{N},
\end{equation} 
where  $\theta^n:=\theta(\frac{1}{n})$ and $\bar \theta^n:=\bar \theta(\frac{1}{n})$ are given by~\eqref{eq:breve-f-stationary}.

\item  Let Assumption~\ref{ass:invertible_infinite_pop}  and~\ref{ass:invertible-algebraic} hold  for asymptotically vanishing weights.  As $n \rightarrow \infty$,  there exists a stationary sub-game perfect  Nash equilibrium    such that  for any player $ i \in \mathbb{N}$,
\begin{equation}\label{eq:thm-optimal-mfs-u_2_inf}
u^i_t=\theta^{\infty} x^i_t +  (\bar \theta^{\infty}  -\theta^{\infty})\bar x_t, \quad t\in \mathbb{N},
\end{equation} 
where  $\theta^\infty:=\theta(0)$ and $\bar \theta^\infty:=\bar \theta(0)$  are given by \eqref{eq:breve-f-stationary}.

\item  Let Assumption~\ref{ass:decoupled2} hold.  There exists a stationary  sub-game  perfect Nash equilibrium  for arbitrary number of players  and positive weights   such that  for any player $ i \in \mathbb{N}_n$ and any weight $\alpha^i_n > 0$,
\begin{equation}\label{eq:thm-optimal-mfs-u_3_inf}
u^i_t=\theta x^i_t +  ( \bar \theta -\theta)\bar x_t, \quad t\in \mathbb{N},
\end{equation} 
where  $\theta:=\theta(\alpha^i_n)$ and $\bar \theta:=\bar \theta(\alpha^i_n)$ do not depend on the weight $\alpha^i_n$, and   are given by~\eqref{eq:optimal_collaborative_strategy_inf}. 
\end{itemize}
\end{Theorem}
\begin{proof}
The proof is presented in Appendix~\ref{sec:proof_thm:mfs-inf_inf}.
\end{proof}

\subsection{Solutions  of Problem~\ref{prob:ns}}\label{sec:Finite-population solution}
Similar to the finite-horizon NS strategies~\eqref{eq:action-ns} and~\eqref{eq:action-ns-inf}, we define the following  strategies for every $ t \in \mathbb{N}$:
\begin{equation}\label{eq:action-ns-infinite_horizon}
\hat u^i_t=\theta^n \hat x^i_t + ( \bar \theta^n -\theta^n) z^n_t, \hspace{.5cm} \text{(SAPDE)}
\end{equation}
and
\begin{equation}\label{eq:action-ns-infinite_horizon_inf}
\hat u^i_t=\theta^\infty \hat x^i_t + ( \bar \theta^\infty -\theta^\infty) z^\infty_t. \hspace{.5cm} \text{(SWMFE)}
\end{equation}

In general,  NS strategies~\eqref{eq:action-ns-infinite_horizon} and~\eqref{eq:action-ns-infinite_horizon_inf} can destabilize the game due to  the error propagation  induced by the imperfection  of NS information structure. 
To see this, one can construct a simple counterexample as follows.
\begin{Counterexample}
\emph{ Consider a game with homogeneous weights and   coupled dynamics. In this case, the game under  strategy~\eqref{eq:action-ns-infinite_horizon}  is unstable if matrix $\tilde A^n_t$ in Lemma~\ref{lemma:breve_equality} is not stablizable (Hurwitz), where the relative distances in~\eqref{eq:dynamics_error_relative} grow unboundedly, despite the fact that  the PS strategy~\eqref{eq:thm-optimal-mfs-u_1_inf} is  bounded under Assumptions~\ref{ass:invertible} and~\ref{ass:invertible-algebraic}.}
\end{Counterexample}
To overcome this hurdle,  an additional assumption is imposed on the model to ensure that  the games under strategies~\eqref{eq:action-ns-infinite_horizon} and~\eqref{eq:action-ns-infinite_horizon_inf} are stable.

\begin{Assumption}\label{ass:hurwitz}
\edit{Let  Assumptions~\ref{ass:invertible_infinite_pop} and~\ref{ass:invertible-algebraic} hold,  and in addition,  matrix $\tilde A^n$ in Lemma~\ref{lemma:breve_equality} is Hurwitz for any $n \geq n_0$, where $\tilde A^n$ is defined as  $\tilde A^n_t$ where subscript $t$ is omitted.
}
\end{Assumption}

\begin{Proposition}\label{proposition:hurwitz}
Let Assumption~\ref{ass:hurwitz} hold. Then, for any $ n \geq n_0$, $|J^{i,\gamma}_n(\hat{\mathbf g}^{i}_n,\hat{\mathbf g}^{-i}_n)| <\infty$ and $|J^{i,\gamma}_n(\hat{\mathbf g}^{i}_\infty,\hat{\mathbf g}^{-i}_\infty)| <\infty$.
\end{Proposition}

\begin{proof}
The proof is presented in Appendix~\ref{sec:proof_proposition:hurwitz}.
\end{proof}

\begin{Remark}[A sufficient condition for Assumption~\ref{ass:hurwitz}]
\emph{Let  Assumptions~\ref{ass:invertible_infinite_pop} and~\ref{ass:invertible-algebraic} hold for any $n \geq n_0$. When  players are dynamically decoupled (i.e., $\bar A_t$ and $\bar B_t$ are zero, $\forall t \in \mathbb{N}$), matrix $\tilde A^n$ in Lemma~\ref{lemma:breve_equality}  becomes Hurwitz. In particular, stability of  diagonal matrix  $\tilde A^n$ reduces to the stability of matrices $A+B\theta^n$ and $A+B \bar \theta^n$,  which represent the dynamics of the finite- and infinite-population  states under  SPNE strategies (that  have infinite-horizon limits from Assumption~\ref{ass:invertible-algebraic}).}
\end{Remark}

\begin{Theorem}\label{thm:ns-inf}
Let Assumptions~\ref{ass:invertible_infinite_pop},~\ref{assump:iid},~\ref{ass:invertible-algebraic} and~\ref{ass:hurwitz} hold for the special case of homogeneous weights for any $n \geq n_0$, $n_0 \in \mathbb{N}$.  The following holds for Problem~\ref{prob:ns} with the infinite-horizon cost function.  
\begin{itemize}
\item The NS strategy~\eqref{eq:action-ns-infinite_horizon} is a sequential asymptotic reciprocal equilibrium such that
\begin{align}\label{eq:delta_algebraic_time_mfs}
\varepsilon_{t_0}(n)&:= \Delta J_P(\hat{\mathbf g}_n, \mathbf g^\ast_n)= \gamma^{t_0-1}\big( (1-\gamma)\TR(H^x_{t_0} \tilde M^n) \nonumber \\
&+ \gamma \TR(H^w \tilde M^n)\big) \in \mathcal{O}(\frac{1}{n}),\quad t_0 \in \mathbb{N},
\end{align} 
where $\tilde M^n$ is the solution of the algebraic Lyapunov equation: $
\tilde M^n=\gamma (\tilde A^n)^\intercal \tilde M^n \tilde A^n + \tilde Q^n$.

\item Let also Assumption~\ref{ass:rationality_gap} hold. Strategy~\eqref{eq:action-ns-infinite_horizon_inf} is a sequential asymptotic Nash equilibrium such that
$
\varepsilon_{t_0}(n)=\Delta J_P(\hat{\mathbf g}_n, \mathbf g^\ast_n)+\Delta J_R(\hat{\mathbf g}_n, \mathbf g^\ast_n), \quad t_0 \in \mathbb{N},
$
where $ \lim_{n \rightarrow \infty} \bar \varepsilon_{t_0}(n)=0$.
\end{itemize}
\end{Theorem}
\begin{proof}
The proof  follows along the same lines of the proofs of  Theorems~\ref{thm:delta_j_correlated} and~\ref{thm:NS-optimality}, where the infinite-horizon  strategy in~\eqref{eq:thm-optimal-mfs-u_1_inf} under Assumption~\ref{ass:invertible-algebraic} is the limit of the finite-horizon one in~\eqref{eq:thm-optimal-mfs-u_1}, as  $T \rightarrow \infty$,  upon noting that strategy~\eqref{eq:action-ns-infinite_horizon} is stable according to Proposition~\ref{proposition:hurwitz} under Assumption~\ref{ass:hurwitz}.
\end{proof}

\begin{Theorem}\label{cor:ns-infinite-inf}
Let Assumptions~\ref{ass:invertible_infinite_pop},~\ref{assump:iid},~\ref{ass:invertible-algebraic} and~\ref{ass:hurwitz} hold for any $n \geq n_0$, $n_0 \in \mathbb{N}$.  The following holds for Problem~\ref{prob:ns} with the infinite-horizon cost function.  
\begin{itemize}
\item For the special case of homogeneous weights,  the NS strategy~\eqref{eq:action-ns-infinite_horizon_inf} is a sequential asymptotic reciprocal equilibrium such that
$
\varepsilon_{t_0}(n):= \Delta J_P(\hat{\mathbf g}_\infty, \mathbf g^\ast_n),\quad t_0 \in \mathbb{N},
$
where $\lim_{n \rightarrow \infty} \varepsilon_{t_0}(n)=0$.
\item Let also Assumption~\ref{ass:rationality_gap} hold. For asymptotically vanishing weights, strategy~\eqref{eq:action-ns-infinite_horizon_inf} is a sequential asymptotic Nash equilibrium,
where $ \lim_{n \rightarrow \infty} \bar \varepsilon_{t_0}(n)=0$.
\end{itemize}
\end{Theorem}
\begin{proof}
The proof  follows along the same lines of the proof of  Theorem~\ref{cor:ns-infinite}, where the infinite-horizon  strategy in~\eqref{eq:thm-optimal-mfs-u_2_inf} under Assumption~\ref{ass:invertible-algebraic} is the limit of the finite-horizon one in~\eqref{eq:thm-optimal-mfs-u_2}, as  $T \rightarrow \infty$,  upon noting that strategy~\eqref{eq:action-ns-infinite_horizon_inf} is stable according to Proposition~\ref{proposition:hurwitz} under Assumption~\ref{ass:hurwitz}.
\end{proof}

\section{Relation to mean-field models}\label{sec:special}
In this section, we  connect our  homogeneous weights results  to three different  infinite-population  exchangeable games, called mean-field games, mean-field-type games and mean-field teams. In general, there is no explicit relationship   between the above games as they consider different solution concepts. In particular, mean-field game seeks a Nash solution and mean-field team  a globally team-optimal one while   mean-field-type game (also known as McKean-Vlasov type)  looks for  neither Nash nor globally team-optimal, in general.\footnote{In particular, it is  demonstrated in~\cite{Bensoussan2016} that the solution of LQ mean-field-type game   does not necessarily coincide with that of the LQ mean-field game (which is a Nash equilibrium). In addition, it is shown in~\cite{Jalal2019risk} that the (team-optimal) infinite-population solution of the risk-sensitive LQ mean-field teams is equivalent to its risk-neutral one, which is not the case for risk-sensitive LQ mean-field-type game~\cite{moon2019linear}. These findings  suggest that the solution concept  of mean-field-type game is  generally different from those of mean-field games and mean-field teams.}  In fact, mean-field-type game  has a single agent  that strives  to hedge against uncertainty by including in its dynamics and cost some components depending on the mean-field of a \emph{virtual population} of infinitely many identical copies of itself.  Remarkably,  it turns out that  LQ model is one of the very few special models  in which all the above games happen to have some overlaps, which can  now  be identified due to a unified explicit characterization of the solution presented in this paper.
\subsection{Mean-field-type games}
 In LQ mean-field-type game,  it is well known  that the solution is  linear in the local state and the expectation of the state of  a generic player,  and that   the corresponding gains are obtained by solving  standard Riccati equations~\eqref{eq:Riccati_collaborative}, where the proof technique is based on a completion-of-square method~\cite{yong2013linear,elliott2013discrete}.  Hence,  LQ mean-field-type game  resembles a collaborative game with social cost function,  where the weights are homogeneous,  information structure is DSS and  the number of players is infinite. 
 
It is to be noted that mean-field-type model is not  well defined for finite $n$ and/or deterministic game. For example, in deterministic game, $x_t=\Exp{x_t}$, which reduces the mean-field-type game to a single-agent (unconstrained) optimization problem.  In addition, it is conceptually challenging to  distinguish between DSS and NS information structures in mean-field-type model, mainly because it is not a multi-player game.  For instance,  the role of an extra stability condition introduced in Assumption~\ref{ass:hurwitz} is overlooked in the literature of mean-field-type games, for the case  when the dynamics are coupled and information stricture is NS.
\subsection{Mean-field games}\label{sec:special2}

For  non-collaborative games with decoupled player dynamics and a tracking cost formulation,   approximate Nash solutions for   discounted and time-averaged cost functions   are proposed  in~\cite{huang2007large} and~\cite{li2008asymptotically}, respectively.  Similarly for  collaborative games  with decoupled dynamics, an approximate  Nash solution  is determined   by employing  mean-field game and person-by-person approaches~\cite{Huang2012social}.   The solutions are presented  in  terms of coupled forward and backward ordinary differential equations, and are solved by some fixed-point approaches. A drawback of these fixed-point approaches is  that they can admit more than one solutions because they identify Nash solutions (rather than sub-game perfect Nash equilibrium). To  identify some of these extra (non-unique) Nash solutions, one may take into account the following  two phenomena: 1) informationally non-unique solutions  in the infinite-population game (see~Subsection~\ref{sec:informationally})  and  2)  the fact that  infinite-population solution is not necessarily the limit of the (exact) finite-population solution.

However,  it can be shown that the continuous-time counterparts of  the proposed standard   Riccati equations in~\eqref{eq:Riccati_non_collaborative}  can solve  the coupled  ordinary differential equations  and simplify the  existence and uniqueness conditions in~\cite{huang2007large,li2008asymptotically,Huang2012social}. As a result, one can conclude that the classical mean-field game (Nash)  solution coincides with the sub-game perfect Nash equilibrium of the infinite-population model.  A similar relationship  between forward-backward equations of mean-field games and Riccati equations has recently  been established in~\cite{Huang2019}  in terms of one symmetric and one non-symmetric Riccati equation,  which may be viewed  as a special case of our non-standard Riccati equation~\eqref{eq:riccati-bar-m}, where  the information structure is PS, weights are homogeneous,  matrices $\bar B_t$,  $S^u_t$, $\bar R_t$, $G^x_t$ and $G^u_t$ are zero, $n=\infty$, and $Q_t+2S^x_t+\bar Q_t$ is  a positive definite matrix taking the following special form: $(I-\Gamma_t)^\intercal Q_t (I-\Gamma_t)$, given a matrix $\Gamma_t$. In contrast  to our proposed  proof technique  in Theorem~\ref{thm:aux}, the methodology used in~\cite{Huang2019}  is  more complicated and uses a parametric representation of $n$ coupled Riccati equations to show that in the infinite-population case, the resultant set of coupled Riccati equations converge to two (initially guessed) Riccati equations.

Note that our proposed \emph{weighted extended} mean-field game solution is more general than the classical one, as we consider weighted average (rather than  the unweighed one) with both state and action coupling (rather than only state coupling) and  with a more general cost function (rather than only the special case of tracking cost structure).   In addition,  similar to mean-field-type game, we prove that  for NS information structure with coupled dynamics, an additional stability condition (similar to Assumption~\ref{ass:hurwitz}), apart from the boundedness condition of Riccati equations in Assumption~\ref{ass:invertible-algebraic}, is required to ensure that the propagation of the error associated with the imperfection of NS information structure remains bounded in the finite-population game.  The existence of such assumption is often overlooked in the literature of mean-field games.

\subsection{Mean-field teams}
Mean-field teams are initially introduced in~\cite{arabneydi2016new}, showcased in~\cite{JalalCDC2015},~\cite{Jalal2017linear},~\cite{Jalal2019LCSS},~\cite{JalalCDC2018}, and have recently been extended to deep structured teams~\cite{Jalal2019risk,Jalal2020CCTA}. In such models, the solution concept is  globally team-optimal. It is observed  that 
 the  Nash solution presented in~Theorem~\ref{thm:aux} under Assumption~\ref{ass:collaborative_cost} coincides with the team-optimal solution of  mean-field teams. This implies that the  collaborative LQ game and  LQ  team with an arbitrary number of  players have identical solutions under PS and DSS information structures.  Note that the Nash and  team-optimal solutions are not necessarily the same in LQ games  with identical cost functions, where local  convexity (Nash solution)  is different from global convexity (team-optimal solution).    
In addition, the authors in \cite{sanjari2020optimal} and~\cite{sanjari2019optimal}  characterize
exchangeability properties for convex static and dynamic stochastic team problems and
establish the convergence of a sequence of team-optimal polices for $n$-agent teams to a globally optimal solution as $n \rightarrow \infty$, where the information structure is NS and the dynamics are decoupled.

\section{Two counterexamples}\label{sec:special3}
In this section, we first present two counterexamples to illustrate a fundamental difference between finite- and infinite-population solutions.
\subsection{Finite-population model} Consider a finite-population game with homogeneous weights. Let  $S^x_t=-Q_t$ and $S^u_t=-R_t$, $\forall t \in \mathbb{N}_T$. In addition, let matrices $Q_t+\frac{1}{n-1} G^x_t$ and $\bar Q_t+G^x_t-Q_t$ be positive semi-definite and matrices  $R_t+\frac{1}{n-1} G^u_t$ and $\bar R_t+G^u_t-R_t$  positive definite, for all $t \in \mathbb{N}_T$. In this case,  it can be shown that Assumption~\ref{ass:invertible} is satisfied and, the non-standard Riccati equation~\eqref{eq:riccati-bar-m} decomposes  into decoupled two standard Riccati equations as follows:
    \begin{multline}\label{eq:standard-1}
  \mathbf P_t=\DIAG(Q_t+\frac{1}{n-1}
 G^x_t,\bar Q_t+G^x_t-Q_t)
  +\mathbf A_t^\intercal \mathbf P_{t+1} \mathbf A_t \\
  -  \mathbf A_t^\intercal \mathbf P_{t+1} \mathbf B_t  ( \DIAG(R_t+ \frac{1}{n-1}
G^u_t, \bar R_t+ G^u_t-R_t)\\
+ \mathbf B^\intercal_t \mathbf  P_{t+1} \mathbf B_t)^{-1} \mathbf B^\intercal \mathbf P_{t+1} \mathbf A_t,
  \end{multline} 
  where  at any $t \in \mathbb{N}_T$,  $\mathbf P^{1,1}_t=P_t$ and
\begin{align}
\DIAG(\theta^n_t,\bar \theta^n_t):=- ( \DIAG(R_t+\frac{1}{n-1} G^u_t, \bar R_t+ G^u_t-R_t)\\
+ \mathbf B^\intercal_t \mathbf  P_{t+1} \mathbf B_t)^{-1} \mathbf B^\intercal \mathbf P_{t+1} \mathbf A_t.
\end{align}

\begin{Counterexample}\label{counter1}
\emph{ Consider the following  scalars: $Q_t <0$, $R_t >0$, $S^x_t=-Q_t$, $S^u_t=-R_t$, $\bar Q_t=Q_t$,  $\bar R_t=G^u_t=R_t$, and $G^x_t=-100 Q_t$, $  \forall t\in \mathbb{N}_T$. With this set of parameters,  the solution of the first  Riccati equation  (i.e. $\mathbf P^{1,1}_t, t \in \mathbb{N}_T$)  is positive  for relatively  small $n$ but it   becomes  negative for  sufficiently large $n$,  because   $Q_t$ in~\eqref{eq:standard-1} becomes  dominant.  Therefore, matrix $F_t(0)$ in the infinite-population game  becomes negative,  which violates the necessary (best-response condition obtained in the proof of Theorem~\ref{thm:aux})  for  the existence of a Nash solution (note that $F_t(0)= \left[  R_t+B_t^\intercal \mathbf{P}^{1,1}_{t+1} B_t  \right]$). As a result, the infinite-population game in this example admits no solution while the finite-population game admits a unique  solution for small~$n$.
}
\end{Counterexample}

\subsection{Infinite-population model} Consider an infinite-population game  with homogeneous weights, where the model parameters are chosen according to   Assumption~\ref{ass: mean-field_decoupled}. 

\begin{Counterexample}\label{counter2}
\emph{ Let $\bar Q_t$, $\bar R_t$, $G^x_t$ and $G^u_t$ be  large negative definite matrices  for every $t \in \mathbb{N}_T$.  In this case,   $\bar F_t(1/n)$ in~\eqref{eq:breve-f} becomes  a large negative definite matrix violating  the  positive definiteness condition (that is a necessary condition for the best-response equation).\footnote{Note that a large negative $\bar Q_{t+1}$ leads to a large negative $\bar Q_{t+1}^{2,2}$ and $\mathbf  P_{t+1}{}^{2,2}$, which results in a large negative $\bar F_t(1/n)$ according to~\eqref{eq:breve-f}.}  On the other hand,  according to Proposition~\ref{cor:decoupled1} and Theorem~\ref{thm:aux}, the existence and uniqueness of the solution of the infinite-population game is  independent of matrices $\bar Q_t$, $\bar R_t$, $G^x_t$ and $G^u_t$, because the effect of   $\bar F_t(1/n)$ in Assumption~\ref{ass:invertible_infinite_pop} vanishes at the rate $1/n$ as $n\rightarrow \infty$. This means  that although the Nash solution does  not exist in the finite-population game in this example, an approximate Nash solution may exist (which is the  solution of  the infinite-population game).
}
\end{Counterexample}

\section{Generalizations }\label{sec:generalization}
A salient feature of  the proposed methodology is  the fact that  it  can be naturally extended to other variants of LQ games such as  zero-sum, risk-sensitive,  multiple-subpopulation games with  collaborative and non-collaborative sub-populations.  In addition,  its unified framework can shed light on the similarities and differences between mean-field games and mean-field-type games. This  is useful  in providing  an explicit  closed-form solution for the existing  results in mean-field models.  For example, major-minor  mean-field model~\cite{Huang2010large}  may be  viewed as a special case of the multiple-subpopulation game, where the size of  the major sub-population is one.   In what follows, we  briefly discuss the  extension of our results to multiple sub-populations and  multiple  linear  regressions (also called features).

Consider a  game  with $K \in \mathbb{N}$ disjoint sub-populations, where each sub-population $k \in \mathbb{N}_K$ consists of $n_k \in \mathbb{N}$  players and $f_k \in \mathbb{N}$   linear regressions. The  players  are coupled through $\sum_{k=1}^K f_k \in \mathbb{N}$ linear regressions of the states and actions of all players.   Using the gauge transformation proposed in~\cite{Jalal2019risk},  one can  reduce the ($\sum_{k=1}^K n_k$)-player game  to a $(K+\sum_{k=1}^K f_k)$-player game wherein    $(\sum_{k=1}^K n_k)$ standard coupled   matrix Riccati equations reduce to  $(K+\sum_{k=1}^K f_k)$ non-standard coupled  matrix Riccati equations.   For example,  in this article, one has $K=1$ and $f_k=1$, which leads to  $2$  non-standard coupled matrix Riccati equations,  presented in one formulation  in~\eqref{eq:riccati-bar-m}.  For the special case of major-minor model,  one has   $K=2$ and $f_k=1$, $k \in \mathbb{N}_K$,  where the major's state is equal to its feature,  because  its size is one, i.e., $x^0_t=\bar x^0_t$, leading to  $2+(2-1)=3$ non-standard  coupled  matrix Riccati equations.

Other interesting generalizations  could be reinforcement learning~\cite{Masoud2020CDC,Jalal2020CCTA,Vida2020CDC}, min-LQG games~\cite{salhab2019collective}, noisy observations~\cite{Jalal2021CDC_KF}, constrained optimization~\cite{Jalal2021CDC_MPC},  common-noise~\cite{carmona2016} and systems with  Markov jumps~\cite{wang2012mean}, to  name only a few.

 \begin{figure}[t!]
\centering
\includegraphics[trim={0cm 8cm 0 8cm},clip, width=\linewidth]{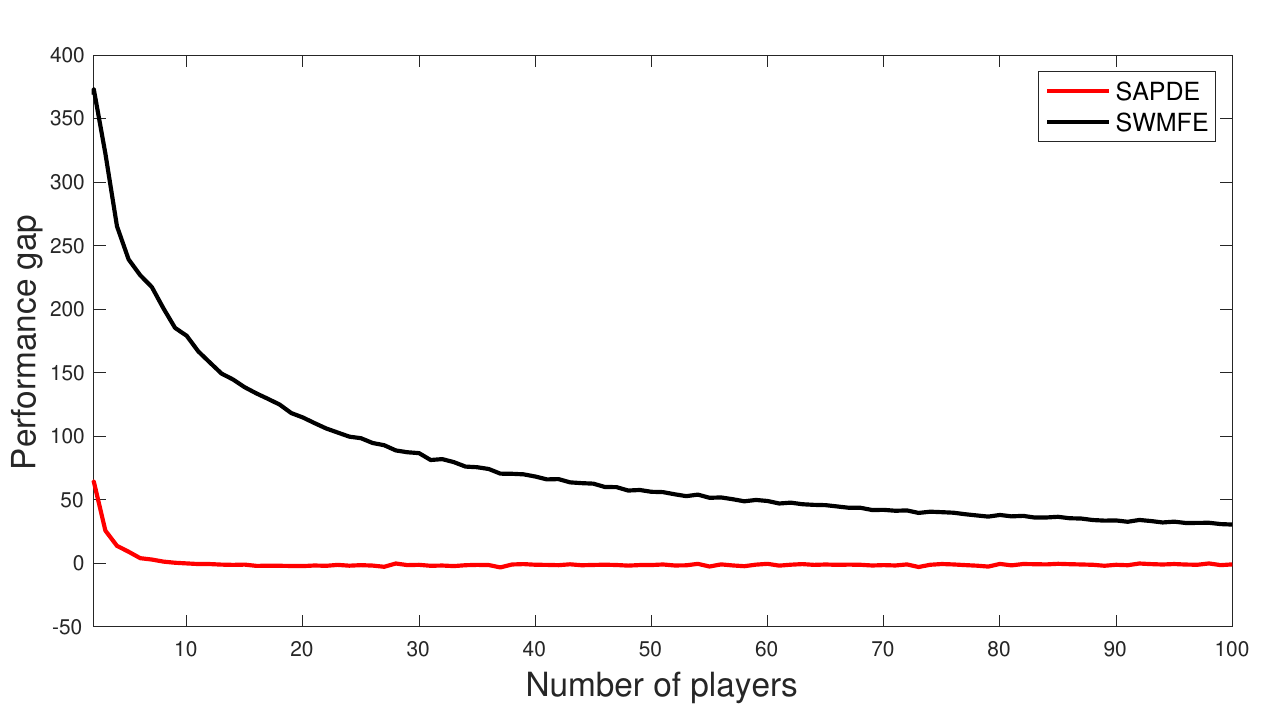}
\caption{The performance gap in Example 1 for  the two NS strategies proposed in~\eqref{eq:action-ns} and~\eqref{eq:action-ns-inf}.}\label{fig1}
\end{figure}
\section{A Numerical Example}\label{sec:numerical}
A numerical example is presented to illustrate the difference between the two NS strategies proposed   in~\eqref{eq:action-ns} and~\eqref{eq:action-ns-inf}, where the weights are homogeneous.

\textbf{Example 1.} Consider a game described in Section~\ref{sec:problem} with the following numerical parameters: $A_t=1$, $B_t=1$, $\bar A_t=0$, $\bar B_t=0$,  $Q_t=1$, $S^x_t=-0.5$, $\bar Q_t=5$, $R_t=5$, $S^u_t=\bar R_t=G^x_t=G^u_t=0$, $\mu_x=10$, $\VAR(x^i_1)=2$, $\VAR(w^i_t)=1$ and $T=50$ for every $t \in \mathbb{N}_T$ and $i \in \mathbb{N}_n$, where the probability distributions of  the initial states and local noises are i.i.d. and Gaussian.
It is shown in Figure~\ref{fig1} that the finite-population NS strategy proposed in \eqref{eq:action-ns},  which takes the number of players into account, converges to the sub-game perfect Nash solution faster than the infinite-population  NS strategy  proposed in~\eqref{eq:action-ns-inf}, as $n \rightarrow \infty$.

\section{Acknowledgement}
The first author would like to thank professors Tamer Basar and Boualem Djehiche for helpful and insightful discussions pertained to the proof of Theorem 1.

\section{Conclusions}\label{sec:conclusion}

In this paper, we  studied  a class of collaborative and non-collaborative linear quadratic games for three different types of weights under   three different information structures,  and obtained exact and approximate sequential equilibria by deriving a  novel non-standard Riccati equation.  In addition, we investigated two special cases wherein the non-standard Riccati equation  reduces to two decoupled standard Riccati equations. The key idea was  to use a gauge transformation to induce   some  orthogonality  and linear dependence  among variables in order to arrive at  a  low-dimensional solution.  Moreover,  we introduced performance and rationality gaps and  established several convergence results and characterized  the role of  the number of players in collaborative and non-collaborative games.  Furthermore, we generalized our main  results to the discounted  infinite-horizon cost function, multiple sub-populations and multiple features.

\bibliographystyle{IEEEtran}
\bibliography{Jalal_Ref}
\appendices
\section{Proof of Proposition~\ref{proposition:exchangeable} }\label{sec:proof_proposition:exchangeable}
From \eqref{eq:dynamics-general},  the dynamics of players $i$ and $ j$  may be expanded as follows:
\begin{align}\label{eq:exchanegable-ii}
x^i_{t+1}=&  a^{i,i}_t x^i_t  +a^{i,j}_t x^j_t +b^{i,i}_t u^i_t +b^{i,j}_t u^j_t \nonumber \\
&+ \sum_{k=1, k\neq i,j}^n a^{i,k}_t x^k_t+  \sum_{k=1, k\neq i,j}^n b^{i,k}_t u^k_t+w^i_t,
\end{align}
and 
\begin{align}\label{eq:exchanegable-jj}
x^j_{t+1}=&  a^{j,j}_t x^j_t  +a^{j,i}_t x^i_t +b^{j,j}_t u^j_t +b^{j,i}_t u^i_t \nonumber\\
&+ \sum_{k=1, k\neq i,j}^n a^{j,k}_t x^k_t+  \sum_{k=1, k\neq i,j}^n b^{j,k}_t u^k_t+w^j_t.
\end{align}
Exchange $(x^i_t,u^i_t,w^i_t)$ with $(x^j_t,u^j_t,w^j_t)$ in the above equations and get
\begin{align}\label{eq:exchanegable-ij}
x^j_{t+1}=&  a^{i,i}_t x^j_t  +a^{i,j}_t x^i_t +b^{i,i}_t u^j_t +b^{i,j}_t u^i_t \nonumber \\
&+ \sum_{k=1, k\neq i,j}^n a^{i,k}_t x^k_t+  \sum_{k=1, k\neq i,j}^n b^{i,k}_t u^k_t+w^j_t,
\end{align}
and 
\begin{align}\label{eq:exchanegable-ji}
x^i_{t+1}=&  a^{j,j}_t x^i_t  +a^{j,i}_t x^j_t +b^{j,j}_t u^i_t +b^{j,i}_t u^j_t \nonumber \\
&+ \sum_{k=1, k\neq i,j}^n a^{j,k}_t x^k_t+  \sum_{k=1, k\neq i,j}^n b^{j,k}_t u^k_t+w^i_t.
\end{align}
From  Definition~\ref{def:exchangeable-players}, \eqref{eq:exchanegable-ii} and \eqref{eq:exchanegable-jj} must be equal to \eqref{eq:exchanegable-ji} and \eqref{eq:exchanegable-ij}, respectively. Hence,  for every $ i,j \in \mathbb{N}_n$, we get 
\begin{align*}
a^{i,i}_t=a^{j,j}_t=:\tilde a_t,& \quad &a^{i,j}_t=a^{j,i}_t=:\tilde d_t,\nonumber \\
b^{i,i}_t=b^{j,j}_t=:\tilde b_t,& \quad &b^{i,j}_t=b^{j,i}_t=:\tilde e_t.
\end{align*}
Consequently, the dynamics of player  $i \in \mathbb{N}_n$ may be re-written in the form of \eqref{eq:dynamics-mf} 
where
\begin{align*}
A_t:=&\tilde a_t -\tilde d_t, \quad  &B_t:=&\tilde b_t -\tilde e_t,\nonumber \\
\bar A_t:=& n \tilde{d}_t, \quad &\bar B_t:=&n \tilde e_t.
\end{align*}
Now,  consider the first part of   $c^i_t$, i.e. ${\mathbf x_t}^\intercal Q^i_t {\mathbf x_t}$,  and expand it as follows:
\begin{equation}\label{eq:proof-exchangeable-1}
\begin{aligned}
{\mathbf x_t}^\intercal Q^i_t {\mathbf x_t}=&{x^i_t}^\intercal q^{i,i}_t x^i_t  +\sum_{k=1,  k\neq i}^n {x^i_t}^\intercal (q^{i,k}_t+ ({q^{k,i}_t})^\intercal) x^k_t\\
&+\sum_{k=1, k \neq i}^n{x^k_t}^\intercal q^{k,k}_t x^k_t \\
&+\sum_{k=1, k \neq i}^n \bigg( \sum_{k'=1, k' \neq i,k}^n {x^k_t}^\intercal (q^{k,k'}_t+ ({q^{k',k}_t})^\intercal)x^{k'}_t\bigg).
\end{aligned}
\end{equation}
According to Definition \ref{def:exchangeable-players}, exchanging any two arbitrary players $k,k' \neq i$ must not change $c^i_t$. Then, we have
\begin{equation}\label{eq:proof-exchangeable1}
\tilde s^i_t:=q^{i,k}_t+ ({q^{k,i}_t})^\intercal, \tilde q^i_t:=q^{k,k}_t=q^{k',k'}_t, \hat s^i_t:= q^{k,k'}_t+ ({q^{k,k'}_t})^\intercal.
\end{equation}
Given \eqref{eq:proof-exchangeable1}, re-arrange \eqref{eq:proof-exchangeable-1} as follows:
\begin{align}\label{eq:proof-exchangeable2}
&{\mathbf x_t}^\intercal Q^i_t {\mathbf x_t}={x^i_t}^\intercal q^{i,i}_t x^i_t  +\left(\sum_{k=1}^n {x^i_t}^\intercal \tilde{s}^i_t x^k_t \right) - {x^i_t}^\intercal \tilde{s}^i_t x^i_t \nonumber\\
&\quad + \left( \sum_{k=1}^n {x^k_t}^\intercal \tilde q^i_t x^k_t \right) - {x^i_t}^\intercal \tilde{q}^i_t x^i_t \nonumber \\
&\quad +\sum_{k=1}^n  \sum_{k'=1}^n {x^k_t}^\intercal \hat s^i_t x^{k'}_t - \sum_{k=1}^n {x^k_t}^\intercal \hat s^i_t x^k_t -\sum_{k=1}^n {x^k_t}^\intercal \hat s^i_t x^i_t \nonumber \\
&\quad - \sum_{k'=1}^n {x^i_t}^\intercal \hat{s}^i_t x^{k'}_t +2 {x^i_t}^\intercal \hat{s}^i_t x^i_t.
\end{align}
Define
\begin{align*}
q^i_t:=&q^{i,i}_t -\tilde s^i_t -\tilde q^i_t+2\hat{s}^i_t, \quad  &s^{x,i}_t:=&n(\tilde s^i_t - 2\hat s^i_t),\nonumber \\
p^{x,i}_t:=& n^2 \hat s^i_t , \quad &g^{x,i}_t:=&\tilde q^i_t - \hat s^i_t.
\end{align*}
Then, \eqref{eq:proof-exchangeable2} can be written as
\begin{equation}\label{eq:proof-exchangeable3}
{\mathbf x_t}^\intercal Q^i_t {\mathbf x_t}= {x^i_t}^\intercal q^i_t x^i_t + {x^i_t}^\intercal s^{x,i}_t \bar x_t +  {\bar x_t}^\intercal p^{x,i}_t\bar x_t+ \sum_{k=1}^n{x^k_t}^\intercal g^{x,i}_t x^k_t.
\end{equation}
From Definition \ref{def:exchangeable-players}, exchanging any two arbitrary players $i$ and~$j$  must exchange $c^i_t$ and $c^j_t$. Hence, it results from  \eqref{eq:proof-exchangeable3} that
\begin{align*}
Q_t:=&q^{i}_t=q^{j}_t, \quad  &2S^{x}_t:=&s^{x,i}_t=s^{x,j}_t,\nonumber \\
\bar Q_t:=& p^{x,i}_t=p^{x,j}_t , \quad &\frac{1}{n} G^{x}_t:=&g^{x,i}_t=g^{x,j}_t.
\end{align*}
Similar argument holds for $\mathbf u_t R^i_t \mathbf u_t$.

\section{Proof of Theorem~\ref{thm:aux}}\label{sec_proof_1}
From~\cite[Theorem 5]{basar1976uniqueness} and \cite[Proposition 2]{van1980note}, there is no loss of optimality in ignoring the history of states and restricting attention to Markovian strategies.  Fix the strategies of all players but  player~$i$, i.e. $\mathbf g^{-i}_{1:T}$. The resultant optimization problem from the player $i$'s view is Markovian; hence,  one can write the following dynamic program to identify the best response strategy of player $i$:
\begin{equation}\label{eq:DP-general-form}
V^i_{t}(\tilde{\mathbf x}_t)= \min_{u^i_t}(\Exp{c^i_t(\tilde{\mathbf x}_t,\tilde{\mathbf u}_t)+V^i_{t+1}(\tilde{\mathbf x}_{t+1}) \mid \tilde{\mathbf x}_t, \tilde{\mathbf u}_t }),
\end{equation}
where $V^i_{T+1}(\tilde{\mathbf{x}}_{T+1})= 0$.  Notice that   the value function $V^i_t(\tilde{\mathbf x}_t)$  depends on $\mathbf g^{-i}_{1:T}$, which naturally  leads to a fixed-point equation. We will show in the sequel that there is only one consistent
solution $V^i_t(\tilde{\mathbf x}_t)$, $i \in \mathbb{N}_n$, to construct a sequential Nash
equilibrium under Assumptions~\ref{ass:invertible},~\ref{ass:invertible_infinite_pop} and \ref{ass:collaborative_cost} (i.e., the fixed-point equation has a unique solution).   To this end,
we  prove  by backward induction that the value function of player $i$ at time~$t \in \mathbb{N}_{T+1}$ takes the following form:
 \begin{multline}\label{eq:value-function-form}
V^i_t(\tilde{\mathbf x}_t)= \begin{bmatrix}
\Delta x^i_t \\
\bar x_t
\end{bmatrix}^\intercal \mathbf P^{\alpha^i_n}_t \begin{bmatrix}
\Delta x^i_t \\
\bar x_t
\end{bmatrix}+ \ell^{i}_t\\
+ \sum_{j \neq i} \alpha^j_n (\Delta x^j_t)^\intercal P^{\alpha^i_n}_t (\Delta x^j_t) - \frac{(\alpha^i_n)^2}{1-\alpha^i_n} (\Delta x^i_t)^\intercal  P^{\alpha^i_n}_t (\Delta x^i_t).
\end{multline}
where  
\begin{equation}\label{eq:variance_o}
\begin{cases}
\ell^{i}_t:=\ell^{i}_{t+1}+\TR \big(\VAR(\VEC(\Delta w^i_t,\bar w_t))\mathbf P^{\alpha^i_n}_{t+1}\big)  \\
\quad  + \sum_{j \neq i}^n \hspace{-.1cm}\alpha^j_n  \TR \big(\VAR(\Delta w^j_t) P^{\alpha^i_n}_{t+1}\big)- \frac{(\alpha^i_n)^2\TR (\VAR(\Delta w^i_t) P^{\alpha^i_n}_{t+1})}{1-\alpha^i_n} , \\
\ell^{i}_{T+1}=0.
\end{cases}
\end{equation}

 It is  straightforward to verify that~\eqref{eq:value-function-form} holds at $t=T+1$ due to  the boundary conditions. 
  Suppose that~\eqref{eq:value-function-form} holds at $t+1$:
\begin{multline}\label{eq:value-function-form-t+1}
V^i_{t+1}(\tilde{\mathbf x}_{t+1})= \begin{bmatrix}
\Delta x^i_{t+1} \\
\bar x_{t+1}
\end{bmatrix}^\intercal \mathbf P^{\alpha^i_n}_{t+1}\begin{bmatrix}
\Delta x^i_{t+1} \\
\bar x_{t+1}
\end{bmatrix}+\ell^{i}_{t+1}\\
+ \sum_{j \neq i}\alpha^j_n (\Delta x^j_{t+1})^\intercal P^{\alpha^i_n}_{t+1} \Delta x^j_{t+1} - \frac{(\alpha^i_n)^2}{1-\alpha^i_n} (\Delta x^i_{t+1})^\intercal P^{\alpha^i_n}_{t+1} \Delta x^i_{t+1}.
\end{multline}
It is desired now to show that~\eqref{eq:value-function-form} holds at time $t$ as well.  From  \eqref{eq:cost-breve-bar}, \eqref{eq:DP-general-form}, and  \eqref{eq:value-function-form-t+1},  it follows that:
\begin{multline}
V^i_t(\tilde{\mathbf{x}}_t)=\min_{u^i_t}   \mathbb{E}\begin{bmatrix}
\Delta x^i_t \\
\bar x_t
\end{bmatrix}^\intercal \hspace{-.1cm} \mathbf Q_t^{\alpha^i_n} \hspace{-.1cm} \begin{bmatrix}
\Delta x^i_t \\
\bar x_t
\end{bmatrix}
 +
 \begin{bmatrix}
\Delta u^i_t \\
\bar u_t \\
 \end{bmatrix}^\intercal \hspace{-.1cm} \mathbf  R_t^{\alpha^i_n}\hspace{-.1cm} \begin{bmatrix}
\Delta u^i_t \\
\bar u_t \\
 \end{bmatrix} \\
+ \sum_{j \neq i} \alpha^j_n(\Delta x^j_t)^\intercal G^x_t (\Delta x^j_t) - \frac{(\alpha^i_n)^2}{1-\alpha^i_n} (\Delta x^i_t)^\intercal G^x_t (\Delta x^i_t)\\
+ \sum_{j \neq i} \alpha^j_n(\Delta u^j_t)^\intercal G^u_t (\Delta u^j_t) - \frac{(\alpha^i_n)^2}{1-\alpha^i_n} (\Delta u^i_t)^\intercal G^u_t (\Delta u^i_t)\\
+\begin{bmatrix}
\Delta x^i_{t+1} \\
\bar x_{t+1}
\end{bmatrix}^\intercal \mathbf P^{\alpha^i_n}_{t+1} \begin{bmatrix}
\Delta x^i_{t+1} \\
\bar x_{t+1}
\end{bmatrix}
 +\ell^{i}_{t+1} +\sum_{j \neq i} \alpha^j_n(\Delta x^j_{t+1})^\intercal \\ \times P^{\alpha^i_n}_{t+1} \Delta x^j_{t+1} - \frac{(\alpha^i_n)^2}{1-\alpha^i_n} (\Delta x^i_{t+1})^\intercal P^{\alpha^i_n}_{t+1} \Delta x^i_{t+1}.
\end{multline}

Incorporating~\eqref{eq:dynamics_joint} in the above equation yields:
\begin{multline}\label{eq:proof-aux-value}
V^i_t(\tilde{\mathbf{x}}_t)=\min_{u^i_t}   \mathbb{E}\begin{bmatrix}
\Delta x^i_t \\
\bar x_t
\end{bmatrix}^\intercal \hspace{-.1cm} \mathbf Q_t^{\alpha^i_n}  \hspace{-.1cm} \begin{bmatrix}
\Delta x^i_t \\
\bar x_t
\end{bmatrix}
 +
 \begin{bmatrix}
\Delta u^i_t \\
\bar u_t \\
 \end{bmatrix}^\intercal \hspace{-.1cm} \mathbf  R_t^{\alpha^i_n}\hspace{-.1cm}\begin{bmatrix}
\Delta u^i_t \\
\bar u_t \\
 \end{bmatrix} \\
+ \sum_{j \neq i} \alpha^j_n(\Delta x^j_t)^\intercal G^x_t (\Delta x^j_t) - \frac{(\alpha^i_n)^2}{1-\alpha^i_n} (\Delta x^i_t)^\intercal G^x_t (\Delta x^i_t)\\
+ \sum_{j \neq i} \alpha^j_n(\Delta u^j_t)^\intercal G^u_t (\Delta u^j_t) - \frac{(\alpha^i_n)^2}{1-\alpha^i_n} (\Delta u^i_t)^\intercal G^u_t (\Delta u^i_t)\\
+ ( \mathbf A_t  \begin{bmatrix}
\Delta x^i_{t} \\
\bar x_{t}
\end{bmatrix}+ \mathbf B_t    \begin{bmatrix}
\Delta u^i_{t} \\
\bar u_{t}
\end{bmatrix}+  \begin{bmatrix}
\Delta w^i_{t} \\
\bar w_{t}
\end{bmatrix} )^\intercal  \mathbf P^{\alpha^i_n}_{t+1} ( \mathbf A_t  \begin{bmatrix}
\Delta x^i_{t} \\
\bar x_{t}
  \end{bmatrix}\\+ \mathbf B_t   \begin{bmatrix}
\Delta u^i_{t} \\
\bar u_{t}
  \end{bmatrix}+ \begin{bmatrix}
\Delta w^i_{t} \\
\bar w_{t}
  \end{bmatrix} )   +\ell^{i}_{t+1}
+  \sum_{j \neq i}  \alpha^j_n (A_t \Delta x^j_{t} + B_t \Delta u^j_t\\+ \Delta w^j_t)^\intercal  P^{\alpha^i_n}_{t+1} (A_t \Delta x^j_{t}+ B_t \Delta u^j_t+ \Delta w^j_t)- \frac{(\alpha^i_n)^2}{1-\alpha^i_n}  (A_t \Delta x^i_{t}\\+ B_t \Delta u^i_t+ \Delta w^i_t)^\intercal P^{\alpha^i_n}_{t+1} (A_t \Delta x^i_{t}+ B_t \Delta u^i_t+ \Delta w^i_t).
\end{multline} 
In order to find a minimizer $u^i_t$,  one can  differentiate \eqref{eq:proof-aux-value} with respect to $u^i_t$ and set the derivative  to zero.  According to~\eqref{eq:def-breve},   the following relations hold for  any $i,j  \in \mathbb{N}_n$, $j \neq i$:  
\begin{equation}\label{eq:u-derivatice-relation}
\dfrac{\partial \Delta u^i_t}{\partial u^i_t}=1 - \alpha^i_n, \quad \dfrac{\partial \bar  u_t}{\partial u^i_t}=\alpha^i_n, \quad \dfrac{\partial \Delta u^j_t}{\partial u^i_t}=-\alpha^i_n.
\end{equation} 
From~\eqref{eq:u-derivatice-relation} and   after eliminating the terms depending on local noises (since they are  independent from   actions and  have zero mean), one arrives at:
\begin{align}\label{eq:proof_centralized_equality}
& \begin{bmatrix}
(1-\alpha^i_n)\mathbf I_{d_u \times d_u}\\
\alpha^i_n \mathbf I_{d_u \times d_u}
\end{bmatrix}^\intercal \mathbf R_t^{\alpha^i_n} \begin{bmatrix}
\Delta u^i_{t} \\
\bar u_{t}
\end{bmatrix}-\alpha^i_n \big( \sum_{j=1}^n  \alpha^j_n(\Delta u^j_t)^\intercal G^u_t \big)      \nonumber          \\
& +\begin{bmatrix}
(1-\alpha^i_n)\mathbf I_{d_u \times d_u}\\
\alpha^i_n \mathbf I_{d_u \times d_u}
\end{bmatrix}^\intercal  \mathbf  B_t^\intercal \mathbf P^{\alpha^i_n}_{t+1} \mathbf A_t \begin{bmatrix}
\Delta x^i_{t} \\
\bar x_{t}
\end{bmatrix} \nonumber \\                      
&+  \begin{bmatrix}
(1-\alpha^i_n)\mathbf I_{d_u \times d_u}\\
\alpha^i_n \mathbf I_{d_u \times d_u}
\end{bmatrix}^\intercal \mathbf B_t^\intercal \mathbf P^{\alpha^i_n}_{t+1} \mathbf B_t  \begin{bmatrix}
\Delta u^i_{t} \\
\bar u_{t}
\end{bmatrix}      -\alpha^i_n  \big( 
\sum_{j \neq i}  \alpha^j_n\nonumber \\ 
&\times  (\Delta x^j_t)^\intercal A_t^\intercal P^{\alpha^i_n}_{t+1} B_t + \alpha^i_n (\Delta x^i_t)^\intercal A_t^\intercal P^{\alpha^i_n}_{t+1} B_t 
 \big)          - \alpha^i_n\big(\sum_{j \neq i} \alpha^j_n  \nonumber  \\
 &\times (\Delta u^j_t)^\intercal  B_t^\intercal P^{\alpha^i_n}_{t+1} B_t + \alpha^i_n (\Delta u^i_t)^\intercal B_t^\intercal P^{\alpha^i_n}_{t+1} B_t 
 \big)=\mathbf 0.
\end{align}
Due to the linear dependence introduced by the gauge transformation, i.e., $\sum_{j=1}^n \alpha^j_n  \Delta x^j_t= \mathbf{0}_{d_x \times 1}$ and $\sum_{j=1}^n \alpha^j_n \Delta u^j_t=\mathbf{0}_{d_u \times 1}$,  equation~\eqref{eq:proof_centralized_equality} is simplified  as:
\begin{multline}\label{eq:eventually-u-derivative}
 \begin{bmatrix}
(1-\alpha^i_n)\mathbf I_{d_u \times d_u}\\
\alpha^i_n \mathbf I_{d_u \times d_u}
\end{bmatrix}^\intercal \left(\mathbf R^{\alpha^i_n}_t + \mathbf B_t^\intercal \mathbf  P^{\alpha^i_n}_{t+1} \mathbf B_t\right) \begin{bmatrix}
\Delta u^i_{t} \\
\bar u_{t}
\end{bmatrix} =\\
 -\begin{bmatrix}
(1-\alpha^i_n)\mathbf I_{d_u \times d_u}\\
\alpha^i_n \mathbf I_{d_u \times d_u}
\end{bmatrix}^\intercal \mathbf B_t^\intercal  \mathbf P^{\alpha^i_n}_{t+1} \mathbf A_t \begin{bmatrix}
\Delta x^i_{t} \\
\bar x_{t}
\end{bmatrix}.
\end{multline}
Equation~\eqref{eq:eventually-u-derivative}  can be rewritten  in terms of matrices $F_t(\alpha^i_n) $, $\bar F_t(\alpha^i_n)$, $K_t(\alpha^i_n)$ and $\bar K_t(\alpha^i_n)$,  defined in \eqref{eq:breve-f},  as follows:
\begin{equation}\label{eq:proof_aux_final}
F_t(\alpha^i_n)  \Delta u^i_t + \bar F_t(\alpha^i_n) \bar u_t= K_t(\alpha^i_n) \Delta x^i_t + \bar K_t(\alpha^i_n) \bar x_t. 
\end{equation}
It is important to notice that equation~\eqref{eq:proof_aux_final} holds irrespective of the strategies  of other players. Equivalently, equation~\eqref{eq:proof_aux_final}  can   be expressed  in terms of the actions of other players, i.e.
\begin{multline}\label{eq:proof_aux_h}
((1-\alpha^i_n) F_t(\alpha^i_n) + \alpha^i_n \bar F_t(\alpha^i_n) ) u^i_t=(\bar K_t(\alpha^i_n)  - K_t(\alpha^i_n)  ) \\
\times \sum_{j \neq i }\alpha^j_n x^j_t +((1-\alpha^i_n) K_t(\alpha^i_n) + \alpha^i_n \bar K_t(\alpha^i_n) ) x^i_t \\
 -( \bar F_t(\alpha^i_n)  - F_t(\alpha^i_n)  ) \sum_{j \neq i } \alpha^j_n u^j_t,
\end{multline}
where \eqref{eq:proof_aux_h} constitutes the standard fixed-point equation  across all players $i \in \mathbb{N}_n$. Since the control laws of all other players, i.e. $\mathbf g^{-i}_{t}$,  at time $t$ are fixed and the action $u^i_t$  has no effect on  its past (i.e.  $\mathbf x_{1:t}$), one can conclude that  $\sum_{j \neq i } \alpha^j_n u^j_t=\sum_{j \neq i } \alpha^j_n g^j_t(\mathbf x_{1:t})$ is independent of action $u^i_t$.  Therefore, action $u^i_t$ in equation~\eqref{eq:proof_aux_h}  is   the unique minimizer and the best-response action of player~$i$  if   the second derivative, i.e., matrix $(1-\alpha^i_n) F_t(\alpha^i_n) + \alpha^i_n \bar F_t(\alpha^i_n)$ is positive definite,   guaranteeing strict convexity. 

We now show that~\eqref{eq:proof_aux_final} (equivalently~\eqref{eq:proof_aux_h}) admits only one consistent solution across all  players  under Assumptions~\ref{ass:invertible},~\ref{ass:invertible_infinite_pop} and \ref{ass:collaborative_cost}.  For this reason, we consider  three different cases:
\begin{enumerate}
\item Let Assumption~\ref{ass:invertible} hold and   $\alpha^i_n=1/n$. The following equalities are obtained by  averaging \eqref{eq:proof_aux_final}  over all players and upon  noting   that $\sum_{i=1}^n  \alpha^i_n \Delta x^i_t= \mathbf{0}_{d_x \times 1}$ and $\sum_{i=1}^n  \alpha^i_n \Delta u^i_t=\mathbf{0}_{d_u \times 1}$:  
\begin{equation} \label{eq:proof_aux_breve_bar_u_1}
 \Delta u^i_t= F^{-1}_t(\frac{1}{n}) K_t(\frac{1}{n}) \Delta x^i_t,\quad \bar u_t= \bar F^{-1}_t(\frac{1}{n})  \bar K_t(\frac{1}{n}) \bar x_t.
\end{equation}
Note that averaging a set of linear equations does not affect their solution, implying~\eqref{eq:proof_aux_breve_bar_u_1} is the unique solution for~\eqref{eq:proof_aux_h}.
\item  Define $\delta_n:=[- \beta_{\text{max}}/n, \beta_{\text{max}}/n]$ and  let Assumption~\ref{ass:invertible_infinite_pop} hold,  where  $\alpha^i_n=\beta^i/n$, $\forall i \in \mathbb{N}_n$, is a monotonically decreasing (or increasing) sequence.   Pick a sufficiently large $n_0$ such that $\alpha^i_n \in \delta_n \subseteq \delta_{n_0}$. Under Assumption~\ref{ass:invertible_infinite_pop}, $F^{-1}_t(\alpha^i_n)$ and $\bar F^{-1}_t(\alpha^i_n)$,  $\forall \alpha^i_n \in \delta_n$,  exist and are uniformly bounded  with respect to $\alpha^i_n$. In addition, from~\eqref{eq:riccati-bar-m} and~\eqref{eq:breve-f},  it follows that $F^{-1}_t(\alpha^i_n)$ and $\bar F^{-1}_t(\alpha^i_n)$, $\forall \alpha^i_n \in \delta_n$, are continuous in $\alpha^i_n$, which implies that  the infinite-population limits exist (due to the uniform boundedness and continuity), i.e.,  $\lim_{n \rightarrow \infty} \mathbf P_t^{\alpha^i_n}=\mathbf P_t^0$, $\lim_{n \rightarrow \infty} F_t(\alpha^i_n)=F_t(0)$, $\lim_{n \rightarrow \infty} K_t(\alpha^i_n)=K_t(0)$, and so on. Therefore,   the infinite-population version of the equality in~\eqref{eq:proof_aux_final} is
\begin{equation}\label{eq:proof_aux_final_test}
F_t(0)  \Delta u^i_t + \bar F_t(0) \bar u_t= K_t(0) \Delta x^i_t + \bar K_t(0) \bar x_t. 
\end{equation}
Subsequently, one solution for~\eqref{eq:proof_aux_final_test} is
\begin{equation} \label{eq:proof_aux_breve_bar_u_2}
 \Delta u^i_t= (F_t(0))^{-1} K_t(0) \Delta x^i_t,\quad 
   \bar u_t= (\bar F_t(0))^{-1}  \bar K_t(0) \bar x_t.
\end{equation}
To establish the fact that the above solution is the unique solution, one can utilize an  averaging technique, similar to the one used for homogeneous case,  upon  noting   that $\sum_{i=1}^n  \alpha^i_n \Delta x^i_t= \mathbf{0}_{d_x \times 1}$ and $\sum_{i=1}^n  \alpha^i_n \Delta u^i_t=\mathbf{0}_{d_u \times 1}$. 
\item Let Assumption~\ref{ass:collaborative_cost} hold.  Since $\alpha^i_n$ is positive, one can multiply the per-step cost of player $i$ by the positive number $\frac{1-\alpha^i_n}{\alpha^i_n}$, which does not affect the best-response optimization at  player $i \in \mathbb{N}_n$. Next, one can use a backward induction to show  that   $\mathbf P^{\alpha^i_n}_t$ is diagonal, i.e.,   ${\mathbf P_t^{\alpha^i_n}}^{\hspace{0cm}(2,1)}$ and     ${\mathbf P_t^{\alpha^i_n}}^{\hspace{0cm}(1,2)}$ are zero. In such a case,~\eqref{eq:eventually-u-derivative} holds irrespective of $\alpha^i_n$ because all terms happen to have the same  coefficient $(1-\alpha^i_n)$, which can be removed from both sides of the equality, i.e, 
\begin{align}\label{eq:proof_thm1_aux}
&(1-\alpha^i_n)(G^u_t+ B_t^\intercal \mathbf P^{1,1}_{t+1} B_t) \Delta u^i_t + (1-\alpha^i_n)(\bar R_t+G^u_t \nonumber \\
&+ (B_t+\bar B_t)^\intercal \mathbf P^{2,2}_{t+1} (B_t+\bar B_t) )\bar u_t=(1-\alpha^i_n) (B_t^\intercal \mathbf P^{1,1}_{t+1} \nonumber \\
&\times A_t) \Delta x^i_t+(1-\alpha^i_n)((B_t+\bar B_t)^\intercal \mathbf P^{2,2}_{t+1} (A_t+\bar A_t))\bar x_t.
\end{align}
Thus, weighted averaging \eqref{eq:proof_thm1_aux}  over all players  yields:\begin{align} \label{eq:proof_aux_breve_bar_u_3}
 \Delta u^i_t&=(G^u_t+ B_t^\intercal \mathbf P^{1,1}_{t+1} B_t)^{-1}  B_t^\intercal \mathbf P^{1,1}_{t+1} A_t) \Delta x^i_t, \nonumber \\
 \bar u_t&=(\bar R_t+ G^u_t+ (B_t+\bar B_t)^\intercal \mathbf P^{2,2}_{t+1} (B_t+\bar B_t))^{-1} \nonumber \\
&\quad \times ( B_t+\bar B_t)^\intercal \mathbf P^{2,2}_{t+1} (A_t+\bar A_t)) \bar x_t.
\end{align}
\end{enumerate}
The last step in the induction  is to  incorporate~\eqref{eq:proof_aux_breve_bar_u_1} and~\eqref{eq:proof_aux_breve_bar_u_2} into \eqref{eq:proof-aux-value} in order to  retrieve~\eqref{eq:value-function-form}, where for $\alpha \in \{0, \frac{1}{n}\}$,
 \begin{equation}
 \Compress
  \mathbf P^{\alpha}_t= \mathbf  Q^{\alpha}_t+ (\boldsymbol \theta^{\alpha}_t)^\intercal \mathbf R^{\alpha}_t \boldsymbol \theta^{\alpha}_t+ (\mathbf A_t+\mathbf B_t \boldsymbol \theta^{\alpha}_t)^\intercal \mathbf P^{\alpha}_{t+1}(\mathbf A_t+\mathbf B_t \boldsymbol \theta^{\alpha}_t),
  \end{equation}
 where $\boldsymbol \theta^{\alpha}_t$ is  described by~\eqref{eq:breve-f}. Similarly, one can plug~\eqref{eq:proof_aux_breve_bar_u_3} into \eqref{eq:proof-aux-value} in order to  obtain~\eqref{eq:Riccati_collaborative}. 
  
Finally, it is concluded that strategy~\eqref{eq:thm-optimal-mfs-u_1} is the unique SPNE under Assumption~\ref{ass:invertible} because it satisfies the unique best-response condition in~\eqref{eq:proof_aux_h} and unique consistent solution across all players in~\eqref{eq:proof_aux_breve_bar_u_1}, upon noting that the original variables can be uniquely computed from the auxiliary variables as follows:  $x^i_t=\Delta x^i_t+\bar x_t$ and $u^i_t=\Delta u^i_t + \bar u_t$.  Similarly, strategy~\eqref{eq:thm-optimal-mfs-u_2} is the unique SPNE under Assumption~\ref{ass:invertible_infinite_pop} because it satisfies the unique best-response condition in~\eqref{eq:proof_aux_h} and unique consistent solution across all players in~\eqref{eq:proof_aux_breve_bar_u_2}.  For the collaborative cost, strategy~\eqref{eq:thm-optimal-mfs-u_3} is the unique SPNE under Assumption~\ref{ass:collaborative_cost}  because it satisfies  the unique consistent solution across all players in~\eqref{eq:proof_aux_breve_bar_u_3}, and the unique best-response condition, where the second derivative  is always positive definite, which is $((1-\alpha^i_n)(G^u_t+ B_t^\intercal \mathbf P^{1,1}_{t+1} B_t)^{-1}  B_t^\intercal)  
+ \alpha^i_n (\bar R_t+ G^u_t+ (B_t+\bar B_t)^\intercal \mathbf P^{2,2}_{t+1} (B_t+\bar B_t))^{-1})$, $\alpha^i_n >0$, guaranteeing strict convexity.

\section{Proof of Proposition~\ref{cor:decoupled1}}\label{sec_proof_2}
For a sufficiently large $n$, $\alpha \in [\frac{-\beta_{max}}{n}, \frac{\beta_{max}}{n}]$  can be  chosen   sufficiently small so that $F_t(\alpha) \approx (1-\alpha) (R_t +B_t^\intercal {\mathbf {P}_{t+1}^{\alpha}}^{ \hspace{-.2cm} 1,1} B_t) $ and $\bar F_t(\alpha)  \approx (1-\alpha)(R_t+S^u_t +B_t^\intercal {\mathbf {P}_{t+1}^{\alpha}}^{ \hspace{-.2cm} 2,1} B_t)$.  Consequently, matrices $F_t(\alpha)$ and $\bar F_t(\alpha)$ become positive definite, and are uniformly bounded in  $\alpha$ and $T$.   Therefore, Assumption~\ref{ass: mean-field_decoupled} satisfies Assumption~\ref{ass:invertible_infinite_pop}, and the infinite-population solution of the non-standard Riccati equation~\eqref{eq:riccati-bar-m} can be  reformulated  as follows:  ${\mathbf {P}_{t}}^{ 2,2}:={\mathbf {P}_{t}^{0}}^{ 2,1}$ and ${\mathbf {P}_{t}}^{1,1}:={\mathbf {P}_{t}^{0}}^{1,1}$.

\section{Proof of Lemma~\ref{lemma:PG_asymptotic}}\label{sec:proof_lemma:PG_asymptotic}
From Definition~\ref{def:performance_gap},   $J^i_n(\hat{\mathbf g}^i_n,\hat{\mathbf g}^{-i}_n) - J^i_n(\mathbf g^{\ast,i}_n, \mathbf g^{\ast,-i}_n)$,  $\forall i \in \mathbb{N}_n$, converges to zero, almost surely,  because the dynamics~\eqref{eq:dynamics-mf}, cost functions~\eqref{eq:cost-mfs}, and strategies~\eqref{eq:suppose_strategy} and~\eqref{eq:action-ns} are continuous and uniformly bounded in $n$ due to Assumptions~\ref{assump:independent_n} and~\ref{ass:existence_finite}, upon noting that at any time $t \in \mathbb{N}_T$, the deep state $\bar x_t$ in~\eqref{eq:dynamics_joint} under strategy $\mathbf g^\ast_n$, the prediction $z^n_t$ in~\eqref{eq:deterministic_process},  and the deep state $\hat{\bar x}_t$  under strategy~\eqref{eq:action-ns} converge to the same limit given by $z^\infty_t$ in~\eqref{eq:deterministic_process-inf} under strategy $\mathbf g^\ast_\infty$,  according to the  strong law of large numbers under Assumption~\ref{assump:iid}.    In particular, $J^i_n(\hat{\mathbf g}^i_n,\hat{\mathbf g}^{-i}_n)$ and  $J^i_n(\mathbf g^{\ast,i}_n,\mathbf g^{\ast,-i}_n)$ can  be described by   polynomial regressions  in terms of $\mu_x$, $\VAR(x^i_1)$ and  $\VAR(w^i_t)$, $t \in \mathbb{N}_T$,  because: (a)  the state dynamics are linear  under strategies $\hat{\mathbf g}_n$ and $\mathbf g^\ast_n$; (b)  local noises are zero-mean; (c)  per-step cost~\eqref{eq:cost-mfs} is a quadratic function of the states  under strategies $\hat{\mathbf g}_n$ and $\mathbf g^\ast_n$, and (d) the initial states as well as local noises are independent,  meaning that $\VAR(\bar x_1)=\frac{1}{n^2}\sum_{i=1}^n \VAR(x^i_1)$ and  $\VAR(\bar w_t)=\frac{1}{n^2}\sum_{i=1}^n \VAR(w^i_t)$.  Therefore, $J^i_n(\hat{\mathbf g}^i_n,\hat{\mathbf g}^{-i}_n)$ and  $J^i_n(\mathbf g^{\ast,i}_n,\mathbf g^{\ast,-i}_n)$  are bounded and continuous with respect to $n$ since the resultant  weighting matrices  as well as  auto-covariance matrices are all bounded and continuous with respect to~$n$, in view of Assumptions~\ref{assump:independent_n} and~\ref{assump:iid}.  Due to the strong law of large numbers, $\hat{\mathbf g}_n$ converges to $\mathbf g^\ast_n$, almost surely; hence, $J^i_n(\hat{\mathbf g}^i_n,\hat{\mathbf g}^{-i}_n)$ and  $J^i_n(\mathbf g^{\ast,i}_n,\mathbf g^{\ast,-i}_n)$  have the same limit,  which is $J^i_\infty(\hat{\mathbf g}^{i}_\infty,\hat{\mathbf g}^{-i}_\infty)=J^i_\infty(\mathbf g^{\ast,i}_\infty,\mathbf g^{\ast,-i}_\infty)$. A similar argument holds for  $J^i_n(\hat{\mathbf g}^i_\infty,\hat{\mathbf g}^{-i}_\infty) - J^i_n(\mathbf g^{\ast,i}_n, \mathbf g^{\ast,-i}_n)$,  $\forall i \in \mathbb{N}_n$.

\section{Proof of Lemma~\ref{lemma:rationality_gap}}\label{sec:proof_lemma:rationality_gap}
Suppose player $i \in \mathbb{N}_n$ plays  strategy  $\mathbf g^i \in \GNS$.  Let $y^j_t$ and $v^j_t$ denote,  respectively, the state and action of player $j \in \mathbb{N}_n$ at time $t \in \mathbb{N}_T$  under the joint strategy $\{\mathbf g^i, \mathbf g^{\ast, -i}_n\}$.  Similarly, let  $\hat y^j_t$ and $\hat v^j_t$ denote the state and action of player $j \in \mathbb{N}_n$  under the joint strategy $\{\mathbf g^i, \hat{ \mathbf g}^{-i}_n\}$. From~\eqref{eq:cost-mfs} and~\eqref{eq:total_cost}, one can write $J^i_n(\mathbf g^i,\mathbf g^{\ast, -i}_n)$ as a summation of quadratic terms (continuous functions) of local states $\{y^j_t\}_{j=1}^n$ and  local actions $\{v^j_t\}_{j=1}^n$ over the game  horizon $T$. Analogously,  one can express $J^i_n(\mathbf g^i, \hat{\mathbf g}^{ -i}_n)$ as a summation of quadratic terms of local states $\{\hat y^j_t\}_{j=1}^n$ and  local actions $\{\hat v^j_t\}_{j=1}^n$.  In what follows, we show that for every strategy $ \mathbf g^i \in \GNS$ under Assumption~\ref{ass:rationality_gap}, the set $\{y^j_t,v^j_t\}_{j=1}^n$  converges to  the set $\{\hat y^j_t,\hat v^j_t\}_{j=1}^n$,  almost surely, as $n \rightarrow \infty$.  Due to the fact that the cost function is uniformly bounded and continuous  with respect to  $n$ under Assumption~\ref{ass:invertible_infinite_pop}, and  continuous  in local states and local actions,   the gap  $ \Delta J_R( \mathbf g^i, \hat{\mathbf g}_n, \mathbf g^\ast_n)= J^i_n(\mathbf g^i, \hat{\mathbf g}^{-i}_n) -J^i_n(\mathbf g^i, \mathbf g^{\ast,-i}_n) $ converges to zero, as $n \rightarrow \infty$. In particular, for every $j \in \mathbb{N}_n$, one has:
\begin{align}
y^j_{t+1}=A_t y^j_t+B_t v^j_t+\bar A_t \bar y_t +\bar B_t \bar v_t+ w^j_t,\\
\hat y^j_{t+1}=A_t \hat y^j_t+B_t \hat v^j_t+\bar A_t \hat{\bar y}_t +\bar B_t \hat{\bar v}_t+ w^j_t,
\end{align}
where
\begin{align}
v^i_t&=g^i_t(y^i_{1:t}),\quad v^j_t=\theta^n_t y^j_t+(\bar \theta^n_t - \theta^n_t)\bar y_t, j \neq i, \\
\hat v^i_t&=g^i_t(\hat y^i_{1:t}),\quad \hat v^j_t=\theta^n_t \hat y^j_t+(\bar \theta^n_t - \theta^n_t)z^n_t, j \neq i.
\end{align}
In addition, 
\begin{equation}\label{eq:proof_bar_y}
\bar y_{t+1}=(A_t +\bar A_t)\bar y_t+(B_t +\bar B_t) \bar  v_t+ \bar w_t,
\end{equation}
\begin{equation}\label{eq:proof_hat_bar_y}
\hat{\bar y}_{t+1}=(A_t +\bar A_t) \hat{\bar y}_t+(B_t+\bar B_t) \hat{\bar v}_t+\bar  w_t,
\end{equation}
where
\begin{align}
\bar v_t&= \alpha^i_n\big(g^i_t(y^i_{1:t}) - \theta^n_t y^i_t -(\bar \theta^n_t - \theta^n_t) \bar y_t \big)+ \bar \theta^n_t \bar y_t, \\
\hat{\bar v}_t&= \alpha^i_n\big(g^i_t(\hat y^i_{1:t}) - \theta^n_t \hat y^i_t -(\bar \theta^n_t - \theta^n_t) z^n_t \big)+ \bar \theta^n_t z^n_t+ \theta^n_t (\hat{\bar y}_t -z^n_t).
\end{align}
From Assumptions~\ref{ass:invertible_infinite_pop} and~\ref{ass:existence_finite},  strategy~\eqref{eq:suppose_strategy} is uniformly bounded and continuous in $n$.   In this case, as $n$ grows to infinity, $\bar w_t$ converges to zero due to the strong law of large numbers and Assumption~\ref{assump:iid}. On the other hand, $ \lim_{n \rightarrow \infty} \alpha^i_n=0$; hence,  one can inductively show   that $\bar y_t$ in~\eqref{eq:proof_bar_y}, $\hat{\bar y}_t$ in~\eqref{eq:proof_hat_bar_y}, and $z^n_t$ in~\eqref{eq:deterministic_process} converge to the same limit, given by~\eqref{eq:deterministic_process-inf}, almost surely. In addition,  it can be shown  that $g^i_t(y^i_{1:t}) =g^i_t(\hat y^i_{1:t})$, almost surely, under Assumption~\ref{ass:rationality_gap}.II  as $n \rightarrow \infty$, and  $g^i_t(y^i_{1:t})=g^i_t(\hat y^i_{1:t})$ under Assumption~\ref{ass:rationality_gap}.I everywhere for any arbitrary $n$. Consequently,  all local states and local actions under the above two different joint strategies converge to the same limit,  as $n \rightarrow \infty$,  irrespective of the strategy $\mathbf g^i$.   A similar argument holds for $\hat{\mathbf  g}_\infty$, where  the actions under $\hat{\mathbf  g}_\infty$ are given by:
$
\hat v^j_t=\theta^\infty_t \hat y^j_t+(\bar \theta^\infty_t - \theta^\infty_t)z^\infty_t, j \neq i,$ and $
\hat{\bar v}_t= \alpha^i_n\big(g^i_t(\hat y^i_{1:t}) - \theta^\infty_t \hat y^i_t -(\bar \theta^\infty_t - \theta^\infty_t) z^\infty_t \big)+ \bar \theta^\infty_t z^\infty_t+ \theta^\infty_t (\hat{\bar y}_t -z^\infty_t).
$
\section{Proof of Lemma~\ref{lemma:breve_equality}}\label{sec:proof_lemma:breve_equality}
The proof follows from an induction.  Initially,  $\Delta x^i_1=x
^i_1 -\bar x_1=\hat x^i_1-\hat {\bar x}_1=\Delta \hat{x}^i_1$ and $\Delta u^i_1=\theta^n_1 \Delta x^i_1=\theta^n_1  \Delta \hat x^i_1=\Delta \hat{u}^i_1$. Suppose $\Delta x^i_t=\Delta \hat x^i_t$ and $\Delta u^i_t=\Delta \hat u^i_t$ at time $t$.  Then, it  results  from~\eqref{eq:dynamics-mf} and~\eqref{eq:dynamics-ns} that:
$\Delta x^i_{t+1}=A_t \Delta x^i_t+B_t \Delta u^i_t+\Delta w^i_t 
=A_t \Delta \hat x^i_t+B_t \Delta \hat u^i_t+\Delta w^i_t=\Delta \hat x^i_{t+1}.
$
In addition, one arrives at:
$
\Delta u^i_{t+1}= \theta^n_{t+1} \Delta x^i_{t+1}=\theta^n_{t+1} \Delta \hat x^i_{t+1}=\Delta \hat u^i_{t+1}$.
According to~\eqref{eq:dynamics_joint},~\eqref{eq:thm-optimal-mfs-u_1},~\eqref{eq:dynamics-ns},~\eqref{eq:action-ns} and \eqref{eq:relative_error}, the proof is completed for SAPDE in~\eqref{eq:action-ns}. For SWMFE in~\eqref{eq:action-ns-inf},  a similar argument holds for $\tilde A^\infty_t$ in the finite-population game.

\section{Proof of Lemma~\ref{lemma:relative_dynamics_1}}\label{sec:proof_lemma:relative_dynamics_1} 
 From Definition~\ref{def:performance_gap}and equation~\eqref{eq:total_cost}, one has:
\begin{align}
&\Delta J^i_P= \mathbb{E} [\sum_{t=t_0}^T [ \begin{array}{c}
\Delta \hat x^i_t \\
\hat{\bar x}_t \\
\end{array}]^\intercal \mathbf Q_t [ \begin{array}{c}
\Delta \hat x^i_t \\
\hat{\bar x}_t \\
\end{array}] + [ \begin{array}{c}
\Delta \hat u^i_t \\
\hat{\bar u}_t \\
\end{array}]^\intercal \mathbf R_t [ \begin{array}{c}
\Delta \hat u^i_t \\
\hat{\bar u}_t \\
\end{array}]\\
& +\frac{1}{n}( \sum_{j \neq i} (\Delta \hat x^j_t)^\intercal G^x_t (\Delta \hat x^j_t) - \frac{1}{n-1} (\Delta \hat x^i_t)^\intercal G^x_t (\Delta \hat x^i_t))\\
&+\frac{1}{n}( \sum_{j \neq i} (\Delta \hat u^j_t)^\intercal G^u_t (\Delta \hat u^j_t) - \frac{1}{n-1} (\Delta \hat u^i_t)^\intercal G^u_t (\Delta \hat u^i_t))\\
& -[ \begin{array}{c}
\Delta x^i_t \\
\bar x_t \\
\end{array}]^\intercal \mathbf Q_t [ \begin{array}{c}
\Delta x^i_t \\
\bar x_t \\
\end{array}] - [ \begin{array}{c}
\Delta u^i_t \\
\bar u_t \\
\end{array}]^\intercal \mathbf  R_t [ \begin{array}{c}
\Delta u^i_t \\
\bar u_t \\
\end{array}]\\
& -\frac{1}{n}( \sum_{j \neq i} (\Delta x^j_t)^\intercal G^x_t (\Delta x^j_t) - \frac{1}{n-1} (\Delta x^i_t)^\intercal G^x_t (\Delta x^i_t))\\
& -\frac{1}{n}( \sum_{j \neq i} (\Delta u^j_t)^\intercal G^u_t (\Delta u^j_t) - \frac{1}{n-1} (\Delta u^i_t)^\intercal G^u_t (\Delta u^i_t))] \nonumber \\
&\substack{(a)\\=} \mathbb{E} [\sum_{t=t_0}^T   \begin{bmatrix}
 e^i_t \\
 e_t+ z_t \\
\end{bmatrix}^\intercal \mathbf Q_t \begin{bmatrix}
 e^i_t \\
 e_t +z_t\\
\end{bmatrix} - \begin{bmatrix}
 e^i_t \\
 \zeta_t+ z_t 
\end{bmatrix}^\intercal \mathbf Q_t \begin{bmatrix}
 e^i_t \\
 \zeta_t +z_t\\
\end{bmatrix}
 \nonumber\\
&+\begin{bmatrix}
\theta_t  e^i_t \\
\theta_t  e_t+ \bar \theta_t z_t 
\end{bmatrix}^\intercal \hspace{0cm}\mathbf  R_t\begin{bmatrix}
\theta_t e^i_t \\
\theta_t  e_t+ \bar \theta_t z_t 
\end{bmatrix}  \nonumber  \\
&-\begin{bmatrix}
\theta_t e^i_t \\
\bar \theta_t  \zeta_t+ \bar \theta_t z_t 
\end{bmatrix}^\intercal  \hspace{0cm} \mathbf R_t  \begin{bmatrix}
\theta  e^i_t \\
\bar \theta_t  \zeta_t+ \bar \theta_t z_t 
\end{bmatrix}   \nonumber \\
& \substack{(b)\\=} \mathbb{E} [\sum_{t=1}^T  \begin{bmatrix}
 e^i_t \\
 e_t 
\end{bmatrix}^\intercal \mathbf Q_t  \begin{bmatrix}
 e^i_t \\
 e_t 
\end{bmatrix}
- \begin{bmatrix}
 e^i_t \\
 \zeta_t
\end{bmatrix}^\intercal \mathbf Q_t \begin{bmatrix}
 e^i_t \\
 \zeta_t 
\end{bmatrix}\\
&+\begin{bmatrix}
\theta_t  e^i_t \\
\theta_t  e_t
\end{bmatrix}^\intercal \mathbf R_t \begin{bmatrix}
\theta_t e^i_t \\
\theta_t  e_t
\end{bmatrix} -\begin{bmatrix}
\theta_t e^i_t \\
\bar \theta_t \zeta_t
\end{bmatrix}^\intercal \mathbf  R_t\begin{bmatrix}
\theta_t e^i_t \\
\bar \theta_t \zeta_t
\end{bmatrix}], 
\end{align}
where $(a)$ follows from Lemma~\ref{lemma:breve_equality} and equations~\eqref{eq:thm-optimal-mfs-u_1},~\eqref{eq:action-ns} and~\eqref{eq:relative_error}, and $(b)$ is a consequence of $\Exp{e_t}=\Exp{\zeta_t}=\mathbf{0}_{d_x \times 1}$, on noting that  $z_t$ is deterministic. The proof is now complete from the definition of $\tilde Q_t$.

\section{Proof of Theorem~\ref{thm:NS-optimality}}\label{sec:proof_thm:NS-optimality}
Prior to the proof, we establish the following  useful lemma.
\begin{Lemma}\label{lemma:tilde_M}
At any time $t \in \mathbb{N}_T$,   ${{}\tilde M^n_t}^{1,1}= \Zero$.
\end{Lemma}
\begin{proof}
 From Lemma~\ref{lemma:relative_dynamics_1} and Theorem~\ref{thm:delta_j_correlated}, it follows that Lemma~\ref{lemma:tilde_M} holds at the terminal time, i.e. ${{}\tilde M^n_T}^{1,1}={{}\tilde Q^n_T}^{1,1}=\Zero$.  It can be  inductively shown  that Lemma~\ref{lemma:tilde_M} holds at time $t \in \mathbb{N}_{T-1}$, according to equation~\eqref{eq:lyapunov-finite} and Lemma~\ref{lemma:relative_dynamics_1}, i.e.
${{}\tilde M^n_t}^{1,1}= ({{}\tilde A^n_t}^{1,1})^\intercal {{}\tilde M^n_{t+1}}^{ \hspace{-.2cm} 1,1} {{}\tilde A^n_t}^{1,1}+{{}\tilde Q^n_t}^{1,1}=\Zero.$
\end{proof}
 From Assumption~\ref{assump:iid},  one has: $
\VAR(\bar x_1)=\frac{1}{n^2}\sum_{i=1}^n \VAR(x^i_1) \leq \frac{1}{n}C_x, 
\VAR(\bar w_t)=\frac{1}{n^2}\sum_{i=1}^n \VAR(w^i_t) \leq \frac{1}{n}C_w$,
where $C_x$ and $C_w$ are some  upper bounds on the auto-covariance matrices of the initial states and local noises (that do not depend on $n$). In addition, for any $i \in \mathbb{N}_n$,  
$
\Exp{(x^i_1-\bar x_1)\bar x_1^\intercal}=\Exp{x^i_1 \bar x_1^\intercal} - \Exp{\bar x_1 \bar x_1^\intercal}
=\frac{1}{n} \VAR(x^i_1)  -\VAR(\bar x_1).
$
Similarly, $\Exp{(w^i_t-\bar w_t)\bar w_t^\intercal}=\frac{1}{n} \VAR(w^i_t)  -\VAR(\bar w_t)$   at any  $t \in \mathbb{N}_T$. Consequently,   all  block matrices of the  matrices $H^x_{t_0}$ and $H^w_t$, except the  ones in the first row and  first column,  decay to zero at the rate $1/n$. On the other hand,  the   block matrices in the first row and  first column of matrices $H^x_{t_0}$ and $H^w_t$  have no effect on $\Delta J^i_P(\hat{\mathbf g}_n, \mathbf g^\ast_n)_{t_0}$,  according to  equation~\eqref{eq:delta_expansion} and Lemma~\ref{lemma:tilde_M}.  Consequently, one can conclude that $\Delta J^i_P(\hat{\mathbf g}_n, \mathbf g^\ast_n)_{t_0}$, $i \in \mathbb{N}_n$,  described in Definition~\ref{def:performance_gap}, converges to zero at the rate $1/n$ because  matrices $\tilde M^n_{t_0:T}$ are uniformly bounded with respect to $n$ according to Assumption~\ref{assump:independent_n}, Lemmas~\ref{lemma:breve_equality} and~\ref{lemma:relative_dynamics_1} and equation~\eqref{eq:lyapunov-finite}.  Thus,  NS strategy~\eqref{eq:action-ns} is SARE according to Definition~\ref{def:pDSS}, where $\varepsilon_{t_0}(n):=\Delta J_P(\hat{\mathbf g}_n, \mathbf g^\ast_n)_{t_0}$ converges  to zero at the rate $1/n$.

Suppose Assumption~\ref{ass:rationality_gap} holds. From Theorem~\ref{thm:aux}, the following inequality holds for any $\mathbf g^i \in \GNS$,
\begin{align}
J^i_n(\hat{\mathbf g}^i_n&,\hat{\mathbf g}^{-i}_n)_{t_0}= J^i_n(\hat{\mathbf g}^i_n,\hat{\mathbf g}^{-i}_n)_{t_0} \pm  J^i_n(\mathbf g^{\ast,i}_n, \mathbf g^{\ast,-i}_n)_{t_0} \\
 & \leq J^i_n(\mathbf g^{\ast,i}_n, \mathbf g^{\ast,-i}_n)_{t_0}+ \Delta J_P(\hat{\mathbf g}_n, \mathbf g^\ast_n)_{t_0}\\
  & \leq J^i_n(\mathbf g^i, \mathbf g^{\ast,-i}_n)_{t_0} + \Delta J_P(\hat{\mathbf g}_n, \mathbf g^\ast_n) \pm J^i_n(\mathbf g^i, \hat{\mathbf g}^{-i}_n)_{t_0}\\
  & \leq   J^i_n(\mathbf g^i, \hat{\mathbf g}^{-i}_n)_{t_0} +  \Delta J_P(\hat{\mathbf g}_n, \mathbf g^\ast_n)+  \Delta J_R(\hat{\mathbf g}_n, \mathbf g^\ast_n)_{t_0},
\end{align}
where $\Delta J_P(\hat{\mathbf g}_n, \mathbf g^\ast_n)_{t_0} $ converges to zero according to~\eqref{eq:epsilon_finite},  and  $\Delta J_R(\hat{\mathbf g}_n, \mathbf g^\ast_n)_{t_0}$ converges to zero according to Lemma~\ref{lemma:rationality_gap}. The proof is now completed from  Definition~\ref{def:pDSS2}.

\section{Proof of Theorem~\ref{cor:ns-infinite}}\label{sec:proof_cor:ns-infinite}
From Definitions~\ref{def:pDSS} and~\ref{def:performance_gap}, let $\varepsilon_{t_0}(n)=\max_{i \in \mathbb{N}_n} | J^i_n(\hat{\mathbf g}^i_\infty, \hat{\mathbf g}^{-i}_\infty)_{t_0}- J^i_n(\mathbf g^{\ast,i}_n, \mathbf g^{\ast,-i}_n)_{t_0}|$. It is shown in Lemma~\ref{lemma:PG_asymptotic} that the above  performance gap converges to zero as $n \rightarrow \infty$ under Assumption~\ref{ass:existence_finite}. Since Assumption~\ref{ass:invertible_infinite_pop} satisfies Assumption~\ref{ass:existence_finite} for homogeneous weights, according to Corollary~\ref{cor:sufficient_7}, the proof of the first part follows accordingly.

 To prove the second part, one can use Lemmas~\ref{lemma:PG_asymptotic} and~\ref{lemma:rationality_gap} in a similar way proposed in Theorem~\ref{thm:NS-optimality} to  establish the convergence result under Assumption~\ref{ass:existence_finite}; however,  this approach comes at a price that the finite-population solution must exist. The beauty of the asymptotic Nash equilibrium is that it does not need such an assumption, and one can still establish the convergence result without  Assumption~\ref{ass:existence_finite} by replacing the finite-population solution with the infinite-population one. In particular,  from the triangle inequality, one has the following for every $n \geq n_0$ and $\mathbf g^i \in \GNS \subseteq \GC$:
\begin{align}
&J^i_n(\hat{\mathbf g}^i_\infty,\hat{\mathbf g}^{-i}_\infty)_{t_0}= J^i_n(\hat{\mathbf g}^i_\infty,\hat{\mathbf g}^{-i}_\infty)_{t_0} \pm  J^i_n(\mathbf g^{\ast,i}_\infty, \mathbf g^{\ast,-i}_\infty)_{t_0} \\
 & \leq J^i_n(\mathbf g^{\ast,i}_\infty, \mathbf g^{\ast,-i}_\infty)_{t_0} + \max_{i \in \mathbb{N}_n}|J^i_n(\hat{\mathbf g}^i_\infty,\hat{\mathbf g}^{-i}_\infty)_{t_0} - J^i_n(\mathbf g^{\ast,i}_\infty, \mathbf g^{\ast,-i}_\infty)_{t_0}  |\\
  &\substack{(a)\\ \leq} J^i_n(\mathbf g^i, \mathbf g^{\ast,-i}_\infty)_{t_0} +
  \max_{i \in \mathbb{N}_n}|J^i_\infty(\mathbf g^i, \mathbf g^{\ast,-i}_\infty)_{t_0}  - J^i_n(\mathbf g^i, \mathbf g^{\ast,-i}_\infty)_{t_0}  |\\
  &+   \max_{i \in \mathbb{N}_n}|J^i_n(\hat{\mathbf g}^i_\infty,\hat{\mathbf g}^{-i}_\infty)_{t_0} - J^i_n(\mathbf g^{\ast,i}_\infty, \mathbf g^{\ast,-i}_\infty)_{t_0}  |  J^i_n(\mathbf g^i,\hat{\mathbf g}^{-i}_\infty)_{t_0}\\ 
  &\substack{(b)\\ \leq} J^i_n(\mathbf g^i,\hat{ \mathbf g}^{-i}_\infty)_{t_0} +\bar \varepsilon_{t_0}(n),
\end{align}
where $(a)$ follows from  Theorem~\ref{thm:aux} and the triangle inequality,   and $(b)$ follows from 
\begin{align}\label{eq:bar_vareepsilon}
\bar \varepsilon_{t_0}(n)&:= \max_{i \in \mathbb{N}_n}|J^i_\infty(\mathbf g^i, \mathbf g^{\ast,-i}_\infty)_{t_0}  - J^i_n(\mathbf g^i, \mathbf g^{\ast,-i}_\infty)_{t_0}  |\nonumber \\
&\quad + \max_{i \in \mathbb{N}_n}|J^i_n(\hat{\mathbf g}^i_\infty,\hat{\mathbf g}^{-i}_\infty)_{t_0} - J^i_n(\mathbf g^{\ast,i}_\infty, \mathbf g^{\ast,-i}_\infty)_{t_0}  |\nonumber \\
&\quad+ \sup_{\mathbf g^i}| J^i_n(\mathbf g^i,\hat{\mathbf g}^{-i}_\infty)_{t_0}-J^i_n(\mathbf g^i,\mathbf g^{\ast,-i}_\infty)_{t_0}|.
\end{align}
The first term of the right-hand side of inequality~\eqref{eq:bar_vareepsilon} converges to zero because $J^i_n(\mathbf g^i, \mathbf g^{\ast,-i}_\infty)_{t_0}$ is continuous and uniformly bounded in $n$ under population-size-independent model condition. The second term of the right-hand side of inequality~\eqref{eq:bar_vareepsilon} is zero, almost surely, because $\hat{\mathbf g}_\infty$ in~\eqref{eq:action-ns-inf} is equal to $\mathbf g^\ast_\infty$ in~\eqref{eq:thm-optimal-mfs-u_2} with probability one, due to the strong law of large numbers, Assumption~\ref{assump:iid}, and Definition~\ref{def:PI_prediction}. The third term  of the right-hand side of inequality~\eqref{eq:bar_vareepsilon}  can be shown to converge  to zero as $n \rightarrow \infty$ as well, where the proof  follows  the same steps as those of the proof of Lemma~\ref{lemma:rationality_gap}, where   $\hat{\mathbf g}_n$ and $\mathbf g^\ast_n$ are replaced by  $\hat{\mathbf g}_\infty$ and $\mathbf g^\ast_\infty$, respectively.
 
 \section{Proof of Theorem~\ref{thm:mfs-inf_inf}}\label{sec:proof_thm:mfs-inf_inf}
 The proof follows from the notion of $\varepsilon^T$-perfect equilibrium  of a truncated game  by  horizon (stage) $T$, where as $T \rightarrow \infty$,  $\varepsilon^T\rightarrow 0$. In particular, given a discounted cost repeated game, it is shown in~\cite[Theorem 3.3]{fudenberg1983subgame} that the finite-horizon solution converges to the infinite-horizon one, as $T \rightarrow \infty$, i.e.,   $\varepsilon^T$-perfect equilibrium converges to an equilibrium in the infinite horizon. To avoid repetition, we only show herein that the finite-horizon solution (which is proved in Theorem~\ref{thm:aux} to be the sub-game-perfect Nash equilibrium) converges to the infinite-horizon solution, and refer the interested reader to a procedure described in~\cite{fudenberg1983subgame}  for constructing a sequence of  $\varepsilon^T$-perfect strategies that converge to the infinite-horizon limit.

 For any  game with finite  horizon $T$, fix  the strategies of all players but player $i$.  From the proof of Theorem~\ref{thm:aux},  the best response strategy of player $i \in \mathbb{N}_n$ can be identified by the following dynamic program~\eqref{eq:DP-general-form}  for player $i \in \mathbb{N}_n$ at time $t \in \mathbb{N}_T$, i.e.,  
\begin{equation}\label{eq:DP_infinite_horizon_proof}
V^i_t(\tilde{ \mathbf x}_t)=(1-\gamma) \min_{u^i_t}(\Exp{\gamma^{t-1} c^i(\tilde{ \mathbf x}_t,\tilde{ \mathbf u}_t)+ V^i_{t+1}(\tilde{\mathbf x}_{t+1})\mid \tilde{\mathbf x}_t,\tilde{ \mathbf u}_t }).
\end{equation} 
Define the non-standard Riccati equation  for any $\alpha \in \mathcal{A}_n$:
    \begin{equation}\label{eq:riccati-bar-m_test}
\Compress
  \mathbf P^{\alpha}_t=\gamma^{t-1}( \mathbf  Q^{\alpha}+ (\boldsymbol \theta^{\alpha}_t)^\intercal \mathbf R^{\alpha} \boldsymbol \theta^{\alpha}_t)+(\mathbf A+\mathbf B \boldsymbol \theta^{\alpha}_t)^\intercal \mathbf P^{\alpha}_{t+1}(\mathbf A+\mathbf B \boldsymbol \theta^{\alpha}_t),
  \end{equation} 
  where $ \mathbf P^\alpha_{T+1}=\mathbf 0_{2 d_x \times 2 d_x}$,
and  $\boldsymbol \theta^\alpha_t=:\DIAG(\theta_t(\alpha), \bar \theta_t(\alpha))$,  $\theta_t(\alpha)$ and  $ \bar \theta_t(\alpha)$ are described by:
\begin{align}\label{eq:breve-f_test}
&  \theta_t(\alpha):= (\DeltaF)^{-1} \DeltaK, \quad  \bar \theta^\alpha_t:=(\barF)^{-1} \barK, \nonumber \\
&\DeltaF:= (1-\alpha)\Big[\gamma^{t-1}(R+ \frac{\alpha}{1-\alpha}G^u) + B^\intercal {{}\mathbf P_{t+1}^{\alpha}}^{\hspace{-.2cm} 1,1} B  \Big] \nonumber  \\
&\quad + \alpha \Big[\gamma^{t-1}(R + S^u) +(B+\bar B)^\intercal {{}\mathbf P_{t+1}^\alpha}^{\hspace{-.2cm}1,2} B \Big], \nonumber \\
&\barF:= (1-\alpha)\left[\gamma^{t-1}(R+ S^u) + B^\intercal {{}\mathbf P_{t+1}^\alpha}^{\hspace{-.2cm} 2,1} (B+\bar B)  \right]  \nonumber \\
 &+ \alpha \left[\gamma^{t-1}(R + 2S^u + \bar R+ G^u)+ (B+\bar B)^\intercal {{}\mathbf P_{t+1}^\alpha}^{ \hspace{-.2cm}2,2} (B+ \bar B)  \right], \nonumber \\
&\DeltaK:= -(1-\alpha) \Big[B^\intercal {{}\mathbf P_{t+1}^\alpha}^{\hspace{-.2cm }1,1} A  \Big] - \alpha \Big[(B + \bar B)^\intercal {{}\mathbf P_{t+1}^\alpha}^{\hspace{-.2cm} 1,2} A   \Big], \nonumber \\
&\barK:= -(1-\alpha) \left[B^\intercal {{}\mathbf P_{t+1}^\alpha}^{\hspace{-.2cm} 2,1} (A+\bar A)  \right] \nonumber  \\ 
&\hspace{.9cm}- \alpha \left[(B + \bar B)^\intercal {{}\mathbf P_{t+1}^\alpha}^{ \hspace{-.2cm} 2,2} (A+\bar A)   \right].
\end{align}
Furthermore,  define
$\DeltaM:=\gamma^{t-1}(G^x+  (\theta^\alpha_t)^\intercal G^u \theta^\alpha_t )+ (A-B \theta^\alpha_t) P^{\alpha}_{t+1} (A-B\theta^\alpha_t)$. Hence, one arrives at:
 \begin{multline}\label{eq:value-function-form_inf}
V^i_t(\tilde{\mathbf x}_t)=(1-\gamma)\big(  \begin{bmatrix}
\Delta  x^i_t \\
\bar{ x}_t
\end{bmatrix}^\intercal \mathbf P^{\alpha^i_n}_t \begin{bmatrix}
\Delta x^i_t \\
\bar x_t
\end{bmatrix}+ \ell^{i}_t\\
+ \sum_{j \neq i} \alpha^j_n (\Delta x^j_t)^\intercal P^{\alpha^i_n}_t (\Delta x^j_t) - \frac{(\alpha^i_n)^2}{1-\alpha^i_n} (\Delta x^i_t)^\intercal  P^{\alpha^i_n}_t (\Delta x^i_t)\big).
\end{multline}
where  $\ell^{i}_t:=\ell^{i}_{t+1}+  \TR \big(\VAR(\VEC(\Delta w^i_t,\bar w_t))\mathbf P^{\alpha^i_n}_{t+1}\big)   + \sum_{j\neq i}^n \hspace{-.1cm}\alpha^j_n  \TR \big(\VAR(\Delta w^j_t) P^{\alpha^i_n}_{t+1}\big)- \frac{(\alpha^i_n)^2\TR (\VAR(\Delta w^i_t) P^{\alpha^i_n}_{t+1})}{1-\alpha^i_n}$. Denote by $\bar \Sigma^j=\VAR(\VEC(\Delta w^j_t,\bar w_t))$  and $\Sigma^i= \VAR(\Delta  w^j_t)$,  $t \in \mathbb{N}_T,$ the  time-homogeneous auto-covariance matrices associated with any player $j \in \mathbb{N}_n$. 

 Let  $W^i_{t}(\tilde{\mathbf x}):=\gamma^{-t'-1} V^i_{t'+2}(\tilde{\mathbf x})$ for every $i \in \mathbb{N}_n$,  where  $t':=T-t$, $t \in \mathbb{N}_{T+1}$. Let also $\boldsymbol  \Phi^{\alpha}_t:=\gamma^{-t'-1} \mathbf P^{\alpha}_{t'+2}$ and $ \Phi^{\alpha}_t:= \gamma^{-t'-1} P^{\alpha}_{t'+2}$, $\alpha \in \mathcal{A}_n$.  From~\eqref{eq:DP_infinite_horizon_proof} and~\eqref{eq:value-function-form_inf}, one arrives at the following equation:
 \begin{multline}\label{eq:BE_test}
W^i_{t+1}(\tilde{\mathbf x}_{t'+1})=(1-\gamma) \min_{u^i_{t'+1}}(\Exp{c^i(\tilde{ \mathbf x}_{t'+1},\tilde{ \mathbf u}_{t'+1})\\
+\gamma W^i_{t}(\tilde{\mathbf x}_{t'+2})\mid \tilde{\mathbf x}_{t'+1},\tilde{ \mathbf u}_{t'+1} }),
\end{multline} 
 where
\begin{align}\label{eq:value-function-form_inf1}
W^i_{t+1}(\tilde{\mathbf x}_{t'+1})&=(1-\gamma) \gamma^{-t'}\big(  \begin{bmatrix}
\Delta  x^i_{t'+1} \\
\bar{ x}_{t'+1}
\end{bmatrix}^\intercal \boldsymbol  \Phi^{\alpha^i_n}_{t+1} \begin{bmatrix}
\Delta x^i_{t'+1}\\
\bar x_{t'+1}
\end{bmatrix}\\
&+\tilde \ell^{i}_{t+1}+ \sum_{j \neq i} \alpha^j_n (\Delta x^j_{t'+1})^\intercal  \Phi^{\alpha^i_n}_{t+1} (\Delta x^j_{t'+1}) \\
&- \frac{(\alpha^i_n)^2}{1-\alpha^i_n} (\Delta x^i_{t'+1})^\intercal   \Phi^{\alpha^i_n}_{t+1} (\Delta x^i_{t'+1})\big),
\end{align}
and  $\tilde \ell^{i}_{t+1}:=\tilde \ell^{i}_{t}+ \gamma\big( \TR(\bar \Sigma^i\boldsymbol \Phi^{\alpha^i_n}_{t})    + \sum_{j \neq i }^n \hspace{-.1cm}\alpha^j_n  \TR (\Sigma^j \Phi^{\alpha^i_n}_{t+1})- \frac{(\alpha^i_n)^2\TR (\Sigma^i \Phi^{\alpha^i_n}_{t+1})}{1-\alpha^i_n}\big)$. Subsequently, the  control laws in~\eqref{eq:breve-f_test}  can be expressed as follows:
\begin{align}\label{eq:breve-f_test2}
&  \theta_{t'}(\alpha):= (F_{t'})^{-1} K_{t'}, \quad  \bar \theta^\alpha_{t'}:=(\bar F_{t'})^{-1} \bar K_{t'}, \nonumber \\
&F_{t'}(\alpha):= (1-\alpha)\Big[R+ \frac{\alpha}{1-\alpha}G^u+  \gamma B^\intercal {{}\boldsymbol  \Phi_{t+1}^{\alpha}}^{\hspace{-.2cm} 1,1} B  \Big] \nonumber  \\
&\quad + \alpha \Big[R + S^u +  \gamma(B+\bar B)^\intercal {{}\boldsymbol \Phi_{t+1}^\alpha}^{\hspace{-.2cm}1,2} B \Big], \nonumber \\
&\bar F_{t'}(\alpha)= (1-\alpha)\left[R+ S^u +\gamma B^\intercal {{}\boldsymbol \Phi_{t+1}^\alpha}^{\hspace{-.2cm} 2,1} (B+\bar B)  \right]  \nonumber \\
 &+ \alpha \left[R + 2S^u + \bar R+ G^u+ \gamma (B+\bar B)^\intercal {{}\boldsymbol \Phi_{t+1}^\alpha}^{ \hspace{-.2cm}2,2} (B+ \bar B)  \right], \nonumber \\
&K_{t'}(\alpha)= -(1-\alpha) \gamma(B^\intercal {{}\boldsymbol \Phi_{t+1}^\alpha}^{\hspace{-.2cm }1,1} A ) \nonumber \\
&\hspace{1.5cm}- \alpha  \gamma ((B + \bar B)^\intercal {{}\boldsymbol \Phi_{t+1}^\alpha}^{\hspace{-.2cm} 1,2} A   )), \nonumber \\
&\bar K_{t'}(\alpha):= -(1-\alpha)  \gamma (B^\intercal {{}\boldsymbol \Phi_{t+1}^\alpha}^{\hspace{-.2cm} 2,1} (A+\bar A)) \nonumber  \\ 
&\hspace{1.5cm}- \alpha  \gamma ((B + \bar B)^\intercal {{}\boldsymbol \Phi_{t+1}^\alpha}^{ \hspace{-.2cm} 2,2} (A+\bar A) ).
\end{align}
Due to the uniformly bounded assumption imposed on the admissible actions in Section~\ref{sec:problem} and the fact that the discount factor is less than one, it is well-known that the Bellman equation~\eqref{eq:BE_test} is a contractive mapping and admits a unique solution. In particular, with a slight abuse of notation, one has:
 \begin{equation} 
 \Compress
W^i_{\infty}(\tilde{\mathbf x}_{t'})=(1-\gamma) \min_{u^i_{t'}}(\Exp{c^i(\tilde{ \mathbf x}_{t'},\tilde{ \mathbf u}_{t'})+\gamma W^i_{\infty}(\tilde{\mathbf x}_{t'+1})\mid \tilde{\mathbf x}_{t'},\tilde{ \mathbf u}_{t'} }),
\end{equation} 
where under Assumption~\ref{ass:invertible-algebraic} and  for any $\alpha \in \mathcal{A}_n$, 
\begin{equation}\label{eq:change_proof_inf1}
\Compress
\lim_{T \rightarrow \infty}  \boldsymbol  \Phi^{\alpha}_{t+1}=\lim_{T\rightarrow \infty} \gamma^{-T+t} \mathbf P_{T-t+1}=\mathbf P^{\alpha}, 
\end{equation}
\begin{equation}\label{eq:change_proof_inf2}
\Compress
 \lim_{T \rightarrow \infty}  \Phi^{\alpha}_{t+1}=\lim_{T \rightarrow \infty} \gamma^{-T+t} P^{\alpha}_{T-t+1}=P^{\alpha}.
\end{equation}
Therefore, from~\eqref{eq:breve-f_test2},~\eqref{eq:change_proof_inf1}  and~\eqref{eq:change_proof_inf2},  one can conclude that the terms on  the right-hand side of~\eqref{eq:breve-f_test2} are  time-homogeneous, as  $T \rightarrow \infty$, i.e.,  stationary strategies~\eqref{eq:thm-optimal-mfs-u_1_inf},~\eqref{eq:thm-optimal-mfs-u_2_inf} and~\eqref{eq:thm-optimal-mfs-u_3_inf} are the limit of their finite-horizon counterparts.  
\section{Proof of Proposition~\ref{proposition:hurwitz}}\label{sec:proof_proposition:hurwitz}
From  equations~\eqref{eq:dynamics_joint},~\eqref{eq:thm-optimal-mfs-u_1_inf} and~\eqref{eq:deterministic_process},   $\bar x_{t+1}$ and $z^n_{t+1}$  have  the same  dynamics (except the drift terms) under strategies  $\{ \mathbf g^{\ast,i}_n,\mathbf g^{\ast,-i}_n\}$ and $\{\hat{ \mathbf g}^{i}_n,\hat{\mathbf g}^{-i}_n\}$, respectively,  implying that  trajectory $z^n_{1:\infty}$ is bounded because  trajectory $\bar x_{1:\infty}$ is bounded    under Assumption~\ref{ass:invertible-algebraic} according to Theorem~\ref{thm:mfs-inf_inf}. On the other hand, the relative distances in Lemma~\ref{lemma:breve_equality} are also bounded because matrix $\tilde A^n$ is Hurwitz. Therefore, trajectory $\hat{\bar x}_{1:\infty}$  should be  bounded, which means $\hat{\bar u}_{1:\infty}$  is  also bounded, according to~\eqref{eq:action-ns}. Consequently, the   infinite-horizon discounted cost function $J^{i,\gamma}_n(\hat{\mathbf g}^{i}_n,\hat{\mathbf g}^{-i}_n)$ is  bounded. A similar argument holds for $\{ \mathbf g^{\ast,i}_\infty,\mathbf g^{\ast,-i}_\infty\}$ and $\{\hat{ \mathbf g}^{i}_\infty,\hat{\mathbf g}^{-i}_\infty\}$. 
\end{document}